\documentclass[11pt,reqno]{amsart}

\usepackage[T1]{fontenc}
\usepackage[a4paper,margin=1in]{geometry}
\usepackage{amsmath,amssymb,amsthm,mathtools,mathrsfs}
\usepackage{enumitem}
\usepackage{microtype}
\usepackage[hidelinks]{hyperref}
\allowdisplaybreaks
\newtheorem{theorem}{Theorem}[section]
\newtheorem{lemma}[theorem]{Lemma}
\newtheorem{proposition}[theorem]{Proposition}
\newtheorem{corollary}[theorem]{Corollary}
\theoremstyle{definition}

\theoremstyle{remark}
\newtheorem{remark}[theorem]{Remark}

\numberwithin{equation}{section}

\newcommand{\C}{\mathbb C}
\newcommand{\R}{\mathbb R}
\newcommand{\CHyp}{\mathbb H_{\mathbb C}}
\newcommand{\Heis}{\mathcal H}
\newcommand{\Ucal}{\mathcal U}
\newcommand{\Gell}{\mathcal G}
\newcommand{\dd}{\,d}

\title[Critical Geller Equations on Complex Hyperbolic Space]{Critical Geller Equations on Complex Hyperbolic Space:\\
Sharp Stability, Ground-State Symmetry, and Global Compactness}
\author{Jungang Li}
\address{Department of Mathematics, University of Science and Technology of
China, Hefei, Anhui, China}
\email{jungangli@ustc.edu.cn}
\date{}
\hypersetup{
 pdftitle={Critical Geller Equations on Complex Hyperbolic Space:
 Sharp Stability, Ground-State Symmetry, and Global Compactness},
 pdfauthor={Jungang Li},
 pdfsubject={Critical Geller equations on complex hyperbolic space},
 pdfkeywords={Geller operator, complex hyperbolic space, Bianchi-Egnell
 stability, sharp Sobolev inequality, profile decomposition,
 Brezis-Nirenberg problem, optimal stability quotient, ground-state
 nondegeneracy, second-order asymptotics}
}

\begin{document}

\begin{abstract}
For integers \(n,\ell\geq1\), set
\(Q_\ell=2(n+\ell)+2\) and
\(q_\ell=2Q_\ell/(Q_\ell-2)\). On the Siegel domain
\(\Ucal=\Heis^n\times(0,\infty)\), we study
\[
 -\Delta_{\Heis^n}v
 -4\rho\bigl(v_{\rho\rho}+T^2v\bigr)-4\ell v_\rho
 =|v|^{q_\ell-2}v.
\]
For its Dirichlet form \(E_\ell\), we determine the sharp Sobolev constant
and all extremals, prove
\[
 S_\ell\|v\|_{q_\ell}^2\leq E_\ell(v),\qquad
 E_\ell(v)-S_\ell\|v\|_{q_\ell}^2
 \geq\kappa_{n,\ell}
 \operatorname{dist}_{\dot S^1_\ell}
 \bigl(v,\mathfrak M_\ell^{\R}\bigr)^2,
\]
where \(\mathfrak M_\ell^{\R}\) is the extremal cone. We classify
nonnegative finite-energy solutions and establish the linearized kernel,
profile decomposition,
and attainment of the optimal stability quotient. Cayley conjugation yields
ground-state symmetry and nondegeneracy, global Palais--Smale compactness,
and perturbative existence for critical equations on
\(\CHyp^{n+1}\).

The mechanism is the radial lift
\(v^\uparrow(z,w,t)=v(z,t,|w|^2)\), which reverses the
interior-to-boundary construction by realizing the Siegel domain as a
symmetry-reduced slice of a larger Heisenberg group. The completed-space
reduction resolves the degenerate axis, hidden auxiliary concentration,
and the mismatch of extremal cones. Together with the Cayley transform,
this supplies the missing nonlinear layer---classification, stability,
bubbling, and variational compactness---on complex hyperbolic space and
provides a blueprint for other rank-one symmetric spaces.
\end{abstract}

\maketitle

\section{Introduction}

The analysis of a critical Sobolev equation on a noncompact geometry
requires more than a coercive embedding.  A nonlinear theory must also
identify the best constant and its equality manifold, determine the
linearized kernel, quantify the distance to the extremals, describe every
loss of compactness, and place perturbed critical points below the
corresponding concentration thresholds.  On Euclidean space these
components form the classical Sobolev--Brezis--Nirenberg theory.  On real
hyperbolic space they interact with the spectral gap and with a different
noncompactness geometry: profiles may occur at a small Euclidean scale or
escape under hyperbolic isometries.

The broader analysis of noncompact homogeneous and symmetric spaces
already supplies qualitative variational and concentration--compactness
frameworks for semilinear elliptic problems
\cite{BiroliSchindlerTintarev2003Symmetries}.  Among rank-one spaces, real
hyperbolic space provides the closest comparison with the present work.
Mancini--Sandeep \cite{ManciniSandeep2008Semilinear} established the
second-order existence and classification theory, and Bhakta--Sandeep
\cite{BhaktaSandeep2012Poincare} described Palais--Smale sequences and
constructed sign-changing solutions.  At higher order, Lu--Yang
\cite{LuYang2019PaneitzHyperbolic} proved Hardy--Sobolev--Maz'ya
inequalities for Paneitz-type operators, while Li--Lu--Yang
\cite{LiLuYang2022HyperbolicBN} developed a substantial
Brezis--Nirenberg theory, including existence, nonexistence, and symmetry.
More recently, Bhakta--Ganguly--Karmakar--Mazumdar
\cite{BhaktaGangulyKarmakarMazumdar2025Stability,
BhaktaGangulyKarmakarMazumdar2025Struwe} established quantitative
Poincar\'e--Sobolev stability, from the optimizer manifold to multibubble
Struwe decompositions.  On complex hyperbolic space, Lu--Yang
\cite{LuYang2022SiegelHardySobolevMazya} made the foundational
harmonic-analytic advances relevant here.  They obtained factorization
formulas for the operators associated with the Geller family, proved
Poincar\'e--Sobolev and Hardy--Sobolev--Maz'ya inequalities, and established
sharp Adams and Hardy--Adams inequalities.  In particular, their work
identifies the correct operators and provides the coercive estimates needed
to formulate the present critical problems.  For the power-type nonlinear
Sobolev quotients studied below, the Poincar\'e--Sobolev and
Hardy--Sobolev--Maz'ya estimates serve as this indispensable starting point;
the optimal quotient constants and equality cases are not determined in
that work.  The questions that remain are therefore nonlinear and
geometric rather than spectral: one must recover the optimizer geometry,
classify finite-energy solutions, quantify stability, and resolve every
concentration channel of critical sequences.

A second, closely related line of work uses the Siegel domain in the usual
interior-to-boundary direction.  The foundational scattering construction
of Graham--Zworski \cite{GrahamZworski2003Scattering} recovers conformally
covariant operators at infinity from interior eigenvalue problems on
conformally compact Einstein manifolds.  In the CR setting,
Frank--Gonz\'alez--Monticelli--Tan
\cite{FrankGonzalezMonticelliTan2015CRExtension} realized the conformal
fractional powers of the Heisenberg sub-Laplacian as Dirichlet-to-Neumann,
or equivalently scattering, operators for a linear extension problem on the
Siegel upper half-space, and derived the corresponding sharp CR Sobolev
trace inequality.  Flynn--Lu--Yang \cite{FlynnLuYang2025CRTrace} developed
this picture at higher orders by constructing conformally covariant boundary
operators, proving the higher-order extension theorems, and establishing
sharp higher-order CR Sobolev trace inequalities.  In this picture the
interior equation is the linear mechanism that encodes a critical operator
and inequality on the boundary Heisenberg group.

These developments give complex hyperbolic analysis a strong linear,
harmonic-analytic, and boundary theory.  Prior to the present work, what
remained unavailable for interior equations on complex hyperbolic space was
a unified sharp nonlinear theory: one that simultaneously determines the
critical constant and equality manifold, classifies finite-energy solutions,
proves quantitative stability, resolves the full profile decomposition, and
controls perturbed critical points by exact concentration thresholds.  Our
purpose is to construct this theory for the interior Geller equation.  For a
real parameter \(\ell\), the Geller operator
\begin{equation}\label{eq:intro-geller}
 \Gell_\ell v
 =-\Delta_{\Heis^n}v-4\rho\bigl(v_{\rho\rho}+T^2v\bigr)-4\ell v_\rho
\end{equation}
on the Siegel domain \(\Ucal=\Heis^n\times(0,\infty)\) is the conformally
natural elliptic family introduced in Geller's harmonic analysis of the
CR sphere \cite{Geller1980Kohn}.

Our organizing idea uses the same bulk--boundary geometry in the opposite
direction.  Rather than eliminating the interior variable in order to
extract an operator on \(\Heis^n\), we retain the Siegel domain as the space
of the nonlinear equation, adjoin an auxiliary block \(w\in\C^\ell\), and
write
\[
 v^\uparrow(z,w,t)=v(z,t,|w|^2),
 \qquad w\in\C^\ell.
\]
When \(\ell\) is an integer, the map
\((z,w,t)\mapsto(z,t,|w|^2)\) induces an identification of \(\Ucal\) with
the principal stratum of the orbit space
\(\Heis^{n+\ell}/U(\ell)\), and \(\Gell_\ell\) becomes the radial part of
\(-\Delta_{\Heis^{n+\ell}}\) on the fixed sector.  Thus the critical
Geller equation is realized as the fixed sector of the critical
Folland--Stein equation on the larger Heisenberg group
\(\Heis^{n+\ell}\).  This dimension-raising radial lift is exact in the
completed energy space.  From this viewpoint, the Siegel domain is not
merely an extension space over its Heisenberg boundary, but a symmetry-reduced
radial sector of a larger Heisenberg geometry.  The resulting lift--descent
mechanism is what makes the missing nonlinear layer accessible.  It also
suggests a route for critical nonlinear analysis across rank-one symmetric
spaces: realize an interior problem through an invariant sector of the
nilpotent boundary model, and then descend sharp and compactness information
through the symmetry.

There is a second, intrinsic realization.  For every integer
\(\ell\geq1\), the smooth equation is exactly conjugated by the Cayley
multiplier
\(\mathcal C_\ell U=\rho^{-(n+\ell)/2}U\), to a Brezis--Nirenberg equation
for a Schr\"odinger operator on complex hyperbolic space
\(\CHyp^{n+1}\).  When \(\ell>1\), this second correspondence extends to
the completed form domains, and the two realizations fit into the diagram
\[
\renewcommand{\arraystretch}{1.4}
\begin{array}{ccc}
 \dot S^1_\ell(\Ucal)
 &\ \xrightarrow{\ \ \uparrow\ \ }\ &
 \dot H^1(\Heis^{n+\ell})^{U(\ell)}\\
 \Big\downarrow{\scriptstyle\ \mathcal C_\ell^{-1}} & &\\
 H^1(\CHyp^{n+1},g) & &
\end{array}
\]
The horizontal map is a similarity onto the invariant sector, with squared
energy factor \(c_\ell=\pi^\ell/\Gamma(\ell)\); for \(\ell>1\), the
vertical map is an isomorphism of Friedrichs form domains.  The two
realizations contribute different parts of the nonlinear theory.  The
Heisenberg realization carries sharp constants
\cite{FrankLieb2012SharpHeisenberg}, extremal classification
\cite{JerisonLee1988Extremals}, profile decomposition
\cite{Benameur2008ProfileHeisenberg}, and quantitative stability
\cite{Loiudice2005ImprovedSobolev}.  The complex hyperbolic realization
carries an intrinsic isometry group, a spectral gap, and the negatively
curved geometry underlying barycenter constructions.  In Siegel
coordinates it is the harmonic \(AN\) group
\(\Heis^n\rtimes\R_+\), whose spherical analysis---radial Laplacian, heat
kernel, and spectral theory---was developed by Anker--Damek--Yacoub
\cite{AnkerDamekYacoub1996HarmonicAN} and, for general noncompact
symmetric spaces, by Anker \cite{Anker1990Multipliers} and Anker--Ji
\cite{AnkerJi1999HeatGreen}.  The critical exponent is
governed by the Heisenberg homogeneous dimension \(Q_\ell=2(n+\ell)+2\),
not by the Riemannian dimension of \(\CHyp^{n+1}\); the integer family is
therefore transverse to the Riemannian Yamabe problem, which it meets only
at the formal value \(\ell=0\).  This places the Geller family within a
sharp nonlinear theory on a rank-one symmetric space, organized jointly
by its sub-Riemannian and negatively curved realizations.

The dimension-raising lift is not a formal transfer principle.  The
homogeneous completion of the reduced energy must cross the degenerate
axis \(\rho=0\), so a reduced weak solution cannot initially be tested
against functions that are smooth through the axis.  Full-space
translations can move a profile into the invisible auxiliary
\(\C^\ell\)-directions, and the full optimizer cone is strictly larger
than its fixed part.  Consequently, neither profile decomposition nor
quantitative stability descends by restricting a full-space theorem.  On
the intrinsic side, the Cayley multiplier converts a constant spectral
term into the singular Hardy weight \(\rho^{-1}\).  The sharp theory,
optimizer classification, compactness, and perturbed equations therefore
require additional fixed-sector arguments.

The main innovation is to use the \(U(\ell)\)-symmetry to resolve the
obstructions created by the reduction.  The auxiliary translation tangent
at a bubble is annihilated by the projection onto the fixed sector; this
orthogonality converts a full-space Bianchi--Egnell estimate into a
uniform fixed-sector estimate
(Proposition~\ref{prop:distance-comparison}).  Rotating an unbounded
auxiliary center would produce arbitrarily many asymptotically orthogonal
weak limits of equal norm, which eliminates the invisible concentration
parameter from the profile decomposition
(Theorem~\ref{thm:intro-profile}).  Finally, the barycenter of a
\(U(\ell)\)-invariant boundary measure is itself fixed and hence lies in a
totally geodesic copy of \(\CHyp^{n+1}\); this supplies continuous
center--scale coordinates for the linking argument
(Lemma~\ref{lem:busemann-coordinates}).  The symmetry is therefore not
merely imposed on the ambient theory: it provides the mechanism by which
the nonlinear theory closes.

\subsection{Setup and notation}

Fix \(n\geq1\) and an integer \(\ell\geq1\), and set
\[
 \Ucal=\Heis^n\times(0,\infty),
 \qquad
 Q_\ell=2(n+\ell)+2,
 \qquad
 q_\ell=\frac{2Q_\ell}{Q_\ell-2}=2+\frac{2}{n+\ell}.
\]
We write \(\nabla_H\) and \(\Delta_{\Heis^n}\) for the horizontal gradient
and sub-Laplacian on \(\Heis^n\), and \(T=\partial_t\) for the central
field; the normalizations are recorded at the start of
Section~\ref{sec:radial-lift}.  The Hermitian product
\(\langle z,z'\rangle=\sum_jz_j\overline{z'_j}\) is linear in its first
variable.  The operator is \(\Gell_\ell\) as in \eqref{eq:intro-geller};
the measure and energy adapted to it are
\[
 \dd\mu_\ell=\rho^{\ell-1}\dd z\dd t\dd\rho,
 \qquad
E_\ell(v)
=\int_{\Ucal}\bigl(
|\nabla_Hv|^2+4\rho(|v_\rho|^2+|v_t|^2)
\bigr)\dd\mu_\ell .
\]
We denote the associated polarized bilinear form by
\(E_\ell(v,w)\).
Unless another measure is displayed we write
\(\|v\|_r=(\int_{\Ucal}|v|^r\dd\mu_\ell)^{1/r}\), and \(\dot
S^1_\ell(\Ucal)\) denotes the completion of \(C_c^\infty(\Ucal)\) in the
norm \(E_\ell^{1/2}\).  No boundary trace at \(\rho=0\) is part of this
definition.

The completed linear energy-space identities, the sharp inequality, and the
profile decomposition are valid over either \(\R\) or \(\C\).  All
solutions, Palais--Smale sequences, and variational manifolds in the Hardy
and compact-potential problems are taken over \(\R\).

For the Hardy problem we assume \(\ell>1\).  Since then
\(n+\ell\geq3\),
\begin{equation}\label{eq:intro-subcubic-range}
 2<q_\ell\leq\frac83<3.
\end{equation}
Thus the critical Nemytskii map is naturally \(C^{1,q_\ell-2}\), rather
than twice differentiable, on the energy space.  This is the source of the
fractional remainder orders in the small-coupling expansions below.

The restriction to integer \(\ell\) is structural: it is exactly the
condition under which \(\Gell_\ell v\) is the restriction of
\(-\Delta_{\Heis^{n+\ell}}v^\uparrow\) to the \(U(\ell)\)-invariant sector.
This observation transfers sharp full-space information to the Geller
problem, but, as explained above, the transfer is not formal at the level
of homogeneous completions.

\subsection{Main results}

The results below form three connected layers of one theory.  The sharp
fixed-sector layer consists of the sharp inequality, equality and
solution classification, profile decomposition, quantitative stability,
and the spectral geometry of the optimizer manifold.  The intrinsic
layer uses the Cayley correspondence to treat attractive Hardy
perturbations, including global compactness, uniqueness, symmetry,
nondegeneracy, and small-coupling asymptotics.  The compact-potential
layer combines the sharp thresholds with the negatively curved geometry
to produce solutions for repulsive perturbations.  We state the results
in this order to make their logical dependence explicit.

The sharp Sobolev constant is defined by its variational role,
\begin{equation}\label{eq:intro-Sl}
 S_\ell
 =\inf_{0\ne v\in\dot S^1_\ell(\Ucal)}
 \frac{E_\ell(v)}{\|v\|_{q_\ell}^2}.
\end{equation}
The sharp fixed-sector theory begins by determining this constant, the
full equality family, and all nonnegative finite-energy solutions of the
critical equation.  The closed form of the constant, obtained from the
polar Jacobian of the integer lift and the sharp full-Heisenberg constant
of Frank--Lieb, is recorded in Section~\ref{sec:sharp-profile}.

\begin{theorem}[Sharp inequality and classification]
\label{thm:intro-sharp}
For every \(v\in\dot S^1_\ell(\Ucal)\),
\[
 S_\ell\|v\|_{q_\ell}^2\leq E_\ell(v).
\]
The constant is optimal.  The nonzero real equality cases are precisely
\[
 A\left[
  (\rho+\beta+|z-z_0|^2)^2
  +(t-t_0+2\operatorname{Im}\langle z,z_0\rangle)^2
 \right]^{-(n+\ell)/2},
\]
where \(A\ne0\), \(\beta>0\), \(z_0\in\C^n\), and \(t_0\in\R\).  Moreover,
every nonzero nonnegative weak solution in \(\dot S^1_\ell(\Ucal)\) of
\[
 \Gell_\ell v=v^{q_\ell-1}
\]
is a coefficient-one member of this family.
\end{theorem}

We write \(B_0\) for the positive centered extremal normalized by
\(\|B_0\|_{q_\ell}=1\), so that \(E_\ell(B_0)=S_\ell\).  For \(h>0\) and
\(\eta=(z_*,t_*)\in\Heis^n\), the translation--dilation action is
\begin{equation}\label{eq:intro-frame}
 (\mathcal T_{h,\eta}\phi)(z,t,\rho)
 =h^{-(n+\ell)}\phi\!\left(
 \frac{z-z_*}{h},
 \frac{t-t_*+2\operatorname{Im}\langle z,z_*\rangle}{h^2},
 \frac{\rho}{h^2}\right).
\end{equation}
Both \(E_\ell\) and \(\|{\cdot}\|_{q_\ell}\) are invariant under
\(\mathcal T_{h,\eta}\).  We call \(g=(h,\eta)\) a \emph{frame} and
\(\mathcal T_g\) its associated \emph{frame action}.  The real amplitude
cone and the signed unit-norm
orbit generated by \(B_0\) are
\[
 \mathfrak M_\ell^{\R}
 =\{A\,\mathcal T_{h,\eta}B_0:A\in\R,\;h>0,\;\eta\in\Heis^n\},
\]
\[
 \mathcal M_{1,\ell}
 =\{\sigma\,\mathcal T_{h,\eta}B_0:
 \sigma\in\{-1,1\},\;h>0,\;\eta\in\Heis^n\}.
\]

The same fixed-sector structure identifies the entire defect of
compactness.  The essential point is the parameter set: only Heisenberg
translations and homogeneous dilations occur, with no
\(\rho\)-translation and no auxiliary \(\C^\ell\)-center.

\begin{theorem}[Profile decomposition in the invariant sector]
\label{thm:intro-profile}
Every bounded sequence \((v_k)\) in \(\dot S^1_\ell(\Ucal)\) admits, after
passage to a subsequence, profiles \(\phi^j\in\dot S^1_\ell(\Ucal)\),
frames \(g_k^j=(h_k^j,\eta_k^j)\), and remainders \(r_k^J\) such that, for
every fixed \(J\),
\begin{align}
 v_k&=\sum_{j=1}^J\mathcal T_{g_k^j}\phi^j+r_k^J,
 \label{eq:intro-profile-expansion}\\
 E_\ell(v_k)
 &=\sum_{j=1}^JE_\ell(\phi^j)+E_\ell(r_k^J)+o_k(1),
 \label{eq:intro-profile-energy}\\
 \|v_k\|_{q_\ell}^{q_\ell}
 &=\sum_{j=1}^J\|\phi^j\|_{q_\ell}^{q_\ell}
  +\|r_k^J\|_{q_\ell}^{q_\ell}+o_k(1),
 \label{eq:intro-profile-mass}
\end{align}
and
\begin{equation}\label{eq:intro-profile-remainder}
 \lim_{J\to\infty}\limsup_{k\to\infty}\|r_k^J\|_{q_\ell}=0.
\end{equation}
For \(i\ne j\), either
\(\bigl|\log(h_k^i/h_k^j)\bigr|\to\infty\), or comparable scales may be
represented by the same sequence and
\[
 \left|\delta_{1/h_k^i}
 \bigl((\eta_k^i)^{-1}\eta_k^j\bigr)\right|_{\Heis^n}
 \longrightarrow\infty.
\]
\end{theorem}

Quantitative stability is the next component of this layer.  Let
\(\widetilde B\) be a full-space bubble
on \(\Heis^{n+\ell}\), and let \(\widetilde{\mathfrak M}_\ell^{\mathrm{full}}\)
and \(\widetilde{\mathfrak M}_\ell^{\mathrm{fix}}\) denote, respectively,
its full real translation--dilation cone and the subcone whose auxiliary
\(\C^\ell\)-center is zero.  We prove that every \(U(\ell)\)-fixed real
\(F\in\dot H^1(\Heis^{n+\ell})\) satisfies
\begin{equation}\label{eq:intro-cone-comparison}
 \operatorname{dist}_{\dot H^1}
 \bigl(F,\widetilde{\mathfrak M}_\ell^{\mathrm{fix}}\bigr)
 \leq C_{n,\ell}
 \operatorname{dist}_{\dot H^1}
 \bigl(F,\widetilde{\mathfrak M}_\ell^{\mathrm{full}}\bigr).
\end{equation}
The polar lift then converts \eqref{eq:intro-cone-comparison} and
Loiudice's full-space stability theorem
\cite{Loiudice2005ImprovedSobolev} into the following.

\begin{theorem}[Quantitative fixed-sector stability]
\label{thm:intro-stability}
There exists \(\kappa_{n,\ell}>0\) such that for every real
\(v\in\dot S^1_\ell(\Ucal;\R)\),
\[
 E_\ell(v)-S_\ell\|v\|_{q_\ell}^2
 \geq\kappa_{n,\ell}
 \operatorname{dist}_{\dot S^1_\ell}
 \!\bigl(v,\mathfrak M_\ell^{\R}\bigr)^2.
\]
\end{theorem}

The new ingredient is Proposition~\ref{prop:distance-comparison}, which
supplies \eqref{eq:intro-cone-comparison}: the distance from a
\(U(\ell)\)-fixed function to the reduced optimizer cone is bounded by a
constant multiple of its distance to the strictly larger full-space cone.
Three regimes of the auxiliary center are treated separately---an
infinitesimal regime controlled by the tangent estimate of
Lemma~\ref{lem:auxiliary-tangent}, compact annuli on which the ratio is
continuous with nonvanishing denominator, and a large-center regime in
which the Haar average of the bubble correlation vanishes.  The estimate is
uniform in the amplitude and in the frame, and this uniformity is what
survives the polar descent.

The completed lift also permits a direct spectral analysis of the
linearization.  Let
\[
 \mathcal Q_{B_0}(\phi,\psi)
 =E_\ell(\phi,\psi)
 -(q_\ell-1)S_\ell
 \int_{\Ucal}B_0^{q_\ell-2}\phi\psi\dd\mu_\ell,
\]
let \(\mathcal Z_\ell=T_{B_0}\mathcal M_{1,\ell}\), and put
\[
 Y_\ell
 =\left\{\phi:
 \int_{\Ucal}B_0^{q_\ell-1}\phi\dd\mu_\ell=0,\quad
 E_\ell(\phi,Z)=0\ \text{for every }Z\in\mathcal Z_\ell
 \right\}.
\]
We write \(\widehat X_j,\widehat Y_j\), \(\partial_t\), and
\(\Lambda_\ell\) for the infinitesimal translations, central
translation, and dilation, respectively.

\begin{theorem}[Fixed-sector nondegeneracy and local spectral gap]
\label{thm:intro-nondegeneracy}
The linearized operator at \(B_0\) satisfies
\[
 \ker L_{B_0}
 =\mathcal Z_\ell
 =\operatorname{span}_{\R}
 \left\{
 \widehat X_jB_0,\widehat Y_jB_0\ (1\leq j\leq n),
 \partial_tB_0,\Lambda_\ell B_0
 \right\}.
\]
Moreover,
\[
 \mathcal Q_{B_0}(\phi,\phi)
 \geq\frac{2}{n+\ell+4}E_\ell(\phi)
 \qquad(\phi\in Y_\ell),
\]
and the constant is optimal.  Equivalently, the exact local
Hessian-to-energy coefficient of the Sobolev deficit along the real
optimizer cone is \(2/(n+\ell+4)\).
\end{theorem}

The proof takes the \(U(\ell)\)-fixed part of the full Heisenberg kernel.
The auxiliary translation modes form the standard representation of
\(U(\ell)\) and have no fixed vector, whereas all visible translations
and the dilation survive.  At the next spectral level, however,
\(U(\ell)\)-fixed bispherical harmonics remain; consequently the local
fixed-sector spectral gap is exactly the full-space gap, not a larger one.
The value in Theorem~\ref{thm:intro-nondegeneracy} is local and should not
be confused with the best constant in the global stability inequality.

In fact, the optimal fixed-sector quotient has its own compactness
theorem.  Define
\[
 \kappa_{n,\ell}^{\mathrm{opt}}
 =\inf_{v\in\dot S^1_\ell(\Ucal;\mathbb R)
 \setminus\mathfrak M_\ell^{\mathbb R}}
 \frac{
 E_\ell(v)-S_\ell\|v\|_{q_\ell}^2
 }{
 \operatorname{dist}_{\dot S^1_\ell}
 (v,\mathfrak M_\ell^{\mathbb R})^2
 }.
\]

\begin{theorem}[Optimal fixed-sector stability quotient]
\label{thm:intro-optimal-stability}
For all integers \(n,\ell\geq1\),
\[
 0<\kappa_{n,\ell}^{\mathrm{opt}}
 <\min\left\{
 \frac2{n+\ell+4},\,
 2-2^{(n+\ell)/(n+\ell+1)}
 \right\}.
\]
The infimum is attained by a real function outside the optimizer cone.
Every minimizing sequence normalized in \(L^{q_\ell}\) is precompact in
\(\dot S^1_\ell(\Ucal)\), after passage to a subsequence and application
of reduced translations and dilations.
\end{theorem}

The strict local inequality comes from a fourth-order normal perturbation
on the CR sphere based on
\(\operatorname{Re}\zeta_{n+\ell+1}^2\).  Its cubic moment vanishes, but
an explicit fixed second-order correction makes the quartic coefficient
strictly negative.  The second strict inequality is obtained from the raw
sum of two centered bubbles at diverging scales.  Together they exclude
the two possible losses of compactness for an optimal sequence: approach
to the optimizer cone and two-bubble splitting.

The intrinsic nonlinear layer begins with the completed Cayley
correspondence.  Equip \((\CHyp^{n+1},g)\) with the metric of holomorphic
sectional curvature \(-4\).  For \(\ell>1\), write
\[
 \mathsf H_\ell(v,v)=\int_{\Ucal}\rho^{-1}v^2\dd\mu_\ell,
 \qquad
 A^H_{\lambda,\ell}(v)=E_\ell(v)-\lambda\,\mathsf H_\ell(v,v),
\]
\[
 S^H_{\lambda,\ell}
 =\inf_{\|v\|_{q_\ell}=1}A^H_{\lambda,\ell}(v),
 \qquad
 J^H_{\lambda,\ell}(v)
 =\tfrac12A^H_{\lambda,\ell}(v)-\tfrac1{q_\ell}\|v\|_{q_\ell}^{q_\ell}.
\]
Under the Cayley correspondence, a constant complex hyperbolic spectral term
becomes
the Hardy weight \(\rho^{-1}\), and the sharp constant in the associated
Hardy inequality is \((\ell-1)^2\); this is the coercive range below.

\begin{theorem}[Ground states and global compactness]
\label{thm:intro-hardy}
Let \(\ell>1\) be an integer and \(0<\lambda<(\ell-1)^2\).  There exists
\[
 U_\lambda\in H^1(\CHyp^{n+1},g)\cap C^\infty(\CHyp^{n+1}),
 \qquad U_\lambda>0,
\]
such that
\[
 \bigl[-\Delta_g-(n+1)^2+(\ell-1)^2-\lambda\bigr]U_\lambda
 =U_\lambda^{q_\ell-1},
\]
and this solution has least action among all nonzero real critical points.
Every finite-level Palais--Smale sequence for \(J^H_{\lambda,\ell}\) is,
after passage to a subsequence, a finite sum of pairwise orthogonal
profiles solving
\[
 \Gell_\ell\phi-\lambda\rho^{-1}\phi=|\phi|^{q_\ell-2}\phi,
\]
plus a remainder converging strongly in \(\dot S^1_\ell(\Ucal)\); the
completed Cayley map transports these profiles to solutions of the
displayed complex hyperbolic equation.  With
\begin{equation}\label{eq:intro-quantum}
 d_{\lambda,\ell}
 =\frac1{Q_\ell}\bigl(S^H_{\lambda,\ell}\bigr)^{Q_\ell/2},
\end{equation}
every nonzero profile has action at least \(d_{\lambda,\ell}\), every nodal
profile has action at least \(2d_{\lambda,\ell}\), and the action splits
exactly over the profiles.  The corresponding intrinsic action quantum is
\(d_{\lambda,\ell}/4\).
\end{theorem}

The preceding theorem produces minimizers throughout the coercive range.
For small coupling, nondegeneracy determines them much more rigidly.

\begin{theorem}[Uniqueness, geodesic radial symmetry, and nondegeneracy]
\label{thm:intro-symmetry}
Let \(\ell>1\) be an integer.  For all sufficiently small
\(\lambda>0\), the normalized minimizers of
\(S^H_{\lambda,\ell}\) form one orbit under the transported identity
component of the holomorphic isometry group of \(\CHyp^{n+1}\), together
with sign.  Equivalently, the corresponding positive intrinsic minimizer
is unique up to holomorphic isometry and is geodesically radial about some
point of \(\CHyp^{n+1}\).

Let \(v_\lambda\) be the centered positive normalized minimizer, put
\[
 \bar v_\lambda
 =(S^H_{\lambda,\ell})^{1/(q_\ell-2)}v_\lambda,
\]
and let \(A_\lambda\) be the Riesz operator, with respect to \(E_\ell\),
of the linearized form
\[
 \begin{aligned}
 \mathcal Q_\lambda(\phi,\psi)
 ={}&E_\ell(\phi,\psi)-\lambda\mathsf H_\ell(\phi,\psi)\\
 &-(q_\ell-1)\int_{\Ucal}
 \bar v_\lambda^{q_\ell-2}\phi\psi\dd\mu_\ell.
 \end{aligned}
\]
If \(G=\operatorname{Isom}_0(\CHyp^{n+1},g)\), acting through the Cayley
correspondence, then
\[
 \ker A_\lambda
 =T_{\bar v_\lambda}(G\cdot\bar v_\lambda),
 \qquad
 \dim\ker A_\lambda=2n+2.
\]
Moreover, for some \(c_{n,\ell}>0\),
\[
 \|A_\lambda\phi\|_{\dot S^1_\ell}
 \geq c_{n,\ell}\|\phi\|_{\dot S^1_\ell}
 \qquad
 \bigl(\phi\perp_E\ker A_\lambda\bigr).
\]
\end{theorem}

In the real hyperbolic setting, Mancini--Sandeep
\cite{ManciniSandeep2008Semilinear} obtain radiality by a moving-plane
argument for the differential equation, whereas Li--Lu--Yang
\cite{LiLuYang2022HyperbolicBN} use moving planes for the associated
integral equation in the higher-order problem.  Our proof uses no
moving-plane argument.  Fixed-sector nondegeneracy gives an equivariant
normal slice, and uniqueness within that slice forces invariance under the
stabilizer of its center.

We call \(d_{\lambda,\ell}\) the \emph{action quantum}: it is the least
action of a nonzero critical point, and it is the unit in which the
Palais--Smale levels are measured.  Combining
Theorems~\ref{thm:intro-stability} and~\ref{thm:intro-hardy} yields a
quantitative rigidity statement for the Hardy minimizers as the coupling
vanishes.  Since the frames act transitively on
\(\mathcal M_{1,\ell}\), these are orbit-distance estimates; they select
neither a center nor a scale.

\begin{theorem}[Quantitative rigidity as \(\lambda\downarrow0\)]
\label{thm:intro-rigidity}
Let \(\ell>1\) be an integer.  For every
\(0<\lambda\leq(\ell-1)^2/2\) and every real normalized minimizer
\(v_\lambda\) of \(S^H_{\lambda,\ell}\),
\[
 \operatorname{dist}_{\dot S^1_\ell}
 \!\bigl(v_\lambda,\mathcal M_{1,\ell}\bigr)\leq C_{n,\ell}\lambda,
 \qquad
 0\leq E_\ell(v_\lambda)-S_\ell\leq C_{n,\ell}\lambda^2,
\]
and
\[
 S^H_{\lambda,\ell}
 =S_\ell-\lambda\,\mathsf H_\ell(B_0,B_0)+r_\lambda,
 \qquad
 -C_{n,\ell}\lambda^2\leq r_\lambda\leq0.
\]
Consequently the action quantum \eqref{eq:intro-quantum} satisfies
\[
 d_{\lambda,\ell}
 =\frac1{Q_\ell}S_\ell^{Q_\ell/2}
 -\frac12S_\ell^{Q_\ell/2-1}\mathsf H_\ell(B_0,B_0)\lambda
 +O_{n,\ell}(\lambda^2).
\]
Setting \(\theta_\ell=(q_\ell-2)^{-1}\),
\(\bar v_\lambda=(S^H_{\lambda,\ell})^{\theta_\ell}v_\lambda\) and
\(\widehat B_0=S_\ell^{\theta_\ell}B_0\), the function \(\bar v_\lambda\)
solves the coefficient-one perturbed equation and
\[
 \operatorname{dist}_{\dot S^1_\ell}
 \bigl(\bar v_\lambda,\widehat{\mathcal M}_{1,\ell}\bigr)
 \leq C_{n,\ell}\lambda,
 \qquad
 \widehat{\mathcal M}_{1,\ell}
 =\{\sigma\mathcal T_{h,\eta}\widehat B_0:
 \sigma\in\{-1,1\},\ h>0,\ \eta\in\Heis^n\}.
\]
All constants depend only on \((n,\ell)\), uniformly over the choice of
normalized minimizer.
\end{theorem}

Set
\[
 H_{0,\ell}=\mathsf H_\ell(B_0,B_0).
\]
By Theorem~\ref{thm:intro-nondegeneracy}, there is a unique
\(w_{1,\ell}\in Y_\ell\) such that
\begin{equation}\label{eq:intro-first-correction}
 \mathcal Q_{B_0}(w_{1,\ell},\phi)
 =\mathsf H_\ell(B_0,\phi)
 -H_{0,\ell}\int_{\Ucal}B_0^{q_\ell-1}\phi\dd\mu_\ell
 \qquad
 \bigl(\phi\in\dot S^1_\ell(\Ucal;\R)\bigr).
\end{equation}
Define
\[
 c_{2,\ell}
 =\mathsf H_\ell(B_0,w_{1,\ell})
 =\mathcal Q_{B_0}(w_{1,\ell},w_{1,\ell})>0.
\]

\begin{theorem}[Second-order expansion]
\label{thm:intro-second-order}
Let \(\ell>1\) be an integer, and let \(v_\lambda\) be any real normalized
minimizer of \(S^H_{\lambda,\ell}\).  As \(\lambda\downarrow0\), signs
\(\sigma_\lambda\) and frames \(g_\lambda\) may be chosen so that
\begin{equation}\label{eq:intro-modulated-shape}
 \mathcal T_{g_\lambda}^{-1}(\sigma_\lambda v_\lambda)
 =B_0+\lambda w_{1,\ell}
 +O_{\dot S^1_\ell}(\lambda^{q_\ell-1}).
\end{equation}
In particular, the remainder in \eqref{eq:intro-modulated-shape} is
\(o_{\dot S^1_\ell}(\lambda)\).  Moreover,
\begin{equation}\label{eq:intro-second-order-quotient}
 S^H_{\lambda,\ell}
 =S_\ell-\lambda H_{0,\ell}
 -\lambda^2c_{2,\ell}
 +O(\lambda^{q_\ell}),
\end{equation}
and hence the remainder after the quadratic term is \(o(\lambda^2)\).
The action quantum satisfies
\[
\begin{aligned}
 d_{\lambda,\ell}
={}&\frac1{Q_\ell}S_\ell^{Q_\ell/2}
 -\frac12S_\ell^{Q_\ell/2-1}H_{0,\ell}\lambda\\
&+\left[
 -\frac12S_\ell^{Q_\ell/2-1}c_{2,\ell}
 +\frac{Q_\ell-2}{8}S_\ell^{Q_\ell/2-2}H_{0,\ell}^2
 \right]\lambda^2
 +O(\lambda^{q_\ell}).
\end{aligned}
\]
\end{theorem}

The compact-potential layer concerns the repulsive problem, whose
variational geometry is different.  For a nonnegative potential whose
quadratic form extends continuously to \(\dot S^1_\ell(\Ucal)\) and whose
multiplier is compact, the bottom quotient remains the pure sharp constant
but is not attained
(Proposition~\ref{prop:compact-quotients}); a solution must therefore be
produced at a higher critical level.  Write
\[
 I_V(u)=\frac12\left(E_\ell(u)+\int_{\Ucal}Vu^2\dd\mu_\ell\right)
 -\frac1{q_\ell}\|u\|_{q_\ell}^{q_\ell},
 \qquad
 d_{0,\ell}=\frac1{Q_\ell}S_\ell^{Q_\ell/2}.
\]

\begin{theorem}[Compact repulsive potentials]
\label{thm:intro-potential}
Let \(\ell\geq1\) be an integer and let \(V\geq0\), \(V\not\equiv0\), be
locally integrable.  Assume that its quadratic form is continuous on
\(\dot S^1_\ell(\Ucal)\) and that
\[
 M_{\sqrt V}:\dot S^1_\ell(\Ucal)\longrightarrow L^2(\dd\mu_\ell),
 \qquad M_{\sqrt V}u=\sqrt V\,u,
\]
is compact.  If
\begin{equation}\label{eq:intro-orbit-smallness}
 \sup_{h>0,\,\eta\in\Heis^n}
 \int_{\Ucal}V|\mathcal T_{h,\eta}B_0|^2\dd\mu_\ell
 <\bigl(2^{2/Q_\ell}-1\bigr)S_\ell,
\end{equation}
then
\[
 \Gell_\ell v+Vv=|v|^{q_\ell-2}v
\]
has a nonzero one-sign weak solution satisfying
\[
 d_{0,\ell}<I_V(v)<2d_{0,\ell}.
\]
It is enough to assume
\(0<\|V\|_{L^{Q_\ell/2}(\dd\mu_\ell)}<(2^{2/Q_\ell}-1)S_\ell\); if in
addition the radial lift of \(V\) belongs locally to
\(L^p(\Heis^{n+\ell})\) for some \(p>Q_\ell/2\), the solution has an
everywhere positive locally H\"older representative.
\end{theorem}

\subsection{Method}
\label{sec:intro-method}

The radial lift organizes the proof, but none of the nonlinear conclusions
above follows from symmetry restriction alone.  The substantive part of
the method consists of six mechanisms that make the fixed-sector descent
valid.  We describe the obstruction and its resolution in each case.

\emph{(i) Capacity of the auxiliary axis.}
The reduced energy space is a completion of \(C_c^\infty(\Ucal)\) with no
trace at \(\rho=0\).  A reduced weak solution is therefore only known to be
tested against functions vanishing near the axis, whereas the
\(U(\ell)\)-average of an arbitrary full-space test is merely smooth
\emph{through} the axis.  The gap is closed by a capacity estimate: for
\(\ell>1\) a cutoff on \(\rho\in[\varepsilon,2\varepsilon]\) has radial
energy \(O(\varepsilon^{\ell-1})\), while for \(\ell=1\) a logarithmic
cutoff on \([\varepsilon^2,\varepsilon]\) has radial energy
\(O(|\log\varepsilon|^{-1})\).  In both cases the axis is null for the
reduced energy, and testing extends (Lemma~\ref{lem:axis}).  The same
estimate places the reduced bubbles in the completion, and it reappears in
the regularity theory of Corollary~\ref{cor:potential-positive}.  It is
also the exact point at which \(\ell=1\) becomes borderline, and the reason
the Hardy form fails to be continuous there
(Remark~\ref{rem:ell-one}).

\emph{(ii) The auxiliary center cannot survive.}
Benameur's decomposition on \(\Heis^{n+\ell}\) produces frames
\((h_k,\xi_k)\) with \(\xi_k=(z_k,w_k,t_k)\), and a priori the normalized
auxiliary center \(w_k/h_k\) may be unbounded.  It cannot be.  If it were,
choose \(M\) rotations \(A_1,\ldots,A_M\in U(\ell)\) separating
\(A_rw_k\); since the lifted sequence is \(U(\ell)\)-fixed, the \(M\)
recentered sequences have weak limits of equal norm whose relative frames
escape, and a finite-family Bessel estimate forces
\(M\|\Psi\|^2\leq\limsup_k\|U_k\|^2\) for every \(M\).  Once the auxiliary
center is bounded it converges, and the corresponding translation is
absorbed into the profile by strong continuity of translations; the exact
group factorization \(\xi_k=\eta_k\,\delta_{h_k}\zeta_k\), valid because
the two complex blocks are Hermitian-orthogonal, shows that no central
correction is created.  Bounded recentering preserves pairwise
orthogonality (Lemma~\ref{lem:bounded-recentering}).  This is why the
sector decomposition has a single noncompactness channel, in contrast with
the two channels present on real hyperbolic space \cite{LiWang2026Global}.

\emph{(iii) The optimizer cone shrinks, but not too much.}
The full-space optimizer cone contains bubbles centered off the fixed
sector, and the reduced cone does not, so a stability estimate is not
inherited by restriction: the deficit is controlled by the distance to the
\emph{larger} set.  Proposition~\ref{prop:distance-comparison} supplies the
missing comparison.  Its infinitesimal regime rests on
Lemma~\ref{lem:auxiliary-tangent}: the auxiliary translation tangent
\(\mathsf T_{\mathrm{aux}}(s)\) is real-linear, bounded below, and
annihilated by the Haar projection \(P_{U(\ell)}\), so moving the auxiliary
center off the sector costs, to first order, exactly as much in the
fixed-sector distance as in the full distance.  The large-center regime
uses that the Haar average of a bubble correlation over \(U(\ell)\)
vanishes when the center escapes, because the stabilizer of a direction has
measure zero.

\emph{(iv) The kernel shrinks exactly to the reduced orbit.}
The full Heisenberg linearization has translation directions in both the
visible \(z\)-variables and the auxiliary \(w\)-variables.  It is not
enough to discard the latter formally, because the calculation takes
place in a homogeneous completion.  The completed polar lift identifies
the reduced kernel with the \(U(\ell)\)-fixed part of the full kernel.
The auxiliary translation space is the standard \(U(\ell)\)
representation exhibited in Lemma~\ref{lem:auxiliary-tangent} and
therefore has no nonzero fixed vector; the reduced translations and
dilation are fixed.  The next full spectral level does retain fixed
vectors, which gives the exact normal coefficient
\(2/(n+\ell+4)\).  This normal inverse is the analytic input for the
modulated expansion \eqref{eq:intro-modulated-shape}.

\emph{(v) The optimal stability quotient has two compactness thresholds.}
A minimizing sequence for
\(\mathscr D_\ell(u)/d_\ell(u)^2\) may lose compactness either by approaching
the optimizer cone or by splitting into asymptotically orthogonal bubbles.
The first alternative is governed by the normal spectral coefficient
\(2/(n+\ell+4)\), whereas the second has the universal two-bubble level
\(2-2^{(n+\ell)/(n+\ell+1)}\).  We construct fixed-sector tests strictly
below both levels.  For the remaining weak splitting, the cone distance is
written in terms of the maximal bubble correlation.  Correlation
max-splitting and Lemma~\ref{lem:fixed-correlation-balance} show that two
nontrivial components must carry equal correlation mass: otherwise
rescaling the component with smaller correlation strictly lowers the
quotient.  The strict two-bubble inequality then rules out splitting and
forces strong convergence modulo a reduced frame.

\emph{(vi) Symmetry chooses the barycenter.}
For the repulsive problem a solution must be produced inside the
compactness window \((d_{0,\ell},2d_{0,\ell})\), and placing a min--max
level there requires continuous center and scale coordinates on the
constraint manifold, calibrated on the bubble orbit.  We obtain them from
the Busemann barycenter of the lifted critical mass, pushed to the ideal
boundary of \(\CHyp^{n+\ell+1}\).  Uniqueness and equivariance of the
barycenter force this point into the \(U(\ell)\)-fixed totally geodesic
copy of \(\CHyp^{n+1}\), on which the rotation-free translation--dilation
group acts simply transitively; the boundary Cayley Jacobian of
Frank--Lieb \cite[Appendix~A]{FrankLieb2012SharpHeisenberg} identifies the
push-forward of the centered bubble mass with the visual measure at the
basepoint, which calibrates the coordinates exactly on the orbit.  A degree
argument relative to \(\partial K_R\) then produces the level.

\subsection{Comparison with prior work}

\emph{Euclidean and polyharmonic theory.}
The sharp Sobolev inequality and its extremals were determined by Aubin
\cite{Aubin1976Sobolev} and Talenti \cite{Talenti1976BestConstant}.
Brezis--Nirenberg \cite{BrezisNirenberg1983Critical} restored compactness
via a spectral lower-order term, Struwe \cite{Struwe1984Global} identified
the full defect of compactness as a bubble decomposition, and
Bianchi--Egnell \cite{BianchiEgnell1991Sobolev} proved the quantitative
stability remainder.  Recent refinements include the sharp stability
constants with optimal dimensional dependence of
Dolbeault--Esteban--Figalli--Frank--Loss
\cite{DolbeaultEstebanFigalliFrankLoss2025Sharp}, K\"onig's strict upper
bound for the best stability constant \cite{Koenig2023BianchiEgnell} and
subsequent attainment theorem \cite{Koenig2025SobolevMinimizer}, the
quantitative theory of near-solutions to the Sobolev Euler--Lagrange
equation developed by Figalli--Glaudo
\cite{FigalliGlaudo2020SharpStability}, and the fractional extension of
Chen--Frank--Weth
\cite{ChenFrankWeth2013Remainder}.  For whole-space potentials,
Benci--Cerami \cite{BenciCerami1990Positive} introduced the sharp-scale
energy window and Passaseo \cite{Passaseo1996PositiveSolutions} developed
the barycenter--concentration construction; higher-order Euclidean
analogues involve new critical dimensions and polyharmonic compactness
phenomena \cite{EdmundsFortunatoJannelli1990Biharmonic,
Gazzola1998Polyharmonic,GazzolaGrunauSquassina2003Critical}.  The present
paper extends the Bianchi--Egnell and Benci--Cerami paradigms
simultaneously to the Geller geometry.  What is new is that both extensions
must pass through a symmetric sector of an ambient group: neither the
stability estimate nor the energy window is inherited by restriction, and
the sector is what supplies the compensating structure.

\emph{Real hyperbolic space.}
On \((\mathbb H^d,g_{\mathbb H})\), the spectral bottom \((d-1)^2/4\)
interacts with the Euclidean Sobolev scale.  The linear analysis behind
this interaction---sharp \(L^p\) multiplier theorems and heat and Green
function estimates on noncompact symmetric spaces---goes back to Anker
\cite{Anker1990Multipliers} and Anker--Ji \cite{AnkerJi1999HeatGreen},
and the spectrally shifted dispersive and Sobolev estimates that exploit
it on real hyperbolic and Damek--Ricci spaces are due to
Anker--Pierfelice \cite{AnkerPierfelice2009NLS} and
Anker--Pierfelice--Vallarino \cite{AnkerPierfeliceVallarino2015Wave}.
On the elliptic side, Mancini--Sandeep
\cite{ManciniSandeep2008Semilinear} established dimension-dependent
existence and nonexistence for the second-order critical equation with a
spectral term, and Bhakta--Sandeep \cite{BhaktaSandeep2012Poincare}
developed the Palais--Smale description and sign-changing solutions.
Lu--Yang \cite{LuYang2019PaneitzHyperbolic} developed the higher-order
inequality theory for Paneitz-type operators, and Li--Lu--Yang
\cite{LiLuYang2022HyperbolicBN} subsequently established GJMS-operator
Brezis--Nirenberg theorems and radial symmetry via moving planes for the
associated integral equations.  Bhakta--Ganguly--Karmakar--Mazumdar
\cite{BhaktaGangulyKarmakarMazumdar2025Stability,
BhaktaGangulyKarmakarMazumdar2025Struwe} proved Bianchi--Egnell and
Euler--Lagrange stability for the Poincar\'e--Sobolev problem and developed
a quantitative multibubble Struwe decomposition.  Li--Wang
\cite{LiWang2026Global} proved global compactness and existence for
higher-order concentrated potentials on real hyperbolic space, with two
genuinely different noncompactness channels.  What is new here is that the
invariant-sector decomposition has a \emph{single} channel: the
auxiliary-center control that collapses the second channel has no
real-hyperbolic analogue, because there is no ambient symmetric sector.

\emph{Heisenberg and CR geometry.}
The Folland--Stein inequality \cite{FollandStein1974Heisenberg} is the
sub-Riemannian Sobolev inequality on \(\Heis^n\).  Jerison--Lee
\cite{JerisonLee1988Extremals} classified its extremals via the CR Yamabe
problem, and Frank--Lieb \cite{FrankLieb2012SharpHeisenberg} computed sharp
constants in a broad Sobolev--Hardy--Littlewood--Sobolev framework.
Benameur \cite{Benameur2008ProfileHeisenberg} proved the full-space profile
decomposition under Heisenberg translations and dilations.  Loiudice
\cite{Loiudice2005ImprovedSobolev} established a full-space
Bianchi--Egnell remainder.  The CR-sphere spectral analysis of
Malchiodi--Uguzzoni \cite{MalchiodiUguzzoni2002Webster} identifies the
kernel at a Jerison--Lee bubble.  Qiang--Tang--Zhang
\cite{QiangTangZhang2026Nondegeneracy} recently obtained a direct
full-Heisenberg nondegeneracy theorem and applied it to the construction of
sign-changing solutions for a slightly subcritical problem.
Tang--Zhang--Zhang
\cite{TangZhangZhang2024Minimizer} proved that the optimal stability
quotient on the full Heisenberg group has a nontrivial minimizer, while
Chen--Lu--Tang--Wang
\cite{ChenLuTangWang2026HeisenbergStability} gave refined local CR-sphere
and global Heisenberg stability estimates with dimension-dependent
constants, paralleling the Euclidean theory of
\cite{DolbeaultEstebanFigalliFrankLoss2025Sharp}.  Chen--Fan--Liao
\cite{ChenFanLiao2025GlobalCompactness} established sharp,
dimension-dependent residual-to-distance estimates near finite sums of
weakly interacting Jerison--Lee bubbles.  Palatucci--Piccinini--Temperini
\cite{PalatucciPiccininiTemperini2025Struwe} proved Struwe-type compactness
on bounded Heisenberg domains.  What is new here is the descent of the
sharp inequality, equality classification, profile decomposition, and
quantitative stability mechanism to a fixed unitary sector: the sector is
not a subdomain, and the obstruction is the auxiliary center rather than a
boundary.
Theorem~\ref{thm:intro-optimal-stability} additionally proves attainment
of the optimal quotient inside the fixed sector itself; neither its
minimizer nor the two strict threshold inequalities follow from the
full-space attainment theorem by restriction.

\emph{The Geller family and complex hyperbolic analysis.}
The general scattering-theoretic passage from an interior eigenvalue
problem to conformally covariant operators at infinity was established by
Graham--Zworski \cite{GrahamZworski2003Scattering}.  In the CR setting,
Frank--Gonz\'alez--Monticelli--Tan
\cite{FrankGonzalezMonticelliTan2015CRExtension} realized the fractional CR
Laplacian through a linear extension problem on the Siegel domain and proved
the associated sharp trace inequality.  The Geller operators originate in
Geller's harmonic analysis of the CR sphere \cite{Geller1980Kohn}.
Lu--Yang
\cite{LuYang2022SiegelHardySobolevMazya} proved factorization formulas,
Poincar\'e--Sobolev and Hardy--Sobolev--Maz'ya inequalities on the Siegel
domain, together with sharp Adams and Hardy--Adams results.  Their work
established a powerful higher-order framework for analysis in the
complex-hyperbolic Siegel model.  Flynn--Lu--Yang
\cite{FlynnLuYang2025CRTrace} extended the extension framework to higher
orders by constructing conformally covariant boundary operators and proving
sharp CR Sobolev trace inequalities.  The present paper builds on these
contributions and develops the complementary nonlinear layer for the
interior critical equation: the equality manifold, its quantitative and
spectral geometry, the complete loss of compactness, and critical existence
and rigidity.
These features require information beyond the coercive inequalities
themselves and are obtained here through the fixed-sector realization.

\subsection{Scope}
\label{sec:intro-scope}

The results above are stated for integer \(\ell\) and, in the nonlinear
parts, for real-valued functions; the integer lift is the reason.  For
noninteger \(\ell>0\) the operator remains
formally self-adjoint for \(\dd\mu_\ell\), but there is no ambient sector
realization, so neither the classification nor the profile decomposition is
available.  At the endpoint \(\lambda=(\ell-1)^2\) the perturbed form loses
coercivity, and one expects a logarithmic Hardy--Sobolev--Maz'ya
improvement to be needed.  The constant \(\kappa_{n,\ell}\) in
Theorem~\ref{thm:intro-stability} is inherited from the full-space
inequality of \cite{Loiudice2005ImprovedSobolev} and is not asserted to be
optimal.  Theorem~\ref{thm:intro-optimal-stability} proves attainment of
the optimal quotient and gives strict threshold bounds, but does not
determine its numerical value.  The stronger
remainders \(O_{\dot S^1_\ell}(\lambda^2)\) in the modulated shape and
\(O(\lambda^3)\) in the quotient are not asserted: because
of the subcubic range \eqref{eq:intro-subcubic-range}, they require
regularity beyond the natural energy-space
linearization and a separate analysis of the parameter-dependent
indicial exponent at \(\rho=0\).  Finally, the smallness condition
\eqref{eq:intro-orbit-smallness} is sufficient and is not expected to be
optimal.

\subsection{Organization}

Section~\ref{sec:radial-lift} establishes the radial lift, the homogeneous
completion, the axis-capacity argument, and the classification of classical
and weak solutions.  Section~\ref{sec:sharp-profile} proves the sharp
inequality and the profile decomposition.  Section~\ref{sec:stability}
carries out the quantitative fixed-sector descent and proves
Theorem~\ref{thm:intro-stability}.
Section~\ref{sec:nondegeneracy} proves fixed-sector nondegeneracy, computes
the exact local spectral coefficient, and constructs the normal slice.
Section~\ref{sec:optimal-stability} proves the two strict threshold
inequalities and attainment of the optimal fixed-sector stability
quotient.
Section~\ref{sec:cayley} gives the
exact Cayley conjugation and identifies the intrinsic form domain.
Section~\ref{sec:hardy} treats the Hardy quotient, the positive
least-action solution, and the small-parameter asymptotics of
Theorems~\ref{thm:intro-rigidity}
and~\ref{thm:intro-second-order}, as well as the uniqueness and symmetry in
Theorem~\ref{thm:intro-symmetry}; it also proves the ground-state
nondegeneracy asserted there.
Section~\ref{sec:hardy-global} proves the global Palais--Smale compactness
theorem and the action quantum.  Section~\ref{sec:compact-potentials}
studies compact quadratic perturbations, and Section~\ref{sec:minmax}
constructs the Busemann center--scale linking and proves
Theorem~\ref{thm:intro-potential}.  Technical splitting and regularity
facts are gathered in
Appendix~\ref{app:technical}, and the constants relating the reduced and
intrinsic formulations in Appendix~\ref{app:normalization}.

\section{Radial lift and classification}
\label{sec:radial-lift}

\subsection{The lifting identity}

For \(m\geq1\), write
\(\Heis^m=\C^m\times\mathbb R\) with group law
\[
 (Z,t)(Z',t')
 =
 \bigl(Z+Z',
 t+t'+2\operatorname{Im}\langle Z,Z'\rangle\bigr),
 \qquad
 \langle Z,Z'\rangle
 =
 \sum_{j=1}^m Z_j\overline{Z'_j}.
\]
Thus the Hermitian product is linear in its first variable.  If
\(Z_j=x_j+iy_j\), the horizontal vector fields are
\[
 X_j=\partial_{x_j}+2y_j\partial_t,
 \qquad
 Y_j=\partial_{y_j}-2x_j\partial_t,
\]
and
\[
 \Delta_{\Heis^m}
 =
 \sum_{j=1}^m(X_j^2+Y_j^2).
\]
We use Lebesgue measure as Haar measure.

Recall that the homogeneous dimension and critical exponent are
\[
 Q_\ell=2(n+\ell)+2,
 \qquad
 q_\ell=\frac{2Q_\ell}{Q_\ell-2}
 =2+\frac{2}{n+\ell}.
\]

\begin{lemma}[Radial lifting identity]\label{lem:lift}
Let \(v\in C^2(\Ucal)\), and define
\[
 U(z,w,t)=v(z,t,|w|^2)
\]
for \(w\in\C^\ell\).  Away from \(w=0\),
\[
 -\Delta_{\Heis^{n+\ell}}U=\Gell_\ell v.
\]
If \(U\) is \(C^2\) across \(w=0\), the identity holds everywhere.
\end{lemma}

\begin{proof}
Write \(w_\alpha=a_\alpha+ib_\alpha\) and
\(\rho=|w|^2\).  The horizontal fields in the \(w_\alpha\)-variable give
\[
 X_{w_\alpha}U
 =
 2a_\alpha v_\rho+2b_\alpha v_t,
 \qquad
 Y_{w_\alpha}U
 =
 2b_\alpha v_\rho-2a_\alpha v_t.
\]
Differentiating once more, the mixed derivatives cancel and
\[
 (X_{w_\alpha}^2+Y_{w_\alpha}^2)U
 =
 4v_\rho
 +4|w_\alpha|^2(v_{\rho\rho}+v_{tt}).
\]
Summing over \(\alpha=1,\ldots,\ell\), and adding the horizontal
Laplacian in the \(z\)-variables, proves the identity for \(w\neq0\).
Continuity gives the last assertion.
\end{proof}

\subsection{Classical solutions}

\begin{theorem}\label{thm:classical}
Let \(v\in C^2(\Ucal)\) be positive and satisfy
\[
 \Gell_\ell v=v^{q_\ell-1}
 \quad\text{in }\Ucal.
\]
Assume that \(v^\uparrow\) extends to a positive \(C^2\) solution of
\[
 -\Delta_{\Heis^{n+\ell}}v^\uparrow
 =
 (v^\uparrow)^{q_\ell-1}
 \quad\text{on }\Heis^{n+\ell},
\]
and that
\[
 v^\uparrow\in L^{q_\ell}(\Heis^{n+\ell}).
\]
Then there exist \(z_0\in\C^n\), \(t_0\in\mathbb R\), and \(\beta>0\)
such that
\[
 v(z,t,\rho)
 =
 C_\beta
 \left[
  \bigl(\rho+\beta+|z-z_0|^2\bigr)^2
  +\bigl(t-t_0+2\operatorname{Im}\langle z,z_0\rangle\bigr)^2
 \right]^{-(n+\ell)/2},
\]
where
\[
 C_\beta^{q_\ell-2}=4(n+\ell)^2\beta,
 \qquad
 C_\beta=\bigl(4(n+\ell)^2\beta\bigr)^{(n+\ell)/2}.
\]
\end{theorem}

\begin{proof}
Lemma~\ref{lem:lift} shows that the lift
\(U=v^\uparrow\) satisfies the critical equation on
\(\Heis^{n+\ell}\).  Local subelliptic regularity upgrades the assumed
positive \(C^2\) solution to a smooth one.  The homogeneous dimension is
\(Q_\ell\), and the assumed exponent \(q_\ell\) is precisely the critical
integrability exponent.  The Jerison--Lee classification
\cite[Corollary~4.2]{JerisonLee1988Extremals}, after the constant rescaling
corresponding to the present scalar-curvature convention, gives
\[
 U(Z,t)
 =
 A
 \left[
  \bigl(\beta+|Z-Z_0|^2\bigr)^2
  +\bigl(t-t_0+2\operatorname{Im}\langle Z,Z_0\rangle\bigr)^2
 \right]^{-(n+\ell)/2}.
\]

The function \(U(z,w,t)\) is invariant under \(w\mapsto Bw\) for every
\(B\in U(\ell)\).  A Heisenberg bubble has a unique maximum.  If
\(Z_0=(z_0,w_0)\), invariance would make every
\((z_0,Bw_0,t_0)\) a maximum; hence \(w_0=0\).  Substituting
\(\rho=|w|^2\) gives the stated reduced profile, apart from its amplitude.

It remains to impose the coefficient-one normalization.  By Heisenberg
translation it is enough to use the centered profile.  Put
\[
 r=\rho+\beta+|z|^2,\qquad s=t,\qquad
 D=r^2+s^2,\qquad f=CD^{-(n+\ell)/2}.
\]
A direct differentiation gives
\[
 f_r=-C(n+\ell)rD^{-(n+\ell)/2-1},
 \qquad
 f_{rr}+f_{ss}
 =
 C(n+\ell)^2D^{-(n+\ell)/2-1}.
\]
Substitution into the reduced operator yields
\[
 \Gell_\ell f
 =
 4C(n+\ell)^2\beta D^{-(n+\ell+2)/2}.
\]
Since \(q_\ell-1=1+2/(n+\ell)\), comparison with
\(f^{q_\ell-1}\) gives
\[
 C^{q_\ell-2}=4(n+\ell)^2\beta.
\]
This is the asserted value of \(C_\beta\).
\end{proof}

\subsection{The homogeneous energy space and the auxiliary axis}

Set
\[
 d\mu_\ell
 =
 \rho^{\ell-1}\dd z\dd t\dd\rho
\]
and
\[
 \Gamma_\ell(v,\varphi)
 =
 \langle\nabla_Hv,\nabla_H\varphi\rangle
 +4\rho(v_\rho\varphi_\rho+v_t\varphi_t).
\]
We denote by \(\dot S^1_\ell(\Ucal)\) the completion of
\(C_c^\infty(\Ucal)\) for the norm
\[
 \|v\|_{\dot S^1_\ell}^2
 =
 \int_{\Ucal}\Gamma_\ell(v,v)\dd\mu_\ell.
\]
No boundary trace at \(\rho=0\) is part of this definition.

\begin{lemma}[Smooth invariant factorization]\label{lem:factorization}
Let \(F(z,w,t)\) be smooth and invariant under
\(w\mapsto Bw\) for every \(B\in U(\ell)\).  Then there is a function
\(\varphi(z,t,\rho)\), smooth for \(\rho\geq0\), such that
\[
 F(z,w,t)=\varphi(z,t,|w|^2).
\]
If \(F\) is compactly supported, then so is \(\varphi\).
\end{lemma}

\begin{proof}
The unitary group acts transitively on every sphere in \(\C^\ell\), so
\(F\) depends on \(w\) only through \(|w|\).  Its restriction to the line
\(w=re_1\) is a smooth even function of \(r\), with \((z,t)\) as smooth
parameters.  The parameter-dependent even-function lemma, obtained from
Hadamard's lemma, writes this restriction smoothly as a function of
\(r^2\).  This gives the asserted factorization through \(|w|^2\).
\end{proof}

\begin{lemma}[Polar lifting]\label{lem:polar-lift}
Let \(c_\ell=|\mathbb S^{2\ell-1}|/2\).  For
\(v\in C_c^\infty(\Ucal)\),
\[
 \int_{\Heis^{n+\ell}}|v^\uparrow|^{q_\ell}
 =
 c_\ell\int_{\Ucal}|v|^{q_\ell}\dd\mu_\ell
\]
and
\[
 \int_{\Heis^{n+\ell}}|\nabla_Hv^\uparrow|^2
 =
 c_\ell\int_{\Ucal}\Gamma_\ell(v,v)\dd\mu_\ell.
\]
Consequently, the radial lift extends continuously from
\(\dot S^1_\ell(\Ucal)\) to a Hilbert-space isomorphism onto the
\(U(\ell)\)-invariant subspace of
\(\dot S^1(\Heis^{n+\ell})\).
\end{lemma}

\begin{proof}
The substitution \(\rho=|w|^2\) in polar coordinates gives
\[
 dw=c_\ell\rho^{\ell-1}\dd\rho.
\]
The first identity follows immediately.  The formulas in the proof of
Lemma~\ref{lem:lift} give
\[
 \sum_{\alpha=1}^{\ell}
 \bigl(|X_{w_\alpha}v^\uparrow|^2
       +|Y_{w_\alpha}v^\uparrow|^2\bigr)
 =
 4\rho(v_\rho^2+v_t^2),
\]
which proves the energy identity.  The extension by completion is a
similarity and therefore has closed range.

It remains to prove surjectivity.  Let
\(F\in\dot S^1(\Heis^{n+\ell})^{U(\ell)}\), and choose
\(F_j\in C_c^\infty(\Heis^{n+\ell})\) converging to \(F\) in homogeneous
energy.  Replacing \(F_j\) by its Haar average over \(U(\ell)\), we may
assume that every \(F_j\) is invariant.  Lemma~\ref{lem:factorization}
then gives
\[
 F_j(z,w,t)=f_j(z,t,|w|^2),
\]
where \(f_j\) is smooth up to \(\rho=0\) and compactly supported in
\(\C^n\times\R\times[0,\infty)\).

Choose a smooth cutoff \(\eta_\varepsilon(\rho)\) which vanishes near the
axis and tends to one away from it.  For \(\ell>1\) it may be supported on
a transition interval of length \(O(\varepsilon)\), with
\[
 \int_0^\infty \rho^\ell|\eta_\varepsilon'|^2\dd\rho
 =O(\varepsilon^{\ell-1});
\]
for \(\ell=1\) a logarithmic transition between
\(\varepsilon^2\) and \(\varepsilon\) gives instead
\[
 \int_0^\infty \rho|\eta_\varepsilon'|^2\dd\rho
 =O(|\log\varepsilon|^{-1}).
\]
The polar energy identity, the boundedness of \(f_j\) near the axis, and
absolute continuity of its energy integral show that
\[
 (\eta_\varepsilon f_j)^\uparrow\longrightarrow F_j
 \quad\text{in }\dot S^1(\Heis^{n+\ell}).
\]
Since \(\eta_\varepsilon f_j\in C_c^\infty(\Ucal)\), every \(F_j\) belongs
to the closed range of the completed lift.  Passing to the limit in \(j\)
proves that \(F\) belongs to the same range.
\end{proof}

\begin{lemma}[Tests through the axis]\label{lem:axis}
Suppose that \(v\in\dot S^1_\ell(\Ucal)\) is nonnegative and satisfies
\[
 \int_{\Ucal}\Gamma_\ell(v,\varphi)\dd\mu_\ell
 =
 \int_{\Ucal}v^{q_\ell-1}\varphi\dd\mu_\ell
 \qquad
 \text{for every }\varphi\in C_c^\infty(\Ucal).
\]
Then the same identity holds for compactly supported tests that are smooth
up to \(\rho=0\).
\end{lemma}

\begin{proof}
Let \(\varphi\) be such a test.  When \(\ell>1\), choose a cutoff
\(\eta_\varepsilon\) which vanishes near \(\rho=0\), equals one for
\(\rho\geq2\varepsilon\), and satisfies
\(|\eta_\varepsilon'|\leq C/\varepsilon\).  Its radial energy is
\[
 \int_\varepsilon^{2\varepsilon}
 \rho^\ell|\eta_\varepsilon'|^2\dd\rho
 =
 O(\varepsilon^{\ell-1}).
\]
When \(\ell=1\), choose instead a smooth logarithmic cutoff which vanishes
on \(0<\rho\leq\varepsilon^2\), equals one for
\(\rho\geq\varepsilon\), and is flat at the endpoints of its transition
interval.  It may be chosen so that
\[
 |\eta_\varepsilon'(\rho)|
 \leq
 \frac{C}{\rho|\log\varepsilon|},
 \qquad
 \int\rho|\eta_\varepsilon'|^2\dd\rho
 =
 O(|\log\varepsilon|^{-1}).
\]
Apply the weak identity to \(\eta_\varepsilon\varphi\).  The
Cauchy--Schwarz inequality and the preceding cutoff estimates remove the
radial-gradient error.  Absolute continuity of the energy integral and the
critical embedding remove the remaining bilinear and nonlinear terms on
the shrinking strip.  Letting \(\varepsilon\downarrow0\) proves the claim.
\end{proof}

\begin{lemma}[The lifted weak equation]\label{lem:weak-lift}
Under the hypotheses of Lemma~\ref{lem:axis}, the radial lift
\(U=v^\uparrow\) satisfies
\[
 \int_{\Heis^{n+\ell}}
 \langle\nabla_HU,\nabla_H\Phi\rangle
 =
 \int_{\Heis^{n+\ell}}U^{q_\ell-1}\Phi
\]
for every \(\Phi\in C_c^\infty(\Heis^{n+\ell})\).
\end{lemma}

\begin{proof}
Average \(\Phi\) over \(U(\ell)\):
\[
 \Phi^\sharp(z,w,t)
 =
 \int_{U(\ell)}\Phi(z,Bw,t)\dd B.
\]
The lift \(U\), Haar measure, and the horizontal metric are invariant under
this action, so both weak pairings are unchanged when \(\Phi\) is replaced
by \(\Phi^\sharp\).  Lemma~\ref{lem:factorization} gives
\[
 \Phi^\sharp(z,w,t)=\varphi(z,t,|w|^2)
\]
with \(\varphi\) smooth through \(\rho=0\).  Lemma~\ref{lem:axis}, followed
by the polar identities in Lemma~\ref{lem:polar-lift}, now gives the full
weak equation.
\end{proof}

\subsection{Weak solutions}

\begin{theorem}\label{thm:weak}
Let \(v\in\dot S^1_\ell(\Ucal)\) be nonnegative and nonzero.  Suppose
\[
 \int_{\Ucal}\Gamma_\ell(v,\varphi)\dd\mu_\ell
 =
 \int_{\Ucal}v^{q_\ell-1}\varphi\dd\mu_\ell
 \qquad
 \text{for every }\varphi\in C_c^\infty(\Ucal).
\]
No boundary condition or trace at \(\rho=0\) is assumed.
Then \(v\) is one of the bubbles in Theorem~\ref{thm:classical}.
\end{theorem}

\begin{proof}
The completion hypothesis is used at two distinct points.  First,
Lemma~\ref{lem:polar-lift} gives
\[
 U=v^\uparrow\in\dot S^1(\Heis^{n+\ell}),\qquad
 \|U\|_{\dot S^1}^2=c_\ell E_\ell(v),
\]
and the Folland--Stein inequality gives \(U\in L^{q_\ell}(\Heis^{n+\ell})\).
Second, the axis cutoffs in Lemma~\ref{lem:axis} allow the reduced weak
identity to be tested against functions smooth through \(\rho=0\).  If
\(\Phi\in C_c^\infty(\Heis^{n+\ell})\), averaging over \(U(\ell)\)
leaves both full-space pairings unchanged.  Lemma~\ref{lem:factorization}
writes the average as
\(\Phi^\sharp(z,w,t)=\varphi(z,t,|w|^2)\), with \(\varphi\) smooth up to
\(\rho=0\).  Polar integration (Lemma~\ref{lem:weak-lift}) therefore yields
\[
 \int_{\Heis^{n+\ell}}\langle\nabla_HU,\nabla_H\Phi\rangle
 =\int_{\Heis^{n+\ell}}U^{q_\ell-1}\Phi
\]
for every compactly supported smooth \(\Phi\).

Thus \(U\) is a nonnegative, nonzero weak solution in the full homogeneous
completion.  The regularity and positivity theorem
\cite[Theorem~3.6(c)]{IvanovVassilev2015Obata} applies to this nonnegative
nonzero full-completion weak solution and gives a smooth strictly positive
representative of \(U\).  Since the representative remains
\(U(\ell)\)-invariant, Lemma~\ref{lem:factorization} gives a smooth reduced
representative through the auxiliary axis.  The critical
\(L^{q_\ell}\) integrability and all hypotheses of
Theorem~\ref{thm:classical} now hold, and that theorem gives the stated
coefficient-one reduced bubble.
\end{proof}

\section{The sharp inequality and the reduced profile decomposition}
\label{sec:sharp-profile}

We determine the best constant in the critical embedding and characterize
all losses of compactness.  The factor
\[
 c_\ell=\frac{|S^{2\ell-1}|}{2}
 =\frac{\pi^\ell}{\Gamma(\ell)}.
\]
is the polar Jacobian obtained by writing \(w\in\C^\ell\) in terms of
\(\rho=|w|^2\); equivalently,
\(\int_{\C^\ell}f(|w|^2)\,dw
=c_\ell\int_0^\infty f(\rho)\rho^{\ell-1}\,d\rho\).

\begin{theorem}[Sharp reduced Sobolev inequality]
\label{thm:sharp-sobolev}
For every \(v\in\dot S^1_\ell(\Ucal)\),
\begin{equation}\label{eq:sharp-reduced}
 S_\ell
 \left(\int_{\Ucal}|v|^{q_\ell}\dd\mu_\ell\right)^{2/q_\ell}
 \leq E_\ell(v),
\end{equation}
where
\begin{equation}\label{eq:sharp-constant}
 S_\ell
 =\frac{4\pi(n+\ell)^2}
 {\bigl(2^{2(n+\ell)}(n+\ell)!c_\ell\bigr)^{1/(n+\ell+1)}}.
\end{equation}
The constant is optimal and is attained.  The nonzero real equality cases
are exactly
\begin{equation}\label{eq:reduced-bubbles}
 B_{A,\beta,z_0,t_0}(z,t,\rho)
 =A\left[
  (\rho+\beta+|z-z_0|^2)^2
  +(t-t_0+2\operatorname{Im}\langle z,z_0\rangle)^2
 \right]^{-(n+\ell)/2},
\end{equation}
where \(A\ne0\), \(\beta>0\), \(z_0\in\C^n\), and \(t_0\in\R\).
If \(v>0\) is an equality case, then
\[
 \Gell_\ell v=\Lambda v^{q_\ell-1},
 \qquad
 \Lambda=S_\ell\|v\|_{q_\ell}^{2-q_\ell},
\]
and \(\Lambda^{1/(q_\ell-2)}v\) solves the coefficient-one equation.
\end{theorem}

\begin{proof}
Let \(\mathscr L_\ell v=v^\uparrow\).  Lemma~\ref{lem:polar-lift}
and density give
\[
 \|\mathscr L_\ell v\|_{\dot S^1(\Heis^{n+\ell})}^2
 =c_\ell E_\ell(v),
 \qquad
 \|\mathscr L_\ell v\|_{q_\ell}^{q_\ell}
 =c_\ell\|v\|_{q_\ell}^{q_\ell}.
\]
In the present sub-Laplacian convention, the sharp full-space constant is
\[
 S_{n+\ell}^{\mathrm{FS}}
 =\frac{4\pi(n+\ell)^2}
 {\bigl(2^{2(n+\ell)}(n+\ell)!\bigr)^{1/(n+\ell+1)}};
\]
this is the Frank--Lieb constant after converting their factor
\(1/4\) convention \cite{FrankLieb2012SharpHeisenberg}.
Applying the full-space inequality to \(\mathscr L_\ell v\) gives
\eqref{eq:sharp-reduced}--\eqref{eq:sharp-constant}.

Suppose equality holds.  The lifted function is a full Heisenberg
extremal.  The full equality classification gives a Heisenberg bubble.
Since the lift is invariant under the action of \(U(\ell)\) on the
auxiliary variable, the uniqueness of the maximum forces the auxiliary
component of its center to be zero.  The resulting function is
\eqref{eq:reduced-bubbles}.

Conversely, every function in \eqref{eq:reduced-bubbles} belongs to the
reduced completion.  We indicate the only point not supplied directly by
the full-space formula.  After translation and dilation, take the centered
bubble and multiply it by an outer cutoff and an axis cutoff.  The decay of
the bubble and its horizontal derivatives makes the outer-cutoff error tend
to zero.  For \(\ell>1\), an axis cutoff supported in
\(\rho\in[\varepsilon,2\varepsilon]\) has radial energy
\[
 \int_\varepsilon^{2\varepsilon}
 \rho^\ell|\eta_\varepsilon'|^2\dd\rho
 =O(\varepsilon^{\ell-1}).
\]
For \(\ell=1\), a logarithmic cutoff on
\([\varepsilon^2,\varepsilon]\) satisfies
\[
 |\eta_\varepsilon'(\rho)|
 \leq\frac{C}{\rho|\log\varepsilon|},
 \qquad
 \int_{\varepsilon^2}^{\varepsilon}
 \rho|\eta_\varepsilon'|^2\dd\rho
 =O(|\log\varepsilon|^{-1}).
\]
The remaining terms vanish by absolute continuity of the bubble energy.
Thus smooth compactly supported functions approximate the bubble in
\(\dot S^1_\ell(\Ucal)\).  The first variation of
\eqref{eq:sharp-reduced} gives the Euler--Lagrange equation and its stated
normalization.
\end{proof}

For \(\eta=(z_*,t_*)\in\Heis^n\) and \(h>0\), write
\[
 (\mathcal T_{h,\eta}\phi)(z,t,\rho)
 =h^{-(n+\ell)}\phi\!\left(
 \frac{z-z_*}{h},
 \frac{t-t_*+2\operatorname{Im}\langle z,z_*\rangle}{h^2},
 \frac{\rho}{h^2}\right).
\]
We also write \(\mathcal T_g=\mathcal T_{h,\eta}\) when
\(g=(h,\eta)\).  Thus \(g\) is a frame parameter and \(\mathcal T_g\) is
the associated frame action.  These actions are unitary for
\(E_\ell\) and preserve the critical norm.

For the full group \(\Heis^{n+\ell}\), set
\[
 \delta_h(Z,t)=(hZ,h^2t),\qquad
 (\tau_\xi F)(x)=F(\xi^{-1}x),\qquad
 (S_hF)(x)=h^{-(n+\ell)}F(\delta_{1/h}x),
\]
and write
\[
 \mathcal T^{\mathrm{full}}_{h,\xi}=\tau_\xi S_h.
\]
The dilation \(\delta_h\) also denotes the restriction of this automorphism
to the embedded subgroup \(\Heis^n\).  These full-space operators are
unitary in the homogeneous energy and in \(L^{q_\ell}\).

\begin{corollary}\label{cor:sharp-compactness}
Let \((v_k)\subset\dot S^1_\ell(\Ucal)\) be a real sequence satisfying
\(\|v_k\|_{q_\ell}=1\) and
\[
 E_\ell(v_k)\longrightarrow S_\ell.
\]
Then, after passage to a subsequence, there exist signs
\(\sigma_k\in\{-1,1\}\) and frame parameters \(g_k\) such that
\[
 \|v_k-\sigma_k\mathcal T_{g_k}B_0\|_{\dot S^1_\ell}
 \longrightarrow0,
\]
where \(B_0\) is any fixed positive normalized centered extremal.
\end{corollary}

\begin{theorem}[Reduced profile decomposition]
\label{thm:reduced-profile}
Let \((v_k)\) be bounded in \(\dot S^1_\ell(\Ucal)\).  After passage to one
subsequence, there are profiles
\(\phi^j\in\dot S^1_\ell(\Ucal)\), parameters
\(g_k^j=(h_k^j,\eta_k^j)\), and remainders \(r_k^J\) such that, for each
fixed \(J\),
\begin{equation}\label{eq:profile-expansion}
 v_k=\sum_{j=1}^J\mathcal T_{g_k^j}\phi^j+r_k^J,
\end{equation}
\begin{align}
 E_\ell(v_k)
 &=\sum_{j=1}^JE_\ell(\phi^j)+E_\ell(r_k^J)+o_k(1),
 \label{eq:profile-energy}\\
 \|v_k\|_{q_\ell}^{q_\ell}
 &=\sum_{j=1}^J\|\phi^j\|_{q_\ell}^{q_\ell}
 +\|r_k^J\|_{q_\ell}^{q_\ell}+o_k(1),
 \label{eq:profile-mass}
\end{align}
and
\begin{equation}\label{eq:profile-remainder}
 \lim_{J\to\infty}\limsup_{k\to\infty}
 \|r_k^J\|_{q_\ell}=0.
\end{equation}
For \(i\ne j\), either
\[
 \left|\log\frac{h_k^i}{h_k^j}\right|\longrightarrow\infty,
\]
or the scales may be represented by the same sequence and
\[
 \left|
 \delta_{1/h_k^i}\bigl((\eta_k^i)^{-1}\eta_k^j\bigr)
 \right|_{\Heis^n}\longrightarrow\infty.
\]
At every extraction step the preceding remainder is weakly null in the
new profile frame.  The profiles and remainders may be chosen real when the
original sequence is real.
\end{theorem}

\begin{proof}
Benameur's profile theorem on \(\Heis^{n+\ell}\), at order one, gives a
full-space decomposition under Heisenberg translations and dilations
\cite{Benameur2008ProfileHeisenberg}.  We give the invariant descent because
it is the step that determines the parameter set in the theorem.  Let
\(K_{\mathrm{aux}}=U(\ell)\) act on the auxiliary variable
\(w\in\C^\ell\).  The polar lift maps
\(\dot S^1_\ell(\Ucal)\) onto the \(K_{\mathrm{aux}}\)-fixed subspace of
\(\dot S^1(\Heis^{n+\ell})\) and is a similarity with squared energy
factor \(c_\ell\).

\begin{lemma}[Spectral-core bridge]\label{lem:S0-bridge}
Let \(X=\dot H^1(\Heis^{n+\ell})\) and let
\(\mathcal S_0(\Heis^{n+\ell})\) be the spectrally localized Schwartz class.
If \((F_k)\) is bounded in \(X\), \(F\in X\), and
\[
 \langle F_k-F,\chi\rangle_{\mathcal S',\mathcal S}\to0
 \qquad\text{for every fixed }\chi\in\mathcal S_0,
\]
then \(F_k\rightharpoonup F\) weakly in \(X\).
\end{lemma}

\begin{proof}
The class \(\mathcal S_0\) is dense in \(X\), and spectral localization
away from zero is preserved by the sub-Laplacian.  Hence, for every fixed
\(\varphi\in\mathcal S_0\),
\(-\Delta_{\Heis^{n+\ell}}\varphi\in\mathcal S_0\) and
\[
 (F_k-F,\varphi)_X
 =\langle F_k-F,-\Delta_{\Heis^{n+\ell}}\varphi
   \rangle_{\mathcal S',\mathcal S}\to0.
\]
Put \(C=\sup_k\|F_k-F\|_X<\infty\).  Given \(\psi\in X\), choose a fixed
\(\varphi\in\mathcal S_0\) with \(\|\psi-\varphi\|_X<\varepsilon\).  Then
\[
 \limsup_{k\to\infty}|(F_k-F,\psi)_X|
 \le C\varepsilon.
\]
Letting \(\varepsilon\downarrow0\) proves weak convergence in \(X\).
\end{proof}

The normalized sequences and finite-stage remainders in the full-space
theorem are bounded in \(X\): translations and dilations are unitary, and the
finite-stage energy identity is uniform.  Lemma~\ref{lem:S0-bridge} upgrades
every \(\mathcal S_0\)-distribution limit used in the extraction to a weak
\(X\)-limit.

Let \((U_k)\) be a bounded lifted sequence and suppose that
\[
 (\mathcal T^{\mathrm{full}}_{h_k,\xi_k})^{-1}U_k
 \rightharpoonup\Psi\ne0,
 \qquad \xi_k=(z_k,w_k,t_k).
\]
The normalized auxiliary center is \(w_k/h_k\).  We claim that it is
bounded.  Otherwise, after passing to a subsequence,
\(|w_k|/h_k\to\infty\).  Given \(M\ge2\), choose rotations
\(A_1,\ldots,A_M\in U(\ell)\) such that
\(|A_rw_k-A_sw_k|\ge c_M|w_k|\) for \(r\ne s\), and replace \(\xi_k\) by
\(\xi_{r,k}=(z_k,A_rw_k,t_k)\).  Since \(U_k\) is
\(K_{\mathrm{aux}}\)-fixed,
\[
 (\mathcal T^{\mathrm{full}}_{h_k,\xi_{r,k}})^{-1}U_k
 \rightharpoonup \mathscr R_{A_r}\Psi,
\]
where \(\mathscr R_A\) denotes rotation by \(A\) in the auxiliary variable.
All these limits have norm \(\|\Psi\|\).  Their relative normalized
centers have horizontal components of size at least
\(c_M|w_k|/h_k\), so the corresponding relative frame actions converge to
zero in the weak operator topology.  For completeness, if
\(g_{r,k}^{-1}g_{s,k}\rightharpoonup0\) for \(r\ne s\) and
\(g_{r,k}^{-1}U_k\rightharpoonup\Psi^r\), then, with
\(\Phi_k=\sum_{r=1}^Mg_{r,k}\Psi^r\),
\[
 \langle U_k,\Phi_k\rangle
 \longrightarrow\sum_{r=1}^M\|\Psi^r\|^2,
 \qquad
 \|\Phi_k\|^2
 \longrightarrow\sum_{r=1}^M\|\Psi^r\|^2.
\]
Cauchy--Schwarz therefore gives the finite-family Bessel estimate
\[
 \sum_{r=1}^M\|\Psi^r\|^2
 \leq\limsup_{k\to\infty}\|U_k\|^2.
\]
Applied to the rotated limits, this yields
\(M\|\Psi\|^2\leq\limsup_k\|U_k\|^2\) for every \(M\), a contradiction.

We may consequently assume that \(w_k/h_k\to\omega\).  Put
\[
 \eta_k=(z_k,0,t_k),\qquad
 \zeta_k=(0,w_k/h_k,0),\qquad \zeta=(0,\omega,0).
\]
The orthogonality of the \(z\)- and \(w\)-variables gives the exact group
factorization \(\xi_k=\eta_k\delta_{h_k}\zeta_k\), and hence
\[
 \mathcal T^{\mathrm{full}}_{h_k,\xi_k}\Psi
 =\mathcal T^{\mathrm{full}}_{h_k,\eta_k}(\tau_{\zeta_k}\Psi).
\]
Strong continuity of translations in the homogeneous energy space shows
that \(\tau_{\zeta_k}\Psi\to\Phi:=\tau_\zeta\Psi\).  Thus the original
profile term differs by \(o(1)\) in both the energy and critical norms from
the zero-auxiliary-center term
\(\mathcal T^{\mathrm{full}}_{h_k,\eta_k}\Phi\).  Moreover,
\(\mathcal T^{\mathrm{full}}_{h_k,\eta_k}\) commutes with
\(K_{\mathrm{aux}}\), so every
\((\mathcal T^{\mathrm{full}}_{h_k,\eta_k})^{-1}U_k\) is
\(K_{\mathrm{aux}}\)-fixed.  Its weak limit \(\Phi\) is therefore
\(K_{\mathrm{aux}}\)-fixed and descends uniquely to a profile on \(\Ucal\).

\begin{lemma}[Orthogonality under bounded recentering]
\label{lem:bounded-recentering}
For \(a\in\{i,j\}\), suppose
\(\xi_k^a=\eta_k^a\,\delta_{h_k^a}\zeta_k^a\), where \((\zeta_k^a)\)
remains in a compact subset of \(\Heis^{n+\ell}\).  Replacing
\((h_k^a,\xi_k^a)\) by \((h_k^a,\eta_k^a)\) preserves pairwise profile
orthogonality.
\end{lemma}

\begin{proof}
The replacement does not change the scales, so it preserves
\(|\log(h_k^i/h_k^j)|\to\infty\).  In the equal-scale alternative,
\(h_k^i=h_k^j\) literally for every \(k\).  The automorphism property and
group law give
\[
 \delta_{1/h_k^i}\bigl((\xi_k^i)^{-1}\xi_k^j\bigr)
 =(\zeta_k^i)^{-1}\,
 \delta_{1/h_k^i}\bigl((\eta_k^i)^{-1}\eta_k^j\bigr)\,
 \zeta_k^j.
\]
A proper homogeneous gauge is bounded on compact sets.  Multiplication on
either side by elements of fixed compact sets preserves escape to infinity.
Thus the normalized relative center on the left escapes if and only if the
middle factor escapes, and the corresponding relative frame actions converge
to zero in the weak operator topology in exactly the same cases before and
after recentering.
\end{proof}

Lemma~\ref{lem:bounded-recentering} confirms that the recentering from
\(\xi_k\) to \(\eta_k\) preserves pairwise orthogonality.

All profile extractions, limits of the bounded auxiliary centers, and local
almost-everywhere limits are chosen on one countable diagonal subsequence.
Local Rellich compactness then makes every preceding remainder almost
everywhere null in the next profile frame.  The Hilbert Pythagorean identity
and a finite telescoping application of the Brezis--Lieb lemma give
\eqref{eq:profile-energy} and \eqref{eq:profile-mass}.  The outer limit in
\(J\) is taken only after this common subsequence has been fixed.  Finally,
the full-space critical remainder estimate descends through the polar
identities and yields \eqref{eq:profile-remainder}.
\end{proof}

\begin{proof}[Proof of Corollary~\ref{cor:sharp-compactness}]
Apply Theorem~\ref{thm:reduced-profile}.  The energy and mass splittings,
together with the strict concavity of \(s\mapsto s^{2/q_\ell}\), leave
exactly one nonzero profile carrying the full critical mass.  Equality in
\eqref{eq:sharp-reduced} identifies that profile, and the Hilbert remainder
then converges strongly.
\end{proof}

Theorems~\ref{thm:sharp-sobolev}, \ref{thm:weak}, and
\ref{thm:reduced-profile} prove
Theorems~\ref{thm:intro-sharp} and \ref{thm:intro-profile}.

\section{Quantitative fixed-sector descent and reduced stability}
\label{sec:stability}

The full-space Heisenberg group \(\Heis^{n+\ell}\) carries a
Bianchi--Egnell-type stability estimate proved by Loiudice
\cite{Loiudice2005ImprovedSobolev}: there exists
\(\alpha_{Q_\ell}>0\) such that for every real
\(F\in\dot H^1(\Heis^{n+\ell})\),
\begin{equation}\label{eq:full-stability}
 \|F\|_{\dot H^1}^2
 -S_{n+\ell}^{\mathrm{FS}}\|F\|_{q_\ell}^2
 \geq\alpha_{Q_\ell}
 \operatorname{dist}_{\dot H^1}
 \!\bigl(F,\mathcal M_{n+\ell}^{\R,\mathrm{full}}\bigr)^2,
\end{equation}
where \(S_{n+\ell}^{\mathrm{FS}}\) is the sharp full-space constant and
\(\mathcal M_{n+\ell}^{\R,\mathrm{full}}\) is the full real optimizer cone.
The goal of this section is to descend from \eqref{eq:full-stability} to the
reduced \(U(\ell)\)-fixed sector and prove
Theorem~\ref{thm:intro-stability}.

\subsection{The fixed and full optimizer cones}

For \(x\in\Heis^{n+\ell}\), let
\[
 (\tau_\xi f)(x)=f(\xi^{-1}x),\qquad
 (S_hf)(x)=h^{-(n+\ell)}f(\delta_{1/h}x),
 \qquad
 \mathcal T^{\mathrm{full}}_{h,\xi}=\tau_\xi S_h.
\]
These operators are isometries of \(\dot H^1(\Heis^{n+\ell})\) and
\(L^{q_\ell}(\Heis^{n+\ell})\).  Write
\(K_{\mathrm{aux}}=U(\ell)\) for the rotations of the auxiliary
\(\C^\ell\)-block, and let \(P_{U(\ell)}\) be the orthogonal projection onto
the fixed subspace, obtained by averaging over normalized Haar probability.
The full-space optimizer cone decomposes as
\[
 \mathcal M_{n+\ell}^{\R,\mathrm{full}}
 =\{A\,\mathcal T^{\mathrm{full}}_{h,\xi}\widetilde B:
 A\in\R,\;h>0,\;\xi\in\Heis^{n+\ell}\},
\]
where \(\widetilde B\) is the positive centered full-space bubble with
\(\|\widetilde B\|_{q_\ell}=1\); the choice \(A=0\) includes the cone
vertex.  Its \(K_{\mathrm{aux}}\)-fixed subset is
\[
 \mathcal M_{n+\ell}^{\R,\mathrm{fix}}
 =\{A\,\mathcal T^{\mathrm{full}}_{h,\bar\eta}\widetilde B:
 A\in\R,\;h>0,\;
 \bar\eta=(z_0,0,t_0),\;(z_0,t_0)\in\Heis^n\}.
\]
Indeed, the absolute value of a nonzero full bubble has a unique maximum at
its center.  If the bubble is fixed by every auxiliary rotation, that center
must have zero auxiliary component; the converse is immediate.  Lemma
\ref{lem:polar-lift} therefore gives the exact cone identity
\[
 \mathcal M_{n+\ell}^{\R,\mathrm{fix}}
 =\{u^\uparrow:u\in\mathfrak M_\ell^{\R}\}.
\]
The normalized full bubble and the raw lift of the chosen reduced bubble
differ only by a nonzero scalar, which is absorbed by the real amplitude in
these cones.
The raw lift is a similarity whose squared energy factor is \(c_\ell\).

\subsection{Auxiliary-center distance comparison}

The small-center regime uses the following energy-space differentiability
fact.  It is stated with all real auxiliary directions because coordinatewise
lower bounds would not rule out cancellation.

\begin{lemma}[Auxiliary translation tangent]\label{lem:auxiliary-tangent}
Let
\[
 U_s=\tau_{(0,s,0)}\widetilde B,
 \qquad s\in\C^\ell.
\]
There is an injective real-linear map
\(\mathsf T_{\mathrm{aux}}:\C^\ell\to
\dot H^1(\Heis^{n+\ell};\R)\) such that, uniformly for
\(s/|s|\in S^{2\ell-1}\),
\[
 U_s-\widetilde B
 =\mathsf T_{\mathrm{aux}}(s)+\mathcal R(s),
 \qquad
 \|\mathcal R(s)\|_{\dot H^1}=o(|s|).
\]
Moreover,
\[
 P_{U(\ell)}\mathsf T_{\mathrm{aux}}(s)=0,
 \qquad
 \|\mathsf T_{\mathrm{aux}}(s)\|_{\dot H^1}\geq c_*|s|
\]
for a constant \(c_*>0\) depending only on \((n,\ell)\).
\end{lemma}

\begin{proof}
Write \(Z=(z,w)\), \(w_\alpha=a_\alpha+ib_\alpha\), and
\[
 \widetilde B=C D^{-(n+\ell)/2},\qquad
 D=(1+|Z|^2)^2+t^2,
\]
where \(C>0\).  With the translation convention fixed above, differentiation
in the real and imaginary parts of \(s_\alpha\) gives the right-invariant
auxiliary fields
\[
 \widehat X_\alpha^w=\partial_{a_\alpha}-2b_\alpha\partial_t,
 \qquad
 \widehat Y_\alpha^w=\partial_{b_\alpha}+2a_\alpha\partial_t
\]
and
\[
 \mathsf T_{\mathrm{aux}}(s)
 =-\sum_{\alpha=1}^{\ell}
 \left((\operatorname{Re}s_\alpha)\widehat X_\alpha^w\widetilde B
 +(\operatorname{Im}s_\alpha)\widehat Y_\alpha^w\widetilde B\right).
\]
Equivalently, with the Hermitian product linear in its first argument,
\[
 \mathsf T_{\mathrm{aux}}(s)(Z,t)
 =2(n+\ell)C D^{-(n+\ell)/2-1}
 \left[(1+|Z|^2)\operatorname{Re}\langle w,s\rangle
       -t\operatorname{Im}\langle w,s\rangle\right].
\]
Put \(R_*=D^{1/4}\).  Direct differentiation shows
\[
 |\nabla_H(\widehat X_\alpha^w\widetilde B)|
 +|\nabla_H(\widehat Y_\alpha^w\widetilde B)|
 \leq C_{n,\ell}R_*^{-Q_\ell}.
\]
Since the volume of \(\{R_*\leq R\}\) is \(O(R^{Q_\ell})\), a dyadic-shell
sum gives finite energy.  To verify membership in the homogeneous
completion, choose \(\chi\in C_c^\infty([0,\infty))\) equal to one on
\([0,1]\) and supported in \([0,2]\), and put
\(\chi_M=\chi(R_*/M)\), \(M\geq1\).  For either of the
displayed real tangent fields \(E\), direct differentiation also gives
\[
 |E|=O(R_*^{-(Q_\ell-1)}),\qquad
 |\nabla_HE|=O(R_*^{-Q_\ell}),
 \qquad |\nabla_HR_*|=O(1).
\]
Thus \(\chi_ME\in C_c^\infty\), the gradient tail tends to zero, and the
cutoff error satisfies
\[
 M^{-2}\int_{\{M\leq R_*\leq2M\}}|E|^2
 =O(M^{-Q_\ell})\longrightarrow0.
\]
Consequently \(\chi_ME\to E\) in homogeneous energy.  The same decay gives
\(E\in L^{q_\ell}\), so a pointwise nonzero tangent is nonzero in the
energy completion.

Translations are strongly continuous on \(\dot H^1\): this follows first
for compactly supported smooth functions and then by density and unitarity.
The path \(r\mapsto(0,rs,0)\) is a one-parameter subgroup, so the
Hilbert-valued fundamental theorem of calculus gives
\[
 U_s-\widetilde B
 =\int_0^1\tau_{(0,rs,0)}\mathsf T_{\mathrm{aux}}(s)\,dr.
\]
Subtracting \(\mathsf T_{\mathrm{aux}}(s)\), let
\(E_{\alpha,1}=-\widehat X_\alpha^w\widetilde B\) and
\(E_{\alpha,2}=-\widehat Y_\alpha^w\widetilde B\).  The finite real basis
estimate
\[
 \frac{\|U_s-\widetilde B-\mathsf T_{\mathrm{aux}}(s)\|_{\dot H^1}}{|s|}
 \leq \sqrt{2\ell}\max_{\alpha,\nu}
 \sup_{0\leq r\leq1}
 \|\bigl(\tau_{(0,rs,0)}-I\bigr)E_{\alpha,\nu}\|_{\dot H^1}
\]
tends to zero by strong continuity, uniformly in the direction of \(s\).

Auxiliary rotations satisfy
\(k\cdot\mathsf T_{\mathrm{aux}}(s)
=\mathsf T_{\mathrm{aux}}(ks)\).  Hence
\[
 P_{U(\ell)}\mathsf T_{\mathrm{aux}}(s)
 =\mathsf T_{\mathrm{aux}}\!\left(\int_{U(\ell)}ks\,dk\right)=0.
\]
The explicit tangent formula is nonzero for every \(s\ne0\); for instance,
its value at \((z,w,t)=(0,s,0)\) is strictly positive.  Thus the real-linear
map is injective.  Continuity and compactness of the real unit sphere give
\(c_*:=\min_{|s|=1}\|\mathsf T_{\mathrm{aux}}(s)\|_{\dot H^1}>0\), which
proves the final estimate.
\end{proof}

The key geometric step is the following uniform estimate.

\begin{proposition}[Fixed/full distance comparison]
\label{prop:distance-comparison}
There exists \(C_{n,\ell}\geq1\) depending only on \((n,\ell)\) such
that for every \(U(\ell)\)-fixed
\(F\in\dot H^1(\Heis^{n+\ell};\R)\),
\begin{equation}\label{eq:dist-comparison}
 \operatorname{dist}_{\dot H^1}
 \!\bigl(F,\mathcal M_{n+\ell}^{\R,\mathrm{fix}}\bigr)
 \leq C_{n,\ell}\,
 \operatorname{dist}_{\dot H^1}
 \!\bigl(F,\mathcal M_{n+\ell}^{\R,\mathrm{full}}\bigr).
\end{equation}
\end{proposition}

\begin{proof}
We first prove the centered estimate
\begin{equation}\label{eq:aux-center-bound}
 \|U_s-\widetilde B\|_{\dot H^1}
 \leq C_{n,\ell}^{\mathrm{bub}}
 \|(I-P_{U(\ell)})U_s\|_{\dot H^1},
 \qquad s\in\C^\ell.
\end{equation}

\emph{Small auxiliary center.}
Lemma~\ref{lem:auxiliary-tangent} gives
\[
 U_s-\widetilde B
 =\mathsf T_{\mathrm{aux}}(s)+\mathcal R(s),
 \qquad
 (I-P_{U(\ell)})U_s
 =\mathsf T_{\mathrm{aux}}(s)
 +(I-P_{U(\ell)})\mathcal R(s).
\]
The projection is contractive, the remainder is \(o(|s|)\) uniformly in
direction, and
\(\|\mathsf T_{\mathrm{aux}}(s)\|\geq c_*|s|\).  Both displayed norms are
therefore \(\|\mathsf T_{\mathrm{aux}}(s)\|+o(|s|)\), uniformly as
\(s\to0\).  Their ratio tends to one, which proves
\eqref{eq:aux-center-bound} on a sufficiently small ball.

\emph{Compact auxiliary annuli.}
Strong continuity of translations makes both norms in
\eqref{eq:aux-center-bound} continuous in \(s\).  If
\((I-P_{U(\ell)})U_s=0\) in the homogeneous completion, the critical Sobolev
embedding gives \(U_s=P_{U(\ell)}U_s\) almost everywhere and hence everywhere
for the smooth representatives.  The unique maximum of \(U_s\) would then be
fixed by every auxiliary rotation, forcing \(s=0\).  The same unique-maximum
argument shows that \(U_s\ne\widetilde B\) when \(s\ne0\).  Thus the ratio
in \eqref{eq:aux-center-bound} is bounded on every compact annulus separated
from zero.

\emph{Large auxiliary center.}
Let \(|s_j|\to\infty\), and pass to a subsequence such that
\(s_j/|s_j|\to\omega\).  Orthogonality of the Haar projection gives
\begin{equation}\label{eq:haar-correlation}
 \|P_{U(\ell)}U_{s_j}\|_{\dot H^1}^2
 =\int_{U(\ell)}
 \langle U_{ks_j},U_{s_j}\rangle_{\dot H^1}\,dk.
\end{equation}
The stabilizer of \(\omega\) is conjugate to \(U(\ell-1)\) when
\(\ell>1\), and is the identity subgroup when \(\ell=1\); in either case it
has Haar measure zero.  For every \(k\) outside this stabilizer,
\(|(k-I)s_j|\to\infty\).  The relative center is
\[
 (0,s_j,0)^{-1}(0,ks_j,0)
 =\bigl(0,(k-I)s_j,
 -2\operatorname{Im}\langle s_j,ks_j\rangle\bigr).
\]
Translations whose horizontal component escapes converge weakly to zero in
\(\dot H^1(\Heis^{n+\ell})\): for compactly supported smooth functions this
follows from separation of supports, and the general statement follows by
density and unitarity.  By unitarity, the integrand in
\eqref{eq:haar-correlation} is the correlation of the fixed bubble with its
translate by the displayed relative center.  It therefore tends to zero for
almost every \(k\), and its absolute value is bounded by
\(\|\widetilde B\|_{\dot H^1}^2\).  Dominated convergence yields
\[
 \|P_{U(\ell)}U_{s_j}\|_{\dot H^1}\longrightarrow0,
 \qquad
 \|(I-P_{U(\ell)})U_{s_j}\|_{\dot H^1}\longrightarrow
 \|\widetilde B\|_{\dot H^1}.
\]
Since
\(\|U_{s_j}-\widetilde B\|_{\dot H^1}\leq
2\|\widetilde B\|_{\dot H^1}\), a sequential contradiction proves a
uniform ratio bound outside a sufficiently large ball.  The three regimes
establish \eqref{eq:aux-center-bound} for every \(s\).

Now take an arbitrary
\(U=A\mathcal T^{\mathrm{full}}_{h,\xi}\widetilde B\) in the full cone.
If \(U=0\), set \(U_0=0\).  Otherwise write
\[
 \xi=(z_0,w_0,t_0),\qquad
 \bar\eta=(z_0,0,t_0),\qquad
 s=\frac{w_0}{h},\qquad \zeta=(0,s,0).
\]
The two complex blocks are Hermitian-orthogonal, so the group law gives the
exact factorization \(\xi=\bar\eta\,\delta_h\zeta\), with no central
correction.  Since
\(\tau_{ab}=\tau_a\tau_b\) and
\(\tau_{\delta_h\zeta}S_h=S_h\tau_\zeta\),
\[
 U=A(\tau_{\bar\eta}S_h)U_s.
\]
The operator \(\tau_{\bar\eta}S_h\) is unitary and commutes with auxiliary
rotations.  Hence
\[
 U_0=A(\tau_{\bar\eta}S_h)\widetilde B
 \in\mathcal M_{n+\ell}^{\R,\mathrm{fix}}
\]
and \eqref{eq:aux-center-bound} gives
\[
 \|U-U_0\|_{\dot H^1}
 \leq C_{n,\ell}^{\mathrm{bub}}
 \|(I-P_{U(\ell)})U\|_{\dot H^1}.
\]
This estimate is uniform in the real amplitude; at \(A=0\) both sides
vanish.

If \(F\) is \(U(\ell)\)-fixed, then
\[
 \|(I-P_{U(\ell)})U\|_{\dot H^1}
 =\|(I-P_{U(\ell)})(U-F)\|_{\dot H^1}
 \leq\|U-F\|_{\dot H^1}.
\]
The triangle inequality consequently gives
\[
 \operatorname{dist}_{\dot H^1}(F,\mathcal M_{n+\ell}^{\R,\mathrm{fix}})
 \leq(1+C_{n,\ell}^{\mathrm{bub}})\|F-U\|_{\dot H^1}.
\]
Taking the infimum over the full cone proves \eqref{eq:dist-comparison} with
\(C_{n,\ell}=1+C_{n,\ell}^{\mathrm{bub}}\).
\end{proof}

\subsection{Polar descent and the reduced stability theorem}

\begin{proof}[Proof of Theorem~\ref{thm:intro-stability}]
Let \(v\in\dot S^1_\ell(\Ucal;\R)\).  Set \(F=v^\uparrow\).
By Lemma~\ref{lem:polar-lift},
\[
 \|F\|_{\dot H^1}^2=c_\ell E_\ell(v),
 \qquad
 \|F\|_{q_\ell}^{q_\ell}=c_\ell\|v\|_{q_\ell}^{q_\ell}.
\]
Hence \(\|F\|_{q_\ell}^2=c_\ell^{2/q_\ell}\|v\|_{q_\ell}^2\) and the
relation between the sharp constants is
\[
 S_\ell
 =S_{n+\ell}^{\mathrm{FS}}\,c_\ell^{2/q_\ell-1}.
\]
The full-space deficit of \(F\) equals a polar factor times the reduced
deficit:
\[
 \|F\|_{\dot H^1}^2-S_{n+\ell}^{\mathrm{FS}}\|F\|_{q_\ell}^2
 =c_\ell\bigl(E_\ell(v)-S_\ell\|v\|_{q_\ell}^2\bigr).
\]
The raw polar lift scales squared distances by the same factor:
\[
 \operatorname{dist}_{\dot H^1}
 (F,\mathcal M_{n+\ell}^{\R,\mathrm{fix}})^2
 =c_\ell\,
 \operatorname{dist}_{\dot S^1_\ell}
 (v,\mathfrak M_\ell^{\R})^2.
\]
Since \(F\) is \(U(\ell)\)-fixed,
Proposition~\ref{prop:distance-comparison} gives
\[
 \operatorname{dist}_{\dot H^1}
 (F,\mathcal M_{n+\ell}^{\R,\mathrm{fix}})
 \leq C_{n,\ell}\,
 \operatorname{dist}_{\dot H^1}
 (F,\mathcal M_{n+\ell}^{\R,\mathrm{full}}).
\]
Applying the full-space inequality \eqref{eq:full-stability} and combining,
\begin{align*}
 c_\ell\bigl(E_\ell(v)-S_\ell\|v\|_{q_\ell}^2\bigr)
 &\geq\alpha_{Q_\ell}\,
 \operatorname{dist}_{\dot H^1}
 (F,\mathcal M_{n+\ell}^{\R,\mathrm{full}})^2\\
 &\geq\frac{\alpha_{Q_\ell}}{C_{n,\ell}^2}\,
 \operatorname{dist}_{\dot H^1}
 (F,\mathcal M_{n+\ell}^{\R,\mathrm{fix}})^2\\
 &=\frac{\alpha_{Q_\ell}}{C_{n,\ell}^2}\,c_\ell\,
 \operatorname{dist}_{\dot S^1_\ell}(v,\mathfrak M_\ell^{\R})^2.
\end{align*}
Dividing by \(c_\ell>0\) proves the theorem with
\[
 \kappa_{n,\ell}
 =\frac{\alpha_{Q_\ell}}{C_{n,\ell}^2}>0.
 \qedhere
\]
\end{proof}

\begin{remark}\label{rem:stability-scope}
Theorem~\ref{thm:intro-stability} is stated for real-valued functions.  The
constant depends on \((n,\ell)\) through the full-space constant
\(\alpha_{Q_\ell}\) of \cite{Loiudice2005ImprovedSobolev} and the distance
comparison \(C_{n,\ell}\) of
Proposition~\ref{prop:distance-comparison}.  Section~\ref{sec:nondegeneracy}
computes the exact local Hessian coefficient, but this does not determine
the optimal global stability constant.
\end{remark}

\section{Fixed-sector nondegeneracy and local spectral geometry}
\label{sec:nondegeneracy}

In this section we identify the kernel of the linearization at a reduced
bubble and compute the first positive normal spectral ratio.  The
surjectivity in Lemma~\ref{lem:polar-lift} is essential here: the argument
takes the fixed part of the full Heisenberg kernel in the completed
homogeneous space, rather than merely restricting a calculation on a
smooth core.

Put \(N=n+\ell\), \(q=q_\ell=2+2/N\), and \(S=S_\ell\).  We use the
polarized energy
\[
 E_\ell(u,v)
 =\int_{\Ucal}\Gamma_\ell(u,v)\dd\mu_\ell.
\]
Let \(B=B_0\) be the centered extremal with
\(\|B\|_q=1\), and set
\[
 \widehat B=S^{1/(q-2)}B=S^{N/2}B.
\]
Thus
\[
 \Gell_\ell B=SB^{q-1},
 \qquad
 \Gell_\ell\widehat B=\widehat B^{q-1}.
\]
Set \(\widetilde{\widehat B}=\widehat B^\uparrow\) on \(\Heis^N\).
The normalized linearized form is
\begin{equation}\label{eq:linearized-form}
 \mathcal Q_B(\phi,\psi)
 =E_\ell(\phi,\psi)
 -(q-1)S\int_{\Ucal}B^{q-2}\phi\psi\dd\mu_\ell.
\end{equation}
It is also the coefficient-one linearized form at \(\widehat B\), since
\(\widehat B^{q-2}=SB^{q-2}\).  We denote its weak operator by \(L_B\).

For \(z_j=x_j+iy_j\), let
\[
 \widehat X_j=\partial_{x_j}-2y_j\partial_t,
 \qquad
 \widehat Y_j=\partial_{y_j}+2x_j\partial_t
\]
be the right-invariant generators of the reduced translations, and let
\[
 \Lambda_\ell B
 =\left.\frac{\dd}{\dd s}\right|_{s=0}
 \mathcal T_{e^s,0}B
\]
be the dilation generator.

\begin{lemma}[Normalization and completeness of the spherical spectrum]
\label{lem:spherical-spectrum-normalization}
The energy-normalized CR Cayley transform is an isomorphism from
\(\dot S^1(\Heis^N)\) onto the form domain of \(\mathcal A_2\) on
\(\mathbb S^{2N+1}\).  The bispherical harmonic spaces
\(\mathcal H_{j,k}\) give a complete orthogonal decomposition of that
domain.  For the full generalized eigenvalue problem
\[
 -\Delta_{\Heis^N}\Phi
 =\Lambda\widetilde{\widehat B}^{\,q-2}\Phi,
\]
the corresponding eigenvalue on \(\mathcal H_{j,k}\) is
\begin{equation}\label{eq:normalized-spherical-spectrum}
 \Lambda_{j,k}
 =\frac{(N+2j)(N+2k)}{N^2},
 \qquad j,k\in\mathbb N_0.
\end{equation}
Under the completed polar lift, the reduced generalized spectrum is the
restriction of this decomposition to its \(U(\ell)\)-fixed part.
\end{lemma}

\begin{proof}
The smooth conformal covariance formula for the CR Cayley transform gives
an energy isometry between the Heisenberg group and the punctured sphere.
Since the homogeneous dimension is \(Q=2N+2>2\), the omitted point has
zero \(S^1\)-capacity: a cutoff on a CR normal ball about the omitted pole
of radius
\(\varepsilon\) has energy \(O(\varepsilon^{Q-2})\).  Functions vanishing
near the pole are therefore a form core, and the Cayley isometry extends
onto the full form domain.  The transformation intertwines the
\(U(\ell)\)-actions.  Together with the surjectivity in
Lemma~\ref{lem:polar-lift}, this proves the final assertion once the full
spectrum has been computed.

For \(0<d<2\), the intertwining operator \(\mathcal A_d\) in the
normalization of Frank--Lieb \cite{FrankLieb2012SharpHeisenberg} has an
explicit Gamma multiplier on \(\mathcal H_{j,k}\).  Taking its regular
endpoint as \(d\uparrow2\), and using Folland's second-order spherical
spectral theory \cite{Folland1972TangentialCR}, gives
\[
 a_{j,k}
 =\frac{
 \Gamma\!\left(\frac{Q+2}{4}+j\right)
 \Gamma\!\left(\frac{Q+2}{4}+k\right)}{
 \Gamma\!\left(\frac{Q-2}{4}+j\right)
 \Gamma\!\left(\frac{Q-2}{4}+k\right)}
 =\left(j+\frac N2\right)\left(k+\frac N2\right).
\]
The algebraic sum of the spaces \(\mathcal H_{j,k}\) is dense in
\(L^2(\mathbb S^{2N+1})\).  Since \(\mathcal A_2\) is positive and
diagonal on these spaces, its spectral truncations converge also in the
form norm; hence the same sum is a form core.

The Cayley image of the coefficient-one bubble is a positive constant,
say \(b\).  The normalization of the sub-Laplacian must be kept explicit
here.  Frank--Lieb use
\[
 \mathcal L_{\Heis^N}=-\frac14\Delta_{\Heis^N},
\]
whereas our full-space equation contains \(-\Delta_{\Heis^N}\).  Thus the
generalized eigenvalue problem becomes
\[
 4\mathcal A_2\psi=\Lambda b^{q-2}\psi
 \quad\text{on }\mathbb S^{2N+1}.
\]
The constant-mode equation is
\[
 4\mathcal A_2b=b^{q-1},
 \qquad b^{q-2}=4a_{0,0}=N^2.
\]
The common factor \(4\) therefore cancels:
\(\Lambda_{j,k}=4a_{j,k}/b^{q-2}=a_{j,k}/a_{0,0}\), which is
\eqref{eq:normalized-spherical-spectrum}.
\end{proof}

\begin{theorem}[Fixed-sector nondegeneracy]
\label{thm:fixed-nondegeneracy}
The kernel of the reduced linearized operator is
\begin{equation}\label{eq:reduced-kernel}
 \ker L_B
 =\operatorname{span}_{\R}
 \left\{
 \widehat X_jB,\widehat Y_jB\ (1\leq j\leq n),
 \partial_tB,\Lambda_\ell B
 \right\}.
\end{equation}
In particular,
\[
 \dim\ker L_B=2n+2,
 \qquad
 \ker L_B=T_B\mathcal M_{1,\ell}.
\]
\end{theorem}

\begin{proof}
Write points of \(\Heis^N\) as \((z,w,t)\), with
\(z\in\C^n\) and \(w\in\C^\ell\).  The polar identities
intertwine the two weak linearized equations:
\[
 \phi\in\ker L_B
 \quad\Longleftrightarrow\quad
 \phi^\uparrow\in
 \ker\widetilde L\cap
 \dot S^1(\Heis^N)^{U(\ell)},
\]
where
\[
 \widetilde L
 =-\Delta_{\Heis^N}
 -(q-1)\widetilde{\widehat B}^{\,q-2}.
\]
Indeed, one implication follows by testing against lifted functions; the
other follows from the surjectivity of the completed lift in
Lemma~\ref{lem:polar-lift}.  The averaged-test argument used in
Lemma~\ref{lem:weak-lift} then upgrades the identity to arbitrary
full-space tests.

By Lemma~\ref{lem:spherical-spectrum-normalization}, the full linearized
kernel is carried by the spherical levels satisfying
\[
 \Lambda_{j,k}=q-1=\frac{N+2}{N}.
\]
Equivalently,
\[
 (N+2j)(N+2k)=N(N+2),
 \qquad
 N(j+k-1)+2jk=0.
\]
For \(j,k\in\mathbb N_0\), the only solutions are
\((j,k)=(1,0)\) and \((0,1)\).  Completeness of the spherical
decomposition therefore gives
\[
 \ker\widetilde L
 \simeq
 (\mathcal H_{1,0}\oplus\mathcal H_{0,1})_{\R},
 \qquad
 \dim_{\R}\ker\widetilde L=2N+2.
\]

The translation--dilation action on the centered bubble is free, so the
differential of its \((2N+2)\)-dimensional orbit map is injective.
Differentiating the orbit of \(\widetilde{\widehat B}\) therefore produces
the linearly independent elements
\[
 \widehat X_\alpha\widetilde{\widehat B},
 \quad
 \widehat Y_\alpha\widetilde{\widehat B}\ (1\leq\alpha\leq N),
 \quad
 \partial_t\widetilde{\widehat B},
 \quad
 \Lambda\widetilde{\widehat B}.
\]
They belong to the kernel by differentiation of the equation, and their
number equals the dimension just computed.  Hence they span the full
kernel.  This characterization agrees with Lemma~5 of
Malchiodi--Uguzzoni \cite{MalchiodiUguzzoni2002Webster}; see also the
direct full-space proof of Qiang--Tang--Zhang
\cite{QiangTangZhang2026Nondegeneracy}.

Separate this space into
\[
 \mathcal V
 =\operatorname{span}_{\R}
 \left\{
 \widehat X_j\widetilde{\widehat B},
 \widehat Y_j\widetilde{\widehat B}\ (1\leq j\leq n),
 \partial_t\widetilde{\widehat B},
 \Lambda\widetilde{\widehat B}
 \right\}
\]
and the \(2\ell\)-dimensional auxiliary translation space
\[
 \mathcal W
 =\operatorname{span}_{\R}
 \left\{
 \widehat X^w_\alpha\widetilde{\widehat B},
 \widehat Y^w_\alpha\widetilde{\widehat B}
 \ (1\leq\alpha\leq\ell)
 \right\}.
\]
Every vector in \(\mathcal V\) is fixed by \(U(\ell)\).  On the other
hand, the equivariant injection
\(\mathsf T_{\mathrm{aux}}:\C^\ell\to\mathcal W\) of
Lemma~\ref{lem:auxiliary-tangent} identifies the action on
\(\mathcal W\) with the standard representation of \(U(\ell)\).
That representation has no nonzero fixed vector.  Hence
\[
 \ker\widetilde L\cap
 \dot S^1(\Heis^N)^{U(\ell)}
 =\mathcal V.
\]
Descending the generators in \(\mathcal V\) gives
\eqref{eq:reduced-kernel}.  They are precisely the derivatives of the
reduced translation--dilation orbit at \(B\), which proves the final
assertion.
\end{proof}

Define the normal gauge
\begin{equation}\label{eq:normal-gauge}
 Y_\ell
 =\left\{
 \phi\in\dot S^1_\ell(\Ucal;\R):
 \int_{\Ucal}B^{q-1}\phi\dd\mu_\ell=0,\quad
 E_\ell(\phi,Z)=0\ \text{for all }Z\in\ker L_B
 \right\}.
\end{equation}
Since
\[
 E_\ell(B,\phi)
 =S\int_{\Ucal}B^{q-1}\phi\dd\mu_\ell,
\]
this is the energy-orthogonal complement of
\(\R B\oplus\ker L_B\).

\begin{theorem}[Exact normal spectral ratio]
\label{thm:normal-spectral-gap}
For every \(\phi\in Y_\ell\),
\begin{equation}\label{eq:exact-normal-gap}
 \mathcal Q_B(\phi,\phi)
 \geq\frac{2}{n+\ell+4}E_\ell(\phi,\phi).
\end{equation}
The constant is optimal.
\end{theorem}

\begin{proof}
By Lemma~\ref{lem:spherical-spectrum-normalization}, the generalized
eigenvalues are given by \eqref{eq:normalized-spherical-spectrum}.
The first three distinct levels are
\[
 1,\qquad \frac{N+2}{N}=q-1,\qquad \frac{N+4}{N}.
\]
The first is the amplitude direction, and the \(U(\ell)\)-fixed part of
the second is the orbit kernel identified in
Theorem~\ref{thm:fixed-nondegeneracy}.  For a generalized eigenfunction
with eigenvalue \(\Lambda\), the quotient of the linearized form by the
energy is
\[
 1-\frac{q-1}{\Lambda}.
\]
Removing the amplitude and kernel levels therefore gives
\[
 1-\frac{(N+2)/N}{(N+4)/N}
 =\frac{2}{N+4}.
\]

The next level survives the invariant restriction.  Indeed, if
\[
 \zeta=(\zeta_z,\zeta_w,\zeta_{N+1})
 \in\C^n\times\C^\ell\times\C
\]
are the sphere coordinates, then
\(\operatorname{Re}\zeta_{N+1}^2\) and
\(\operatorname{Im}\zeta_{N+1}^2\) belong to
\(\mathcal H_{2,0}\oplus\mathcal H_{0,2}\) and are fixed by \(U(\ell)\).
This proves optimality.
\end{proof}

Let \(A_B\) be the bounded self-adjoint Riesz operator defined by
\[
 E_\ell(A_B\phi,\psi)=\mathcal Q_B(\phi,\psi).
\]
The spectral calculation also gives the following inverse estimate.

\begin{corollary}[Normal invertibility]\label{cor:normal-inverse}
The restriction of \(A_B\) to
\((\ker L_B)^{\perp_E}\) is invertible, and
\begin{equation}\label{eq:normal-inverse}
 \|A_B\phi\|_{\dot S^1_\ell}
 \geq\frac{2}{n+\ell+4}\|\phi\|_{\dot S^1_\ell}
 \qquad
 \bigl(\phi\perp_E\ker L_B\bigr).
\end{equation}
\end{corollary}

\begin{proof}
On \(Y_\ell\) this follows from
Theorem~\ref{thm:normal-spectral-gap}.  The remaining orthogonal direction
is \(B\), and
\[
 A_BB=-(q-2)B=-\frac{2}{N}B.
\]
Its absolute spectral value is larger than \(2/(N+4)\).
\end{proof}

\begin{proposition}[Exact local stability coefficient]
\label{prop:local-stability-coefficient}
Let \(\mathfrak M_\ell^\R\) be the real optimizer cone.  Then
\begin{equation}\label{eq:local-stability-coefficient}
 \lim_{\delta\downarrow0}
 \inf_{\substack{0<\|r\|_{\dot S^1_\ell}<\delta\\ r\in Y_\ell}}
 \frac{
 E_\ell(B+r)-S\|B+r\|_q^2
 }{
 \operatorname{dist}_{\dot S^1_\ell}
 (B+r,\mathfrak M_\ell^\R)^2
 }
 =\frac{2}{N+4}.
\end{equation}
\end{proposition}

\begin{proof}
In a neighborhood of \(B\), and hence away from the vertex of the cone,
the local cone chart
\[
 (a,\xi)\longmapsto a\,\mathcal T_{\exp\xi}B
\]
has tangent space \(\R B\oplus\ker L_B\) at \((1,0)\).  The implicit
function theorem applied to the energy-orthogonality equations gives a
smooth tubular projection onto the nonzero cone.  Consequently, for
\(r\in Y_\ell\),
\[
 \operatorname{dist}_{\dot S^1_\ell}
 (B+r,\mathfrak M_\ell^\R)^2
 =E_\ell(r)+o(E_\ell(r)).
\]
The deficit has the expansion
\[
 E_\ell(B+r)-S\|B+r\|_q^2
 =\mathcal Q_B(r,r)+o(E_\ell(r)).
\]
The lower bound follows from \eqref{eq:exact-normal-gap}, and equality is
approached along a nonzero invariant vector in
\(\mathcal H_{2,0}\oplus\mathcal H_{0,2}\).
\end{proof}

\begin{remark}\label{rem:local-versus-global}
The value \(2/(N+4)\) is local and does not identify the optimal global
constant in Theorem~\ref{thm:intro-stability}.  Tang--Zhang--Zhang
\cite{TangZhangZhang2024Minimizer} proved that the optimal full-Heisenberg
stability quotient has a nontrivial minimizer.  This full-space attainment
result does not by itself identify the fixed-sector constant: a strict test
can descend only when the entire test, and not merely its leading
two-bubble configuration, is invariant.  Thus local nondegeneracy alone
does not determine the sharp global Bianchi--Egnell constant in the fixed
sector.  Section~\ref{sec:optimal-stability} constructs genuinely
\(U(\ell)\)-fixed strict tests and proves fixed-sector attainment by a
separate compactness argument.
\end{remark}

\begin{lemma}[Local normal slice]\label{lem:normal-slice}
There are \(\varepsilon>0\) and \(C>0\) such that every real
\(u\) satisfying
\[
 \|u\|_q=1,
 \qquad
 \operatorname{dist}_{\dot S^1_\ell}(u,\mathcal M_{1,\ell})
 <\varepsilon
\]
admits a sign \(\sigma\) and a reduced frame \(g\) for which
\[
 \mathcal T_g^{-1}(\sigma u)=B+r,
 \qquad
 E_\ell(r,Z)=0\quad(Z\in\ker L_B),
\]
and
\[
 \|r\|_{\dot S^1_\ell}
 \leq C\operatorname{dist}_{\dot S^1_\ell}
 (u,\mathcal M_{1,\ell}).
\]
The frame is locally unique after the sign and the nearby orbit component
have been fixed.
\end{lemma}

\begin{proof}
Choose a basis \(Z_1,\ldots,Z_{2n+2}\) of \(\ker L_B\), and choose local
coordinates \(\xi=(\xi_1,\ldots,\xi_{2n+2})\) on the reduced frame group so
that
\[
 \left.\partial_{\xi_b}\right|_{\xi=0}
 \mathcal T_{\exp\xi}B=Z_b.
\]
For \(\widetilde u\) near \(B\), consider the smooth map
\[
 \Phi_a(\xi,\widetilde u)
 =E_\ell\!\left(
 \widetilde u-\mathcal T_{\exp\xi}B,
 \mathcal T_{\exp\xi}Z_a
 \right),
 \qquad 1\leq a\leq2n+2.
\]
Only the finite-dimensional orbit maps are differentiated here.  At
\((\xi,\widetilde u)=(0,B)\),
\[
 \partial_{\xi_b}\Phi_a(0,B)=-E_\ell(Z_b,Z_a),
\]
because the second differentiated factor is paired with \(B-B=0\).
The Gram matrix is invertible, so the implicit function theorem supplies a
unique small \(\xi=\xi(\widetilde u)\) for which
\(\Phi(\xi,\widetilde u)=0\).  By invariance of the energy form, this is
equivalent to
\[
 E_\ell\!\left(
 \mathcal T_{\exp(-\xi)}\widetilde u-B,Z_a
 \right)=0
 \qquad(1\leq a\leq2n+2).
\]

The unit orbit is closed: an escaping frame is weakly null while every
orbit element has the same nonzero energy norm.  Thus, if the orbit
distance is zero, the conclusion is immediate.  Otherwise, for a general
\(u\) satisfying the hypothesis, choose a signed orbit point
\(\sigma\mathcal T_{g_0}B\) within twice its orbit distance, and apply the preceding construction to
\(\widetilde u=\mathcal T_{g_0}^{-1}(\sigma u)\).  The local Lipschitz
bound furnished by the implicit function theorem, followed by energy
invariance, yields the asserted estimate.  It also gives local uniqueness
once the sign and the nearby orbit component are fixed.
\end{proof}

\section{The optimal fixed-sector stability quotient}
\label{sec:optimal-stability}

The stability theorem in Section~\ref{sec:stability} supplies a positive
constant but does not identify the optimal quotient.  In this section we
prove that the optimal value is attained.  Two strict inequalities are
needed: one excludes convergence to the optimizer cone, and the other
excludes asymptotic splitting into two bubbles.
The local strictness test is modeled on K\"onig's Euclidean construction
\cite{Koenig2023BianchiEgnell}, while the two-bubble threshold and
attainment strategy follow the compactness scheme of
\cite{Koenig2025SobolevMinimizer}.  The correlation balance and the
admissible frames must here be established inside the
\(U(\ell)\)-fixed sector.

Put
\[
 N=n+\ell,\qquad q=q_\ell=2+\frac2N,\qquad S=S_\ell,
\]
and write
\[
 \mathscr D_\ell(u)=E_\ell(u)-S\|u\|_q^2,\qquad
 d_\ell(u)=\operatorname{dist}_{\dot S^1_\ell}
 (u,\mathfrak M_\ell^{\mathbb R}).
\]
The optimal real fixed-sector constant is
\begin{equation}\label{eq:optimal-fixed-constant}
 \kappa_{n,\ell}^{\mathrm{opt}}
 =\inf_{u\in\dot S^1_\ell(\Ucal;\mathbb R)
 \setminus\mathfrak M_\ell^{\mathbb R}}
 \frac{\mathscr D_\ell(u)}{d_\ell(u)^2}.
\end{equation}
Both terms in the quotient are homogeneous of degree two, so minimizing
sequences may be normalized by \(\|u\|_q=1\).
Theorem~\ref{thm:intro-stability} gives
\(\kappa_{n,\ell}^{\mathrm{opt}}>0\).

\begin{theorem}[Attainment of the optimal stability quotient]
\label{thm:optimal-fixed-attainment}
For all integers \(n,\ell\geq1\),
\begin{equation}\label{eq:two-strict-stability-thresholds}
 0<\kappa_{n,\ell}^{\mathrm{opt}}
 <\min\left\{
 \frac2{n+\ell+4},\,
 2-2^{(n+\ell)/(n+\ell+1)}
 \right\}.
\end{equation}
There is a real
\[
 u_*\in\dot S^1_\ell(\Ucal)
 \setminus\mathfrak M_\ell^{\mathbb R}
\]
which attains \eqref{eq:optimal-fixed-constant}.  Moreover, if
\((u_k)\) is any minimizing sequence with \(\|u_k\|_q=1\), then, after
passing to a subsequence, there are reduced frames \(g_k\) such that
\(\mathcal T_{g_k}^{-1}u_k\) converges strongly in
\(\dot S^1_\ell(\Ucal)\).
\end{theorem}

\subsection{A strict two-bubble test}

We first work on the full group \(\Heis^N\).  Let
\(\widetilde E\) and \(\widetilde S\) denote its energy and sharp
Folland--Stein constant, and let \(\widetilde B\) be a centered bubble
with \(\|\widetilde B\|_q=1\).  For \(0<h<1\), let
\(\widetilde B_h\) be its centered dilation at scale \(h\), and set
\[
 F_h=\widetilde B+\widetilde B_h,\qquad
 I_h=\widetilde E(\widetilde B,\widetilde B_h)>0.
\]
Both summands are invariant under \(U(N)\).  In particular, \(F_h\) is
\(U(\ell)\)-fixed and belongs to the range of the completed polar lift.

\begin{lemma}[Two-bubble expansions]
\label{lem:fixed-two-bubble-expansions}
As \(h\downarrow0\),
\begin{align}
 I_h&=\gamma_Nh^N+o(h^N),\qquad \gamma_N>0,
 \label{eq:fixed-two-bubble-interaction}\\
 \int_{\Heis^N}F_h^q
 &=2+\frac{2q}{\widetilde S}I_h+O(h^{N+1}),
 \label{eq:fixed-two-bubble-critical-mass}\\
 \operatorname{dist}_{\dot H^1}
 (F_h,\widetilde{\mathfrak M}^{\mathrm{full}})^2
 &=\widetilde S+O(I_h^2).
 \label{eq:fixed-two-bubble-distance}
\end{align}
\end{lemma}

\begin{proof}
Up to the common normalization of the bubbles, write
\[
 a(Z,t)=\bigl((1+|Z|^2)^2+t^2\bigr)^{-N/2},
\]
\[
 b(Z,t)=h^N
 \bigl((h^2+|Z|^2)^2+t^2\bigr)^{-N/2}.
\]
Thus \(F_h=a+b\).  With
\[
 R=(|Z|^4+t^2)^{1/4},
\]
one has
\[
 a\asymp(1+R^2)^{-N},\qquad
 b\asymp h^N(h^2+R^2)^{-N},
\qquad
 \dd Z\,\dd t\asymp R^{2N+1}\dd R.
\]
Dominated convergence applied to \(h^{-N}a^{q-1}b\) gives
\eqref{eq:fixed-two-bubble-interaction}.

The crossing set can be computed exactly:
\[
 \{b\leq a\}=\{R\geq\sqrt h\}.
\]
Indeed, after raising \(b\leq a\) to the power \(2/N\), the difference
between the right- and left-hand sides is
\[
 (1-h^2)(R^4-h^2).
\]
On \(\{b\leq a\}\) and its complement, respectively, use
\[
 (a+b)^q=a^q+qa^{q-1}b+O(a^{q-2}b^2),
\]
\[
 (a+b)^q=b^q+qb^{q-1}a+O(b^{q-2}a^2).
\]
The omitted pure masses and linear interactions satisfy
\[
 \begin{aligned}
 \int_{R<\sqrt h}a^q
 +\int_{R\geq\sqrt h}b^q&=O(h^{N+1}),\\
 \int_{R<\sqrt h}a^{q-1}b
 +\int_{R\geq\sqrt h}b^{q-1}a&=O(h^{N+1}).
 \end{aligned}
\]
The two Taylor remainders are bounded by
\[
 h^{2N}\int_{\sqrt h}^1R^{-2N+1}\dd R
 +h^2\int_h^{\sqrt h}R^{2N-3}\dd R
 +O(h^{2N})
=O(h^{N+1}),
\]
because \(N\geq2\).  Finally, the Euler--Lagrange equations and symmetry
of the energy pairing give
\[
 \int a^{q-1}b=\int b^{q-1}a=\frac{I_h}{\widetilde S}.
\]
This proves \eqref{eq:fixed-two-bubble-critical-mass}.

It remains to compute the cone distance.  For a unit-\(L^q\) bubble
\(\widetilde B_g\), minimization over its real amplitude gives
\begin{equation}\label{eq:full-cone-correlation-formula}
 \operatorname{dist}_{\dot H^1}
 (F,\widetilde{\mathfrak M}^{\mathrm{full}})^2
 =\widetilde E(F)
 -\frac1{\widetilde S}\sup_g
 |\widetilde E(F,\widetilde B_g)|^2.
\end{equation}
Put
\[
 C_h=\sup_g\widetilde E(a+b,\widetilde B_g).
\]
Positivity removes the absolute value, and testing with either summand
gives \(C_h\geq\widetilde S+I_h\).

We record the required two-well localization quantitatively.  There are
fixed coordinate neighborhoods of the scale-one and scale-\(h\) frames
such that every frame \(g_h\) satisfying
\begin{equation}\label{eq:two-well-near-maximizer}
 \widetilde E(a+b,\widetilde B_{g_h})
 \geq C_h-I_h^2
 \geq\widetilde S+I_h-I_h^2
\end{equation}
belongs to their union for all small \(h\).  Suppose otherwise along a
sequence \(h\downarrow0\).  There are two cases.

If \(g_h\) escapes relative to both bubble frames, frame orthogonality and
weak convergence of translated--dilated bubbles give
\[
 \widetilde E(a,\widetilde B_{g_h})
 +\widetilde E(b,\widetilde B_{g_h})=o(1),
\]
contradicting \eqref{eq:two-well-near-maximizer}.  Otherwise \(g_h\)
remains bounded relative to one of the two frames.  Suppose it is the
scale-one frame.  After passage to a subsequence, its normalized parameters
converge to a frame outside the chosen neighborhood.  Strictness in the
Cauchy--Schwarz inequality for two distinct normalized bubbles, followed
by compactness of the normalized parameter set, gives a constant
\(\eta>0\), independent of \(h\), such that
\[
 \widetilde E(a,\widetilde B_{g_h})
 \leq\widetilde S-\eta.
\]
The scale-\(h\) frame is orthogonal to every frame bounded relative to the
scale-one frame, and hence
\[
 \widetilde E(b,\widetilde B_{g_h})=o(1).
\]
The resulting upper bound \(\widetilde S-\eta+o(1)\) again contradicts
\eqref{eq:two-well-near-maximizer}.  The case in which \(g_h\) remains
bounded relative to the scale-\(h\) frame is identical after the centered
CR inversion.  This proves the localization.

Near the scale-one frame, normalized orbit coordinates satisfy
\[
 \widetilde E(a,\widetilde B_{\exp\xi})
 \leq\widetilde S-c|\xi|^2.
\]
For \(J_h(\xi)=\widetilde E(b,\widetilde B_{\exp\xi})\),
\[
 \partial_{\xi_j}J_h(\xi)
 =\widetilde S\int b^{q-1}
 \partial_{\xi_j}\widetilde B_{\exp\xi}.
\]
The tangent functions are uniformly bounded on a fixed coordinate
neighborhood, while critical scaling and
\eqref{eq:fixed-two-bubble-interaction} give
\[
 \int b^{q-1}=h^N\int a^{q-1},\qquad I_h\asymp h^N.
\]
Hence \(|\nabla J_h|\leq CI_h\), and
\[
 \widetilde E(a+b,\widetilde B_{\exp\xi})
 \leq\widetilde S+I_h-c|\xi|^2+CI_h|\xi|.
\]
The maximum in this well is therefore
\(\widetilde S+I_h+O(I_h^2)\).  The centered CR inversion that exchanges
the two scales gives the same estimate in the second well.  Choose \(g_h\)
as in \eqref{eq:two-well-near-maximizer}.  The two-well localization and
the two local estimates give
\[
 C_h-I_h^2
 \leq\widetilde E(a+b,\widetilde B_{g_h})
 \leq\widetilde S+I_h+O(I_h^2).
\]
Together with \(C_h\geq\widetilde S+I_h\), this yields
\[
 C_h=\widetilde S+I_h+O(I_h^2).
\]
Since
\(\widetilde E(F_h)=2\widetilde S+2I_h\),
\eqref{eq:full-cone-correlation-formula} proves
\eqref{eq:fixed-two-bubble-distance}.
\end{proof}

\begin{proposition}[Strictness below the two-bubble level]
\label{prop:strict-two-bubble-level}
For every \(n,\ell\geq1\),
\[
 \kappa_{n,\ell}^{\mathrm{opt}}
 <2-2^{N/(N+1)}.
\]
\end{proposition}

\begin{proof}
Lemma~\ref{lem:fixed-two-bubble-expansions} and the exact identity
\(\widetilde E(F_h)=2\widetilde S+2I_h\) give
\[
 \begin{aligned}
 \widetilde E(F_h)-\widetilde S\|F_h\|_q^2
 ={}&\widetilde S\bigl(2-2^{2/q}\bigr)\\
 &-2\bigl(2^{2/q}-1\bigr)I_h+o(I_h).
 \end{aligned}
\]
The fixed optimizer cone is contained in the full optimizer cone.  Thus,
for a fixed function, the quotient with distance to the fixed cone is no
larger than the quotient with distance to the full cone.  The polar lift
multiplies both the deficit and the fixed-cone distance squared by the
same factor \(c_\ell\).  Using
\eqref{eq:fixed-two-bubble-distance}, we obtain
\[
 \kappa_{n,\ell}^{\mathrm{opt}}
 \leq
 2-2^{2/q}
 -\frac{2(2^{2/q}-1)}{\widetilde S}I_h+o(I_h).
\]
Since \(I_h>0\), the claim follows from
\(2/q=N/(N+1)\).
\end{proof}

\subsection{A strict normal test on the CR sphere}

The two-bubble level lies above the local spectral ratio in the lowest
dimensions.  We next construct a fixed normal perturbation that is
strictly below the local ratio in every dimension.

Let \(\dd\sigma\) be probability measure on
\(\mathbb S^{2N+1}\), and normalize the CR intertwining operator by its
constant eigenvalue.  Denote the resulting operator by \(\mathbf A\).
Lemma~\ref{lem:spherical-spectrum-normalization} gives
\begin{equation}\label{eq:normalized-cr-operator}
 \mathbf A1=1,\qquad
 \mathbf A|_{\mathcal H_{j,k}}
 =\Lambda_{j,k}
 =\frac{(N+2j)(N+2k)}{N^2}.
\end{equation}
This normalization multiplies the deficit and the squared energy distance
by the same positive constant.  The spherical deficit is therefore
\[
 \mathscr D_{\mathbb S}(v)
 =\langle v,\mathbf Av\rangle
 -\left(\int_{\mathbb S^{2N+1}}|v|^q\dd\sigma\right)^{2/q}.
\]

At the constant optimizer, the tangent space to the real nonzero full
optimizer cone is
\[
 \mathbb R1\oplus
 (\mathcal H_{1,0}\oplus\mathcal H_{0,1})_{\mathbb R}.
\]
Let \(Y_{\mathbb S}\) be its \(\mathbf A\)-orthogonal complement.
The normal tubular map of this finite-dimensional smooth cone gives, for
all sufficiently small \(r\in Y_{\mathbb S}\),
\begin{equation}\label{eq:exact-spherical-normal-distance}
 \operatorname{dist}_{\mathbf A}
 (1+r,\mathfrak M_{\mathbb S}^{\mathrm{full}})^2
 =\langle r,\mathbf Ar\rangle.
\end{equation}
Indeed, \(1+r\) lies on the normal fiber over \(1\), so its unique local
metric projection is exactly \(1\).  Properness of the normalized bubble
orbit, together with separation from the cone vertex, excludes a second
nearest point outside the local cone chart.

\begin{proposition}[Strictness below the local spectral ratio]
\label{prop:strict-local-stability-level}
For every \(n,\ell\geq1\),
\[
 \kappa_{n,\ell}^{\mathrm{opt}}<\frac2{N+4}.
\]
\end{proposition}

\begin{proof}
Put
\[
 \varphi=\operatorname{Re}(\zeta_{N+1}^2),\qquad
 X=|\zeta_{N+1}|^2.
\]
The function \(\varphi\) is \(U(\ell)\)-fixed and belongs to
\((\mathcal H_{2,0}\oplus\mathcal H_{0,2})_{\mathbb R}\).  Since
\(X\) has the \(\operatorname{Beta}(1,N)\) distribution and the phase of
\(\zeta_{N+1}\) is uniform,
\begin{equation}\label{eq:fixed-quartic-moments}
 M_2:=\int\varphi^2\dd\sigma
 =\frac1{(N+1)(N+2)},\qquad
 \int\varphi^3\dd\sigma=0,
\end{equation}
\[
 M_4:=\int\varphi^4\dd\sigma
 =\frac9{(N+1)(N+2)(N+3)(N+4)}.
\]
Apart from its constant part, the decomposition
\[
 \varphi^2
 =\frac14\bigl(\zeta_{N+1}^4+\overline{\zeta}_{N+1}^4\bigr)
 +\frac12X^2
\]
has components only in
\[
 \mathcal H_{4,0}\oplus\mathcal H_{0,4},\qquad
 \mathcal H_{1,1},\qquad
 \mathcal H_{2,2}.
\]
If \(p_{jk}\) denotes the squared \(L^2(\dd\sigma)\)-norm of the
corresponding real projection, beta-moment orthogonalization gives
\begin{equation}\label{eq:fixed-quartic-projection-norms}
 \begin{aligned}
 p_{40}
 &=\frac3{(N+1)(N+2)(N+3)(N+4)},\\
 p_{11}
 &=\frac{4N}{(N+1)^2(N+2)(N+3)^2},\\
 p_{22}
 &=\frac{N}{(N+2)^2(N+3)^2(N+4)}.
 \end{aligned}
\end{equation}
More explicitly, if
\(\zeta_{N+1}=\sqrt X e^{i\vartheta}\), then
\[
 \begin{aligned}
 P_{40}\varphi^2
 &=\tfrac12X^2\cos(4\vartheta),\\
 P_{11}\varphi^2
 &=\frac2{N+3}\left(X-\frac1{N+1}\right),\\
 P_{22}\varphi^2
 &=\frac12\left[
 X^2-\frac2{(N+1)(N+2)}
 -\frac4{N+3}\left(X-\frac1{N+1}\right)
 \right].
 \end{aligned}
\]

The eigenvalue of \(\varphi\) is
\[
 \Lambda_2=\frac{N+4}{N}.
\]
Relative to \(\Lambda_2\), the three positive gaps are
\begin{equation}\label{eq:fixed-quartic-spectral-gaps}
 \delta_{40}=\frac4{N+4},\qquad
 \delta_{11}=\frac4{N(N+4)},\qquad
 \delta_{22}=\frac4N.
\end{equation}
Take
\[
 r_\varepsilon=\varepsilon\varphi+\varepsilon^2\psi,\qquad
 \psi\in Y_{\mathbb S},\qquad
 \int\varphi\psi\dd\sigma=0.
\]
For small \(\varepsilon\), the function \(1+r_\varepsilon\) is positive,
and \eqref{eq:exact-spherical-normal-distance} gives
\[
 \operatorname{dist}_{\mathbf A}
 (1+r_\varepsilon,\mathfrak M_{\mathbb S}^{\mathrm{full}})^2
 =\Lambda_2M_2\varepsilon^2
 +\langle\psi,\mathbf A\psi\rangle\varepsilon^4.
\]
Expanding first the \(q\)-power and then the outer \(2/q\)-power yields
\[
 \begin{aligned}
 \left(\int(1+r_\varepsilon)^q\dd\sigma\right)^{2/q}
 &=1+(q-1)M_2\varepsilon^2\\
 &\quad+\varepsilon^4\left[
 (q-1)\|\psi\|_2^2
 +(q-1)(q-2)\langle\varphi^2,\psi\rangle\right.\\
 &\hspace{26mm}\left.
 +\frac{(q-1)(q-2)(q-3)}{12}M_4
 -\frac{(q-2)(q-1)^2}{4}M_2^2
 \right]+o(\varepsilon^4).
 \end{aligned}
\]

Let \(\kappa_{\mathrm{loc}}=2/(N+4)\).  Since
\[
 1-\kappa_{\mathrm{loc}}=\frac{q-1}{\Lambda_2},
\]
the coefficient of \(\varepsilon^4\) in
\[
 \mathscr D_{\mathbb S}(1+r_\varepsilon)
 -\kappa_{\mathrm{loc}}
 \operatorname{dist}_{\mathbf A}
 (1+r_\varepsilon,\mathfrak M_{\mathbb S}^{\mathrm{full}})^2
\]
is
\[
 \begin{aligned}
 \mathcal C(\psi)
 ={}&C_0
 +(q-1)\sum_\mu
 \left(\frac\mu{\Lambda_2}-1\right)\|\psi_\mu\|_2^2\\
 &-(q-1)(q-2)\langle\varphi^2,\psi\rangle,
 \end{aligned}
\]
where
\[
 C_0
 =-\frac{(q-1)(q-2)(q-3)}{12}M_4
 +\frac{(q-2)(q-1)^2}{4}M_2^2.
\]
Completing the square in the three components above gives
\[
 \psi_*
 =\frac{q-2}{2}
 \left(
 \frac{P_{40}\varphi^2}{\delta_{40}}
 +\frac{P_{11}\varphi^2}{\delta_{11}}
 +\frac{P_{22}\varphi^2}{\delta_{22}}
 \right),
\]
or
\[
 \psi_*
 =\frac{N+4}{4N}P_{40}\varphi^2
 +\frac{N+4}{4}P_{11}\varphi^2
 +\frac14P_{22}\varphi^2.
\]
This correction remains real, \(U(\ell)\)-fixed, and normal to the
optimizer cone.  Substitution of
\eqref{eq:fixed-quartic-moments}--\eqref{eq:fixed-quartic-spectral-gaps}
gives
\[
 \mathcal C(\psi_*)
 =-\frac{2N^3+12N^2+20N+7}
 {2N^2(N+1)^2(N+2)(N+3)(N+4)}<0.
\]
Consequently,
\begin{equation}\label{eq:fixed-quartic-quotient}
 \begin{aligned}
 &\frac{\mathscr D_{\mathbb S}(1+r_\varepsilon)}
 {\operatorname{dist}_{\mathbf A}
 (1+r_\varepsilon,\mathfrak M_{\mathbb S}^{\mathrm{full}})^2}\\
 &\quad=\frac2{N+4}
 -\frac{2N^3+12N^2+20N+7}
 {2N(N+1)(N+3)(N+4)^2}\varepsilon^2
 +o(\varepsilon^2).
 \end{aligned}
\end{equation}
The whole test is \(U(\ell)\)-fixed.  Moreover, its nearest point in the
full optimizer cone is \(1\), which also belongs to the fixed cone.
Thus its distances to the two cones agree.  The inverse Cayley transform
and the polar descent preserve the quotient, and the strict negativity in
\eqref{eq:fixed-quartic-quotient} proves the proposition.
\end{proof}

\subsection{Compactness below the two thresholds}

For \(w\in\dot S^1_\ell(\Ucal;\mathbb R)\), define
\begin{equation}\label{eq:fixed-correlation-mass}
 m(w)=\frac1S\sup_g
 |E_\ell(w,\mathcal T_gB_0)|^2,
\end{equation}
where the supremum is over reduced frames.  Minimization over the real
amplitude gives
\begin{equation}\label{eq:fixed-cone-distance-correlation}
 d_\ell(w)^2=E_\ell(w)-m(w),
\end{equation}
and H\"older's inequality gives
\begin{equation}\label{eq:fixed-correlation-bound}
 m(w)\leq S\|w\|_q^2.
\end{equation}
We also write
\[
 \mathcal S(w)=\frac{E_\ell(w)}{\|w\|_q^2}
 \qquad(w\neq0).
\]

\begin{lemma}[Correlation max-splitting]
\label{lem:fixed-correlation-max-splitting}
If \(u_k\rightharpoonup u\) in \(\dot S^1_\ell\) and
\(v_k=u_k-u\rightharpoonup0\), then
\[
 m(u_k)=\max\{m(u),m(v_k)\}+o(1).
\]
\end{lemma}

\begin{proof}
The reduced bubble orbit has the properness alternative used in
Theorem~\ref{thm:intro-profile}: a frame sequence has a relatively compact
subsequence, or the corresponding bubbles converge weakly to zero.  Apply
this alternative to near-maximizers in
\eqref{eq:fixed-correlation-mass}.  A relatively compact frame does not
see \(v_k\), while an escaping frame does not see \(u\), which proves the
upper bound.  For the converse, test first with a fixed near-maximizer for
\(m(u)\).  A near-maximizer for \(m(v_k)\) either escapes, and hence does
not see \(u\), or is relatively compact, in which case \(m(v_k)=o(1)\).
\end{proof}

\begin{lemma}[Balance of correlation masses]
\label{lem:fixed-correlation-balance}
Let \((u_k)\) be a normalized minimizing sequence for
\eqref{eq:optimal-fixed-constant}.  After moving a nonzero profile to the
identity, write
\[
 u_k=f+g_k,\qquad f\neq0,\qquad g_k\rightharpoonup0.
\]
If \(\liminf_kE_\ell(g_k)>0\), then
\[
 m(f)=m(g_k)+o(1).
\]
\end{lemma}

\begin{proof}
Energy orthogonality, the Brezis--Lieb lemma, and
Lemma~\ref{lem:fixed-correlation-max-splitting} split the numerator and
replace the correlation mass in the denominator by the larger of the two
component masses.

We first dispose of zero-mass degeneracies.  If
\(\|g_k\|_q\to0\) while its energy stays positive, its contribution to
both numerator and denominator has quotient one.  The other component has
quotient at least \(\kappa_{n,\ell}^{\mathrm{opt}}<1\), so their weighted
quotient cannot tend to the infimum.  If both critical masses have positive
limits but one limiting correlation mass is zero, strict concavity of the
map \(t\mapsto t^{2/q}\) produces a strictly positive splitting remainder
in the numerator, again contradicting minimality.  The same conclusion
holds if one component lies on the optimizer cone: its zero numerator and
denominator are omitted, while the strict splitting remainder remains.
Thus both limiting correlation masses are positive.

Suppose, after passage to a subsequence, that the correlation mass of
\(g_k\) is strictly larger than \(m(f)\).  Rescale so that
\(\|g_k\|_q=1\), put \(\eta=\|f\|_q\), and define
\[
 \Phi(\eta)=(1+\eta^q)^{2/q}-1.
\]
The limiting quotient has the form
\[
 \frac{A+C}{B+D},
\]
where
\[
 \begin{aligned}
 A&=E_\ell(g_k)-S,&
 B&=E_\ell(g_k)-m(g_k),\\
 C&=E_\ell(f)-S\Phi(\eta),&
 D&=E_\ell(f).
 \end{aligned}
\]
All quantities here denote their limits along the chosen subsequence.
Since \(A/B\geq\kappa_{n,\ell}^{\mathrm{opt}}\) and the combined ratio
tends to the infimum, cross multiplication gives \(A/B\geq C/D\).

Choose \(c>1\) so that \(m(cf)=\lim m(g_k)\), and denote by
\(C'\) and \(D'=c^2D\) the last two quantities after replacing \(f\) by
\(cf\).  Since \(\Phi(t)/t^2\) is strictly increasing for \(q>2\),
\[
 \frac CD
 =1-\frac S{\mathcal S(f)}\frac{\Phi(\eta)}{\eta^2}
 >
 1-\frac S{\mathcal S(f)}
 \frac{\Phi(c\eta)}{(c\eta)^2}
 =\frac{C'}{D'}.
\]
Together with \(D'>D\), the elementary weighted-ratio inequality gives
\[
 \frac{A+C'}{B+D'}<\frac{A+C}{B+D}.
\]
The left-hand side is the limiting quotient of the valid competitor
\(g_k+cf\), a contradiction.  If \(B=0\), the strict inequality
\(C/D>C'/D'\) gives the same conclusion directly.  The reverse ordering
is treated by interchanging the two components.  Hence their correlation
masses must agree.
\end{proof}

\begin{corollary}[Sequential local threshold]
\label{cor:sequential-local-stability-threshold}
If \(\|u_k\|_q=1\),
\(d_\ell(u_k)\to0\), and \(u_k\notin\mathfrak M_\ell^{\mathbb R}\), then
\[
 \liminf_{k\to\infty}
 \frac{\mathscr D_\ell(u_k)}{d_\ell(u_k)^2}
 \geq\frac2{N+4}.
\]
\end{corollary}

\begin{proof}
The normalization excludes convergence to the cone vertex.  The tubular
projection near the nonzero cone gives, after an amplitude and frame
modulation,
\[
 u_k=a_k\mathcal T_{g_k}(B_0+r_k),\qquad
 r_k\in Y_\ell,\qquad r_k\to0.
\]
Homogeneity and Proposition~\ref{prop:local-stability-coefficient} give
the conclusion.
\end{proof}

\begin{proof}[Proof of Theorem~\ref{thm:optimal-fixed-attainment}]
The two strict inequalities in
\eqref{eq:two-strict-stability-thresholds} follow from
Propositions~\ref{prop:strict-two-bubble-level}
and~\ref{prop:strict-local-stability-level}; positivity follows from
Theorem~\ref{thm:intro-stability}.

Let \((u_k)\) be a minimizing sequence with \(\|u_k\|_q=1\).  Since
\(0\in\mathfrak M_\ell^{\mathbb R}\),
\(d_\ell(u_k)^2\leq E_\ell(u_k)\), and therefore
\[
 E_\ell(u_k)-S
 \leq\bigl(\kappa_{n,\ell}^{\mathrm{opt}}+o(1)\bigr)
 E_\ell(u_k).
\]
The strict local bound implies
\(\kappa_{n,\ell}^{\mathrm{opt}}<1\), so \((u_k)\) is bounded.
The profile decomposition in
Theorem~\ref{thm:intro-profile} supplies, after applying reduced frames,
a nonzero weak profile:
\[
 u_k=f+g_k,\qquad f\neq0,\qquad g_k\rightharpoonup0.
\]
It also gives
\[
 E_\ell(u_k)=E_\ell(f)+E_\ell(g_k)+o(1),
\]
\[
 1=\|f\|_q^q+\|g_k\|_q^q+o(1).
\]

Suppose that \(E_\ell(g_k)\not\to0\).  By
Lemma~\ref{lem:fixed-correlation-balance},
\[
 m(f)=m(g_k)+o(1).
\]
After interchanging the two components, assume
\(\|g_k\|_q\leq\|f\|_q+o(1)\), rescale so that
\(\|f\|_q=1\), and put \(r_k=\|g_k\|_q\leq1+o(1)\).
If \(f\notin\mathfrak M_\ell^{\mathbb R}\), its quotient is no smaller
than the infimum.  If instead \(f\in\mathfrak M_\ell^{\mathbb R}\), its
zero deficit and zero squared distance are omitted from the split
quotient.  Thus in either case,
\eqref{eq:fixed-cone-distance-correlation} and the balance identity give
\[
 \begin{aligned}
 \kappa_{n,\ell}^{\mathrm{opt}}+o(1)
 &\geq
 1-\frac S{\mathcal S(g_k)}
 \frac{(1+r_k^q)^{2/q}-1}{r_k^2}\\
 &\geq
 1-\frac S{\mathcal S(g_k)}(2^{2/q}-1)\\
 &\geq 2-2^{2/q}.
 \end{aligned}
\]
Here the second inequality uses the monotonicity of
\(((1+r^q)^{2/q}-1)/r^2\) on \((0,1]\), and the last uses the sharp
Sobolev inequality \(\mathcal S(g_k)\geq S\).  This contradicts
Proposition~\ref{prop:strict-two-bubble-level}.  Hence
\[
 E_\ell(g_k)\longrightarrow0,
\]
and the convergence to \(f\) is strong in \(\dot S^1_\ell\).

If \(f\) belonged to the nonzero optimizer cone, then
\(d_\ell(u_k)\to0\), and
Corollary~\ref{cor:sequential-local-stability-threshold} would contradict
Proposition~\ref{prop:strict-local-stability-level}.  Thus
\[
 f\notin\mathfrak M_\ell^{\mathbb R}.
\]
Strong convergence preserves the energy, the critical norm, and the
distance to the closed cone.  The function \(f\) therefore attains
\eqref{eq:optimal-fixed-constant}.  The same argument applies to every
normalized minimizing sequence and proves the asserted compactness modulo
reduced frames.
\end{proof}

\section{Cayley conjugation and the intrinsic form}
\label{sec:cayley}

Let \(m=n+1\) and equip \(\CHyp^m\) with the metric of holomorphic sectional
curvature \(-4\).  In Siegel coordinates,
\[
 \omega=t+i(|z|^2+\rho),\qquad \rho>0,
\]
the volume and Laplacian are
\begin{equation}\label{eq:ch-volume-laplacian}
 \dd V_g=\frac14\rho^{-(n+2)}\dd z\dd t\dd\rho,
 \qquad
 \Delta_g
 =\rho\Delta_{\Heis^n}
 +4\rho^2(\partial_{\rho\rho}+T^2)-4n\rho\partial_\rho.
\end{equation}
After translating Lu--Yang's
\(\Delta_b=\frac14\Delta_{\Heis^n}\) convention, these are their
Siegel formulas for holomorphic sectional curvature \(-4\)
\cite{LuYang2022SiegelHardySobolevMazya}.  Equivalently,
\eqref{eq:ch-volume-laplacian} exhibits \(\CHyp^m\) in its harmonic
\(AN\) realization \(\Heis^n\rtimes\R_+\); the spherical analysis of
these groups, including the radial part of \(\Delta_g\), is developed in
\cite{AnkerDamekYacoub1996HarmonicAN}.

For a real parameter \(a\) with \(n+a>0\), set
\[
 \gamma_a=\frac{n+a}{2},
 \qquad p_a=1+\frac2{n+a},
\]
\[
 \Gell_a=-\Delta_{\Heis^n}
 -4\rho(\partial_{\rho\rho}+T^2)-4a\partial_\rho
\]
and
\[
 P_a^{\mathrm{CH}}
 =-\Delta_g-(n+1)^2+(a-1)^2.
\]

\begin{proposition}[Exact Cayley conjugation]\label{prop:cayley}
For every compactly supported smooth \(u\),
\begin{equation}\label{eq:cayley-conjugation}
 \Gell_a(\rho^{-\gamma_a}u)
 =\rho^{-\gamma_a-1}P_a^{\mathrm{CH}}u.
\end{equation}
Consequently, if \(u\geq0\) and \(v=\rho^{-\gamma_a}u\), then
\[
 (P_a^{\mathrm{CH}}-\lambda)u=u^{p_a}
 \quad\Longleftrightarrow\quad
 \Gell_av-\lambda\rho^{-1}v=v^{p_a}.
\]
\end{proposition}

\begin{proof}
Formula \eqref{eq:ch-volume-laplacian} follows by transporting the ball
metric through the Cayley map; the powers of the Cayley denominator cancel
in the invariant volume form.  Substituting
\(v=\rho^{-\gamma_a}u\) into \(\Gell_av\) and collecting the
first-order terms gives
\[
 \Gell_av
 =\rho^{-\gamma_a-1}
 \left[
 -\Delta_g
 +4\gamma_a(a-\gamma_a-1)
 \right]u.
\]
The identities
\[
 2\gamma_a-a=n,
 \qquad
 4\gamma_a(a-\gamma_a-1)
 =(a-1)^2-(n+1)^2
\]
give \eqref{eq:cayley-conjugation}.  Finally,
\(\gamma_ap_a=\gamma_a+1\), so the nonlinear powers transform without an
additional weight.
\end{proof}

\begin{remark}
The exact reduced image of a constant complex-hyperbolic spectral term is
\(\rho^{-1}\).  In particular,
\(\bigl[(1+\rho+|z|^2)^2+t^2\bigr]^{-1}\) is a different localized
potential.  Moreover, the ordinary Riemannian critical power in real
dimension \(2n+2\) is \(1+2/n\); it agrees with \(p_a\) only when
\(a=0\).
\end{remark}

For \(\sigma>0\), define on \(C_c^\infty(\CHyp^m)\)
\[
 \mathfrak b_\sigma(U,\Phi)
 =\int_{\CHyp^m}
 \bigl(
 \langle\nabla_gU,\nabla_g\Phi\rangle_g
 -m^2U\Phi+\sigma U\Phi
 \bigr)\dd V_g.
\]

\begin{theorem}[Completed Cayley correspondence]
\label{thm:completed-cayley}
Let \(n\geq1\) and let \(\ell>1\) be an integer.  The Friedrichs form domain of
\[
 -\Delta_g-m^2+(\ell-1)^2
\]
is \(H^1(\CHyp^m,g)\).  The map
\[
 \mathcal C_\ell U=\rho^{-(n+\ell)/2}U
\]
extends from the smooth core onto \(\dot S^1_\ell(\Ucal)\), and
\begin{align}
 E_\ell(\mathcal C_\ell U,\mathcal C_\ell\Phi)
 &=4\mathfrak b_{(\ell-1)^2}(U,\Phi),
 \label{eq:cayley-energy}\\
 \int_{\Ucal}\rho^{-1}(\mathcal C_\ell U)
 (\mathcal C_\ell\Phi)\dd\mu_\ell
 &=4\int_{\CHyp^m}U\Phi\dd V_g,
 \label{eq:cayley-hardy}\\
 \|\mathcal C_\ell U\|_{q_\ell}^{q_\ell}
 &=4\|U\|_{L^{q_\ell}(dV_g)}^{q_\ell}.
 \label{eq:cayley-critical}
\end{align}
If \(0<\lambda<(\ell-1)^2\) and
\(\mu=(\ell-1)^2-\lambda\), then
\[
 A^H_{\lambda,\ell}(\mathcal C_\ell U,\mathcal C_\ell\Phi)
 =4\mathfrak b_\mu(U,\Phi),
\]
where
\[
 A^H_{\lambda,\ell}(v,\varphi)
 =E_\ell(v,\varphi)
 -\lambda\int_{\Ucal}\rho^{-1}v\varphi\dd\mu_\ell.
\]
\end{theorem}

\begin{proof}
Let \(h=\rho^{m/2}\).  A direct calculation gives
\(\Delta_gh=-m^2h\), and the ground-state representation in the sense of
Frank--Seiringer \cite{FrankSeiringer2008Hardy} yields
\begin{equation}\label{eq:ground-state-transform}
 \int_{\CHyp^m}(|\nabla_gU|^2-m^2U^2)\dd V_g
 =\int_{\CHyp^m}h^2|\nabla_g(U/h)|^2\dd V_g\geq0.
\end{equation}
We also record why the lower bound is sharp.  Fix
\(0\ne\eta\in C_c^\infty(\Heis^n)\), and choose
\(\chi_\varepsilon\in C_c^\infty(0,\infty)\) which vanishes below
\(\varepsilon^2\), rises logarithmically to one on
\([\varepsilon^2,\varepsilon]\), equals one on \([\varepsilon,1]\), and
vanishes above \(2\).  Set
\[
 v_\varepsilon(z,t,\rho)=\eta(z,t)\chi_\varepsilon(\rho),
 \qquad U_\varepsilon=\rho^{m/2}v_\varepsilon.
\]
The smooth conjugation with parameter \(a=1\) gives
\[
 4\int_{\CHyp^m}(|\nabla_gU_\varepsilon|^2-m^2U_\varepsilon^2)
 \dd V_g=E_1(v_\varepsilon),
 \qquad
 4\int_{\CHyp^m}U_\varepsilon^2\dd V_g
 =\int_{\Ucal}\rho^{-1}v_\varepsilon^2\dd\mu_1.
\]
The logarithmic transition contributes
\(O(|\log\varepsilon|^{-1})\) to the radial energy; all other energy
terms remain bounded, whereas the last integral is comparable to
\(|\log\varepsilon|\).  The Rayleigh quotient of \(U_\varepsilon\)
therefore tends to \(m^2\).  Together with
\eqref{eq:ground-state-transform}, this proves
\[
 \inf\sigma(-\Delta_g)=m^2.
\]

For every \(\sigma>0\), the spectral inequality implies
\[
 \sigma\|U\|_2^2\leq\mathfrak b_\sigma(U,U)
 \leq\|\nabla_gU\|_2^2+\sigma\|U\|_2^2,
\]
and
\[
 \|\nabla_gU\|_2^2
 \leq\max\left\{1,\frac{m^2}{\sigma}\right\}
 \mathfrak b_\sigma(U,U).
\]
Thus \(\mathfrak b_\sigma^{1/2}\) and the ordinary \(H^1_g\) norm are
equivalent.  The metric is complete; geodesic cutoffs
\(\chi_R(x)=\chi(d_g(o,x)/R)\), with
\(|\nabla_g\chi_R|\leq C/R\), followed by local mollification show that
\(C_c^\infty(\CHyp^m)\) is dense in \(H^1_g\).  Consequently the closure of
\(\mathfrak b_\sigma\) has domain exactly \(H^1_g\).  The general
complete-manifold and semibounded-form frameworks are due to Strichartz and
Friedrichs
\cite{Strichartz1983CompleteLaplacian,Friedrichs1934Spektraltheorie}.

Proposition~\ref{prop:cayley}, \eqref{eq:ch-volume-laplacian}, and
integration by parts give
\eqref{eq:cayley-energy}--\eqref{eq:cayley-critical} on the two smooth
cores.  On these cores, \(\mathcal C_\ell\) is a bijection and
\[
 \|\mathcal C_\ell U\|_{\dot S^1_\ell}
 =2\mathfrak b_{(\ell-1)^2}(U,U)^{1/2}.
\]
Hence \(\mathcal C_\ell\) and its algebraic inverse extend continuously to
the respective completions.  To see surjectivity explicitly, let
\(v_j\in C_c^\infty(\Ucal)\) converge to
\(v\in\dot S^1_\ell(\Ucal)\), and put
\(U_j=\mathcal C_\ell^{-1}v_j\).  The preceding identity makes
\((U_j)\) Cauchy in \(H^1_g\); if \(U_j\to U\), then the extended map sends
\(U\) to \(v\).  The core compositions of the two maps are the identity, so
their extensions are mutual inverses.  The identities
\eqref{eq:cayley-energy}--\eqref{eq:cayley-critical} pass to the completed
spaces by density.  Finally, subtracting \(\lambda\) times
\eqref{eq:cayley-hardy} from \eqref{eq:cayley-energy} proves the perturbed
identity.
\end{proof}

\begin{proposition}[Transported isometries]\label{prop:transported-action}
Let \(\ell\geq1\) be an integer.  For
\(g\in\operatorname{Isom}(\CHyp^m,g)\), define on the reduced smooth core
\[
 (\widetilde\Pi_gv)(x)
 =\left[\frac{\rho(g^{-1}x)}{\rho(x)}\right]^{(n+\ell)/2}
 v(g^{-1}x).
\]
This action extends to an isometric action on
\(\dot S^1_\ell(\Ucal)\), preserves the critical norm, and, when
\(\ell>1\), preserves the Hardy form
\(\int\rho^{-1}v^2\dd\mu_\ell\).
\end{proposition}

\begin{proof}
On the smooth core this is the conjugate under \(\mathcal C_\ell\) of the
ordinary complex-hyperbolic isometry action.  The smooth Cayley identities
from Proposition~\ref{prop:cayley}, together with
\eqref{eq:ch-volume-laplacian}, show for every integer \(\ell\geq1\) that
this conjugate preserves \(E_\ell\) and the critical norm.  Energy density
therefore gives a unique isometric extension to
\(\dot S^1_\ell(\Ucal)\), while the sharp Sobolev inequality passes critical
norm preservation to the completion.  The action of \(g^{-1}\) gives the
inverse.  When \(\ell>1\), Theorem~\ref{thm:completed-cayley} and
\eqref{eq:cayley-hardy} additionally give preservation of the continuous
Hardy form.
\end{proof}

\begin{remark}
Proposition~\ref{prop:transported-action} makes the set of minimizers
invariant.  By itself this does not imply that an individual minimizer is
fixed by a point stabilizer.  For sufficiently small Hardy coupling, the
additional local uniqueness argument in
the proof of Theorem~\ref{thm:intro-symmetry} does yield point-stabilizer
invariance and geodesic radiality.  No analogous assertion is made here for
arbitrary higher-energy solutions.
\end{remark}

\section{The Hardy quotient and an intrinsic positive ground state}
\label{sec:hardy}

For spectral Brezis--Nirenberg equations on real hyperbolic space,
Mancini--Sandeep and Li--Lu--Yang establish existence and nonexistence
results and, in theorem-specific parameter regimes, symmetry conclusions
\cite{ManciniSandeep2008Semilinear,LiLuYang2022HyperbolicBN}.  Here the
integer Geller exponent and the reduced Hardy weight arise from the
completion-level Cayley correspondence of
Theorem~\ref{thm:completed-cayley}.  The small-coupling symmetry conclusion
proved below instead follows from fixed-sector nondegeneracy, local
uniqueness, and the transported holomorphic isometry action.  We then show
that the complete energy-space kernel of the small-coupling ground state
consists precisely of the infinitesimal center-moving isometries.

For \(\ell>1\), define
\[
 \mathsf H_\ell(v,\varphi)
 =\int_{\Ucal}\rho^{-1}v\varphi\dd\mu_\ell.
\]
The following sharp Hardy inequality determines the coercive spectral range
used below.

\begin{lemma}[Exact Hardy bottom]\label{lem:hardy-bottom-paper}
If \(\ell>1\), then
\begin{equation}\label{eq:hardy}
 (\ell-1)^2\mathsf H_\ell(v,v)\leq E_\ell(v)
 \qquad(v\in\dot S^1_\ell(\Ucal)),
\end{equation}
and the constant \((\ell-1)^2\) is optimal.
\end{lemma}

\begin{proof}
For \(v\in C_c^\infty(\Ucal)\), integration by parts gives
\[
 (\ell-1)\int_0^\infty \rho^{\ell-2}v^2\dd\rho
 =-2\int_0^\infty\rho^{\ell-1}vv_\rho\dd\rho.
\]
Cauchy--Schwarz and the radial term \(4\rho|v_\rho|^2\) in \(E_\ell\)
give \eqref{eq:hardy}.  Density extends it to the completion.

To prove optimality, put \(a=(\ell-1)/2\).  Choose
\(\chi\in C_c^\infty(\R)\), \(\chi\not\equiv0\), and set
\[
 g_R(s)=\chi(s/R),\qquad
 f_R(\rho)=\rho^{-a}g_R(\log\rho).
\]
Then \(f_R\in C_c^\infty(0,\infty)\), and the substitution
\(s=\log\rho\) gives
\[
\begin{aligned}
 \int_0^\infty\rho^{\ell-2}f_R^2\,d\rho
   &=\int_{\R}g_R^2\,ds,\\
 4\int_0^\infty\rho^\ell|f_R'|^2\,d\rho
   &=(\ell-1)^2\int_{\R}g_R^2\,ds
     +4\int_{\R}|g_R'|^2\,ds.
\end{aligned}
\]
The mixed term vanishes by integration by parts.  Since
\(\int|g_R'|^2/\int g_R^2=O(R^{-2})\), the one-dimensional quotient tends
to \((\ell-1)^2\).

Fix \(0\ne\eta\in C_c^\infty(\Heis^n)\) and define
\[
 v_{R,\varepsilon}(z,t,\rho)
 =\eta(z,t)f_R(\rho/\varepsilon).
\]
For fixed \(R\), both the Hardy mass and the radial energy have the factor
\(\varepsilon^{\ell-1}\).  After division by the Hardy mass, the
\(\nabla_H\eta\)-term is \(O_R(\varepsilon)\), while the additional
\(4\rho|\partial_t\eta|^2\)-term is \(O_R(\varepsilon^2)\).  Therefore
\[
 \lim_{R\to\infty}\lim_{\varepsilon\downarrow0}
 \frac{E_\ell(v_{R,\varepsilon})}
      {\mathsf H_\ell(v_{R,\varepsilon},v_{R,\varepsilon})}
 =(\ell-1)^2.
\]
A diagonal choice of \(\varepsilon=\varepsilon_R\downarrow0\) gives a
single core sequence and proves optimality.
\end{proof}

\begin{remark}[The case \(\ell=1\)]\label{rem:ell-one}
On the unchanged space \(\dot S^1_1(\Ucal)\), the bottom Hardy quotient is
zero and \(\mathsf H_1\) is not continuous.  Indeed, a logarithmic transition
from \(\rho=\varepsilon^2\) to \(\rho=\varepsilon\) has bounded energy while
its Hardy mass is comparable to \(|\log\varepsilon|\).  Thus there is no
positive coercive interval for the minus-sign Hardy perturbation when
\(\ell=1\).
\end{remark}

Assume henceforth in this section that
\[
 \ell>1,\qquad 0<\lambda<(\ell-1)^2.
\]
Put
\[
 A^H_{\lambda,\ell}(v)
 =E_\ell(v)-\lambda\mathsf H_\ell(v,v)
\]
and
\begin{equation}\label{eq:hardy-quotient}
 S^H_{\lambda,\ell}
 =\inf_{v\ne0}
 \frac{A^H_{\lambda,\ell}(v)}{\|v\|_{q_\ell}^2}.
\end{equation}

\begin{theorem}[Attainment of the Hardy quotient]
\label{thm:hardy-minimizer}
Let \(n\geq1\), let \(\ell>1\) be an integer, and assume
\(0<\lambda<(\ell-1)^2\).  Then
\[
 0<S^H_{\lambda,\ell}<S_\ell.
\]
The infimum in \eqref{eq:hardy-quotient} is attained by a nonnegative
\(v_\lambda\in\dot S^1_\ell(\Ucal)\), normalized by
\(\|v_\lambda\|_{q_\ell}=1\).  Every normalized minimizing sequence is
precompact in \(\dot S^1_\ell(\Ucal)\) and \(L^{q_\ell}\), after reduced
translations and dilations.  The minimizer satisfies
\begin{equation}\label{eq:hardy-euler}
 A^H_{\lambda,\ell}(v_\lambda,\varphi)
 =S^H_{\lambda,\ell}
 \int_{\Ucal}v_\lambda^{q_\ell-1}\varphi\dd\mu_\ell
\end{equation}
for every \(\varphi\in\dot S^1_\ell(\Ucal)\).
\end{theorem}

\begin{proof}
Lemma~\ref{lem:hardy-bottom-paper} gives
\[
 A^H_{\lambda,\ell}(v)
 \geq
 \left(1-\frac{\lambda}{(\ell-1)^2}\right)E_\ell(v),
\]
so the quotient is positive and minimizing sequences are bounded.
 A normalized pure extremal \(B_0\) has positive Hardy mass, hence
\[
 S^H_{\lambda,\ell}
 \leq E_\ell(B_0)-\lambda\mathsf H_\ell(B_0,B_0)
 <S_\ell.
\]

Apply Theorem~\ref{thm:reduced-profile} to a normalized minimizing sequence.
The inclusion
\[
 \dot S^1_\ell(\Ucal)
 \hookrightarrow L^2(\rho^{-1}\dd\mu_\ell)
\]
is bounded by Lemma~\ref{lem:hardy-bottom-paper}, and every reduced frame
action is unitary in the target space.  Framewise weak-null remainders
therefore give, at each fixed stage,
\[
 \mathsf H_\ell(v_k,v_k)
 =\sum_{j=1}^J\mathsf H_\ell(\phi^j,\phi^j)
 +\mathsf H_\ell(r_k^J,r_k^J)+o_k(1).
\]
Together with \eqref{eq:profile-energy} and
\eqref{eq:profile-mass}, this yields finite-stage splitting of
\(A^H_{\lambda,\ell}\).

Set
\[
 m_j=\|\phi^j\|_{q_\ell}^{q_\ell},
 \qquad \alpha=\frac2{q_\ell}\in(0,1).
\]
The critical-mass splitting and the ordered remainder estimate give
\[
 \sum_{j\geq1}m_j=1;
\]
in particular, vanishing is impossible.  Since the perturbed quadratic
form is nonnegative, the finite-stage splitting and the definition of the
quotient imply, for every \(J\),
\[
 S^H_{\lambda,\ell}
 \geq\sum_{j=1}^JA^H_{\lambda,\ell}(\phi^j)
 \geq S^H_{\lambda,\ell}\sum_{j=1}^Jm_j^\alpha.
\]
Letting \(J\to\infty\) gives
\[
 1\geq\sum_{j\geq1}m_j^\alpha.
\]
On the other hand,
\[
 \sum_{j\geq1}m_j^\alpha
 \geq\left(\sum_{j\geq1}m_j\right)^\alpha=1.
\]
Because \(\alpha<1\), equality is possible only when exactly one \(m_j\)
is nonzero; indeed \(m^\alpha+(1-m)^\alpha>1\) for \(0<m<1\).
Thus one profile carries the full critical mass and attains the quotient.
The perturbed energy of the remaining term tends to zero, and coercivity
then gives strong convergence in \(\dot S^1_\ell(\Ucal)\), hence also in
\(L^{q_\ell}\).  Replacing the minimizer by its absolute value does not
increase the quadratic form.  The first variation on the unit
\(L^{q_\ell}\)-sphere gives \eqref{eq:hardy-euler}.
\end{proof}

We transport this minimizer to the intrinsic space.  Set
\[
 \mu=(\ell-1)^2-\lambda>0
\]
and define
\[
 S^{\mathrm{CH}}_{\mu,\ell}
 =\inf_{U\ne0}
 \frac{\mathfrak b_\mu(U,U)}
 {\|U\|_{L^{q_\ell}(dV_g)}^2}.
\]
Theorem~\ref{thm:completed-cayley} gives
\begin{equation}\label{eq:quotient-factor}
 S^H_{\lambda,\ell}
 =4^{1/(n+\ell+1)}S^{\mathrm{CH}}_{\mu,\ell}.
\end{equation}

\begin{theorem}[Positive least-action ground state]
\label{thm:positive-ground-state}
Let \(n\geq1\), let \(\ell>1\) be an integer, and assume
\(0<\lambda<(\ell-1)^2\).  Set
\(\mu=(\ell-1)^2-\lambda\).  Then there exists
\[
 U_\lambda\in H^1(\CHyp^m,g)\cap C^\infty(\CHyp^m),
 \qquad U_\lambda>0,
\]
such that
\begin{equation}\label{eq:intrinsic-equation}
 [-\Delta_g-m^2+\mu]U_\lambda
 =U_\lambda^{q_\ell-1}.
\end{equation}
It has least action among all nonzero real critical points of
\[
 J^{\mathrm{CH}}_{\mu,\ell}(U)
 =\frac12\mathfrak b_\mu(U,U)
 -\frac1{q_\ell}\|U\|_{q_\ell}^{q_\ell}.
\]
More precisely,
\[
 J^{\mathrm{CH}}_{\mu,\ell}(U_\lambda)
 =d^{\mathrm{CH}}_{\mu,\ell}
 =\frac1{Q_\ell}
 \bigl(S^{\mathrm{CH}}_{\mu,\ell}\bigr)^{Q_\ell/2}.
\]
\end{theorem}

\begin{proof}
Let \(v_\lambda\) be the normalized minimizer in
Theorem~\ref{thm:hardy-minimizer} and set
\[
 U_\lambda^{\mathrm{qn}}=\mathcal C_\ell^{-1}v_\lambda,
 \qquad
 U_\lambda
 =\bigl(S^H_{\lambda,\ell}\bigr)^{1/(q_\ell-2)}
 U_\lambda^{\mathrm{qn}}.
\]
Equations \eqref{eq:cayley-energy}--\eqref{eq:cayley-critical} and
\eqref{eq:hardy-euler} give \eqref{eq:intrinsic-equation}.

The ordinary local Sobolev embedding in real dimension \(2n+2\) implies
\[
 U_\lambda^{q_\ell-2}+m^2-\mu
 \in L^{s_0}_{\mathrm{loc}},
 \qquad
 s_0=\frac{(n+1)(n+\ell)}{n}>\frac{2n+2}{2}.
\]
Testing truncated powers of \(U_\lambda\), using a Caccioppoli estimate and
absorbing the coefficient by Hölder and interpolation, gives the standard
local Moser iteration.  Thus \(U_\lambda\in L^\infty_{\mathrm{loc}}\).
Interior elliptic estimates give a \(C^{2,\alpha}_{\mathrm{loc}}\)
representative.  On each relatively compact ball, choose
\[
 K\geq\|(U_\lambda^{q_\ell-2}+m^2-\mu)^-\|_\infty.
\]
The classical strong maximum principle for \(\Delta_g-K\), together with
connectedness and nontriviality, yields \(U_\lambda>0\).  The right-hand
side is then smooth on compact subsets, so elliptic bootstrapping gives
\(U_\lambda\in C^\infty\).

Every nonzero real critical point \(W\) satisfies the Nehari identity
\[
 \mathfrak b_\mu(W,W)=\|W\|_{q_\ell}^{q_\ell}.
\]
The definition of \(S^{\mathrm{CH}}_{\mu,\ell}\) consequently gives
\[
 J^{\mathrm{CH}}_{\mu,\ell}(W)
 =\frac1{Q_\ell}\mathfrak b_\mu(W,W)
 \geq\frac1{Q_\ell}
 \bigl(S^{\mathrm{CH}}_{\mu,\ell}\bigr)^{Q_\ell/2}.
\]
Equality holds for the rescaled quotient minimizer.
\end{proof}

\subsection{Small-parameter asymptotics and quantitative rigidity}

We now apply the profile compactness and stability theorems to the family of
Hardy minimizers.  The reduced translation--dilation group acts on the signed unit-norm
orbit \(\mathcal M_{1,\ell}\) transitively; the estimates below are
therefore orbit-distance estimates and do not select a preferred center or
scale.

\begin{lemma}[Hardy invariance under reduced frames]\label{lem:hardy-frame}
Let \(n\geq1\) and let \(\ell>1\) be an integer.  For every reduced frame
\(\mathcal T_{h,\eta}\) and all
\(u,w\in\dot S^1_\ell(\Ucal;\R)\),
\[
 \mathsf H_\ell(\mathcal T_{h,\eta}u,\mathcal T_{h,\eta}w)
 =\mathsf H_\ell(u,w).
\]
\end{lemma}

\begin{proof}
Heisenberg translations leave both \(\rho\) and \(\dd\mu_\ell\) unchanged.
Under the dilation in \(\mathcal T_{h,\eta}\), the measure, the Hardy weight,
and the two function amplitudes contribute respectively
\(h^{Q_\ell}\), \(h^{-2}\), and \(h^{-2(n+\ell)}\).  The factors cancel
because \(Q_\ell-2-2(n+\ell)=0\).  The identity follows first on the smooth
core and then on the completed space from the Hardy inequality.
\end{proof}

\begin{proof}[Proof of Theorem~\ref{thm:intro-rigidity}]
Write \(S=S_\ell\),
\(c_{H,\ell}=(\ell-1)^2\), and
\(\lambda_0=c_{H,\ell}/2\).  Let \(v_\lambda\) be any
real normalized minimizer of \(S^H_{\lambda,\ell}\) with
\(\|v_\lambda\|_{q_\ell}=1\).

\emph{Step~1: Uniform energy bound.}
The Hardy inequality gives
\(A^H_{\lambda,\ell}(v)\geq(1-\lambda/c_{H,\ell})E_\ell(v)\).
Testing the quotient with \(B_0\) yields
\(S^H_{\lambda,\ell}\leq S-\lambda\,\mathsf H_\ell(B_0,B_0)\leq S\).
Hence
\[
 S\leq E_\ell(v_\lambda)
 \leq\frac{S^H_{\lambda,\ell}}{1-\lambda/c_{H,\ell}}
 \leq\frac{S}{1-\lambda/c_{H,\ell}}
 \leq 2S
\]
for \(0<\lambda\leq\lambda_0\).

\emph{Step~2: Defect identity.}
Set \(\delta_\lambda=E_\ell(v_\lambda)-S\geq0\) and
\(\Delta_\lambda=\mathsf H_\ell(v_\lambda,v_\lambda)
-\mathsf H_\ell(B_0,B_0)\).  The minimizing identity gives
\[
 S^H_{\lambda,\ell}
 =S+\delta_\lambda-\lambda\,\mathsf H_\ell(v_\lambda,v_\lambda).
\]
Comparing with the bubble test bound
\(S^H_{\lambda,\ell}\leq S-\lambda\,\mathsf H_\ell(B_0,B_0)\) yields
\begin{equation}\label{eq:defect-sign}
 0\leq\delta_\lambda\leq\lambda\,\Delta_\lambda.
\end{equation}
In particular \(\Delta_\lambda\geq0\), and the quotient remainder
\(r_\lambda=\delta_\lambda-\lambda\,\Delta_\lambda\leq0\).

\emph{Step~3: Cone-to-orbit comparison.}
Let \(D_\lambda=\operatorname{dist}_{\dot S^1_\ell}
(v_\lambda,\mathfrak M_\ell^{\R})\) and
\(d_\lambda^{\mathrm{orb}}=\operatorname{dist}_{\dot S^1_\ell}
(v_\lambda,\mathcal M_{1,\ell})\).
Since \(\mathcal M_{1,\ell}\subset\mathfrak M_\ell^{\R}\),
\(D_\lambda\leq d_\lambda^{\mathrm{orb}}\).
For the reverse, let \(W=A\,\mathcal T_{h,\eta}B_0\) with
\(A\in\R\).  If \(A\ne0\), set \(\sigma=\operatorname{sgn}A\); otherwise
choose either sign.  The reverse triangle inequality in \(L^{q_\ell}\) and
the sharp embedding give \(||A|-1|\leq S^{-1/2}\|v_\lambda-W\|_{\dot
S^1_\ell}\), so
\[
 \|v_\lambda-\sigma\,\mathcal T_{h,\eta}B_0\|_{\dot S^1_\ell}
 \leq 2\|v_\lambda-W\|_{\dot S^1_\ell}.
\]
Taking the infimum over the cone yields
\(d_\lambda^{\mathrm{orb}}\leq 2D_\lambda\).
Theorem~\ref{thm:intro-stability} then gives
\begin{equation}\label{eq:stability-applied}
 \frac{\kappa_{n,\ell}}{4}\,(d_\lambda^{\mathrm{orb}})^2
 \leq\kappa_{n,\ell}\,D_\lambda^2
 \leq\delta_\lambda.
\end{equation}

\emph{Step~4: Hardy continuity.}
The polarized Hardy form satisfies
\(|\mathsf H_\ell(u,w)|\leq c_{H,\ell}^{-1}\|u\|_{\dot S^1_\ell}
\|w\|_{\dot S^1_\ell}\) by Cauchy--Schwarz and the Hardy inequality.  For
any approximate orbit element
\(W_\varepsilon\in\mathcal M_{1,\ell}\) with
\(\|v_\lambda-W_\varepsilon\|_{\dot S^1_\ell}
\leq d_\lambda^{\mathrm{orb}}+\varepsilon\),
the frame invariance and sign invariance of the Hardy form give
\(\mathsf H_\ell(W_\varepsilon,W_\varepsilon)=\mathsf H_\ell(B_0,B_0)\)
and \(\|W_\varepsilon\|_{\dot S^1_\ell}=\sqrt{S}\).
Thus
\[
 |\Delta_\lambda|
 \leq\frac{\sqrt{2S}+\sqrt{S}}{c_{H,\ell}}\,
 d_\lambda^{\mathrm{orb}}
 =:C_H\,d_\lambda^{\mathrm{orb}},
\]
after sending \(\varepsilon\downarrow0\).

\emph{Step~5: Closure.}
Combining \eqref{eq:defect-sign}, \eqref{eq:stability-applied}, and the
Hardy continuity bound gives
\[
 \frac{\kappa_{n,\ell}}{4}\,(d_\lambda^{\mathrm{orb}})^2
 \leq\delta_\lambda
 \leq C_H\lambda\,d_\lambda^{\mathrm{orb}}.
\]
If \(d_\lambda^{\mathrm{orb}}>0\), division yields
\(d_\lambda^{\mathrm{orb}}\leq
4C_H\lambda/\kappa_{n,\ell}\).  If
\(d_\lambda^{\mathrm{orb}}=0\), then
\(\Delta_\lambda=0\) and \(\delta_\lambda=0\), so the same conclusion holds
without division.  Substituting back gives
\[
 0\leq\delta_\lambda\leq\frac{4C_H^2}{\kappa_{n,\ell}}\lambda^2,
 \qquad
 -\frac{4C_H^2}{\kappa_{n,\ell}}\lambda^2
 \leq r_\lambda\leq0.
\]
All constants depend only on \((n,\ell)\), and no choice among normalized
minimizers enters.

It remains to pass from normalized minimizers to coefficient-one solutions.
Set
\[
 S_\lambda=S^H_{\lambda,\ell},\qquad
 \theta_\ell=\frac1{q_\ell-2},\qquad
 \bar v_\lambda=S_\lambda^{\theta_\ell}v_\lambda,
 \qquad
 \widehat B_0=S^{\theta_\ell}B_0.
\]
Then \(\bar v_\lambda\) satisfies
\[
 \Gell_\ell\bar v_\lambda
 -\lambda\rho^{-1}\bar v_\lambda
 =|\bar v_\lambda|^{q_\ell-2}\bar v_\lambda,
\]
whereas \(\widehat B_0\) is the coefficient-one pure bubble.  If
\[
 \widehat{\mathcal M}_{1,\ell}
 =\{\sigma\mathcal T_{h,\eta}\widehat B_0:
 \sigma\in\{-1,1\},\ h>0,\ \eta\in\Heis^n\},
\]
then comparison with approximate elements of \(\mathcal M_{1,\ell}\) gives
\begin{equation}\label{eq:coefficient-one-distance}
 \operatorname{dist}_{\dot S^1_\ell}
 \bigl(\bar v_\lambda,\widehat{\mathcal M}_{1,\ell}\bigr)
 \leq S_\lambda^{\theta_\ell}d_\lambda^{\mathrm{orb}}
 +\sqrt S\,|S_\lambda^{\theta_\ell}-S^{\theta_\ell}|.
\end{equation}
The quotient bounds used in Step~1 also imply
\[
 \left(1-\frac{\lambda}{c_{H,\ell}}\right)S
 \leq S_\lambda\leq S,
 \qquad
 0\leq S-S_\lambda\leq\frac S{c_{H,\ell}}\lambda.
\]
For \(0<\lambda\leq c_{H,\ell}/2\), the map
\(x\mapsto x^{\theta_\ell}\) has a bounded derivative on \([S/2,S]\).
Combining this observation, \eqref{eq:coefficient-one-distance}, and the
already proved estimate for \(d_\lambda^{\mathrm{orb}}\) gives
\[
 \operatorname{dist}_{\dot S^1_\ell}
 \bigl(\bar v_\lambda,\widehat{\mathcal M}_{1,\ell}\bigr)
 =O_{n,\ell}(\lambda).
\]

Finally, Taylor expansion of \(Q_\ell^{-1}x^{Q_\ell/2}\) on \([S/2,S]\),
together with
\(S_\lambda=S-\lambda\mathsf H_\ell(B_0,B_0)+O_{n,\ell}(\lambda^2)\),
yields
\[
 d_{\lambda,\ell}
 =\frac1{Q_\ell}S_\lambda^{Q_\ell/2}
 =\frac1{Q_\ell}S^{Q_\ell/2}
 -\frac12S^{Q_\ell/2-1}\mathsf H_\ell(B_0,B_0)\lambda
 +O_{n,\ell}(\lambda^2).
\]
The nonpositive remainder established above belongs to the quotient itself,
and does not transfer to this last expansion.  Finally,
\eqref{eq:cayley-energy}--\eqref{eq:cayley-critical} give
\(J^H_{\lambda,\ell}(\mathcal C_\ell U)=4J^{\mathrm{CH}}_{\mu,\ell}(U)\),
so the intrinsic action quantum is \(d_{\lambda,\ell}/4\).
\end{proof}

\subsection{Normal deformation and the second-order quotient}

We first justify the solvability condition and the strict sign of the
coefficient introduced in \eqref{eq:intro-first-correction}.

\begin{lemma}[The first normal correction]
\label{lem:first-normal-correction}
The right-hand side of \eqref{eq:intro-first-correction} annihilates
\(\R B_0\oplus\mathcal Z_\ell\).  Consequently there is a unique
\(w_{1,\ell}\in Y_\ell\) satisfying that equation for every
\(\phi\in\dot S^1_\ell(\Ucal;\R)\), and
\[
 c_{2,\ell}
 =\mathsf H_\ell(B_0,w_{1,\ell})
 =\mathcal Q_{B_0}(w_{1,\ell},w_{1,\ell})>0.
\]
\end{lemma}

\begin{proof}
Differentiating the critical-norm and Hardy invariances along the reduced
orbit through \(B_0\) gives, for every \(Z\in\mathcal Z_\ell\),
\[
 \int_{\Ucal}B_0^{q_\ell-1}Z\dd\mu_\ell=0,
 \qquad
 \mathsf H_\ell(B_0,Z)=0.
\]
For \(\phi=B_0\), the same functional vanishes because
\(\|B_0\|_{q_\ell}=1\).  Corollary~\ref{cor:normal-inverse} therefore
produces a unique solution in \(Y_\ell\).  Since
\(\mathcal Q_{B_0}(w,B_0)=0\) for \(w\in Y_\ell\) and
\(\mathcal Z_\ell=\ker L_{B_0}\), the equation then holds against the
whole energy-orthogonal decomposition
\[
 \dot S^1_\ell(\Ucal;\R)
 =\R B_0\oplus_E\mathcal Z_\ell\oplus_EY_\ell.
\]
Testing it with \(w_{1,\ell}\) gives the two formulas for
\(c_{2,\ell}\), and Theorem~\ref{thm:normal-spectral-gap} gives
\(c_{2,\ell}\geq0\).

If equality held, the same spectral gap would force \(w_{1,\ell}=0\).
The forcing functional would then vanish on all of
\(\dot S^1_\ell(\Ucal;\R)\), and hence
\[
 \rho^{-1}B_0=H_{0,\ell}B_0^{q_\ell-1}
\]
in the distributional sense.  This is impossible.  Indeed, for the
centered bubble
\(B_0=C[(1+\rho+|z|^2)^2+t^2]^{-(n+\ell)/2}\), the asserted identity
would require \([(1+\rho+|z|^2)^2+t^2]/\rho\) to be constant.  Thus
\(c_{2,\ell}>0\).
\end{proof}

\begin{proof}[Proof of Theorem~\ref{thm:intro-second-order}]
Write
\[
 q=q_\ell,\qquad S=S_\ell,\qquad B=B_0,\qquad
 H_0=H_{0,\ell},\qquad w_1=w_{1,\ell},\qquad c_2=c_{2,\ell}.
\]
By \eqref{eq:intro-subcubic-range}, \(2<q<3\).
Theorem~\ref{thm:intro-rigidity} and
Lemma~\ref{lem:normal-slice} give signs and reduced frames such that
\begin{equation}\label{eq:hardy-modulation}
 u_\lambda
 :=\mathcal T_{g_\lambda}^{-1}(\sigma_\lambda v_\lambda)
 =B+r_\lambda,
 \qquad
 r_\lambda\perp_E\mathcal Z_\ell,
 \qquad
 \|r_\lambda\|_{\dot S^1_\ell}=O(\lambda).
\end{equation}
The Hardy quotient and the critical norm are invariant under these
operations; hence
\[
 \|B+r_\lambda\|_q=1
\]
and
\begin{equation}\label{eq:normalized-hardy-equation}
 E_\ell(B+r_\lambda,\phi)
 -\lambda\mathsf H_\ell(B+r_\lambda,\phi)
 =S^H_{\lambda,\ell}
 \int_{\Ucal}|B+r_\lambda|^{q-2}(B+r_\lambda)\phi\dd\mu_\ell
\end{equation}
for every \(\phi\in\dot S^1_\ell(\Ucal;\R)\).

The second-order expansion of the normalization gives
\begin{equation}\label{eq:normalization-expansion}
\begin{aligned}
 q\int_{\Ucal}B^{q-1}r_\lambda\dd\mu_\ell
 &+\frac{q(q-1)}2
 \int_{\Ucal}B^{q-2}r_\lambda^2\dd\mu_\ell\\
 &=O\!\left(\|r_\lambda\|_{\dot S^1_\ell}^q\right).
\end{aligned}
\end{equation}
Consequently,
\[
 a_\lambda
 :=\int_{\Ucal}B^{q-1}r_\lambda\dd\mu_\ell
 =O(\lambda^2).
\]
Writing
\[
 r_\lambda=a_\lambda B+y_\lambda,
\]
we have \(y_\lambda\in Y_\ell\).

For \(2<q<3\), the critical Nemytskii map satisfies
\begin{equation}\label{eq:nemytskii-remainder}
\begin{aligned}
 \bigl\|
 |B+r|^{q-2}(B+r)-B^{q-1}
 -(q-1)B^{q-2}r
 \bigr\|_{q'}
 \leq C\|r\|_q^{q-1}.
\end{aligned}
\end{equation}
Subtract the equation for \(B\) from
\eqref{eq:normalized-hardy-equation} and test with \(\phi\in Y_\ell\).
The term
\((S^H_{\lambda,\ell}-S)\int B^{q-1}\phi\) vanishes, and
Theorem~\ref{thm:intro-rigidity},
\eqref{eq:hardy-modulation}, and
\eqref{eq:nemytskii-remainder} yield
\[
 \left|
 \mathcal Q_B(y_\lambda,\phi)
 -\lambda\mathsf H_\ell(B,\phi)
 \right|
 \leq C\lambda^{q-1}\|\phi\|_{\dot S^1_\ell}.
\]
For \(\phi\in Y_\ell\), equation
\eqref{eq:intro-first-correction} reads
\[
 \mathcal Q_B(w_1,\phi)=\mathsf H_\ell(B,\phi).
\]
The exact normal coercivity
\eqref{eq:exact-normal-gap} therefore gives
\[
 \|y_\lambda-\lambda w_1\|_{\dot S^1_\ell}
 \leq C\lambda^{q-1}.
\]
Since \(a_\lambda=O(\lambda^2)\) and \(q-1<2\), this proves
\eqref{eq:intro-modulated-shape}.

We next expand the quotient.  The scalar Taylor estimate
\[
\begin{aligned}
 \bigl|
 |a+b|^q-|a|^q-q|a|^{q-2}ab
 -\tfrac{q(q-1)}2|a|^{q-2}b^2
 \bigr|
 \leq C|b|^q
\end{aligned}
\]
and \eqref{eq:normalization-expansion} imply
\[
 E_\ell(B+r_\lambda)-S
 =\mathcal Q_B(r_\lambda,r_\lambda)
 +O\!\left(\|r_\lambda\|_{\dot S^1_\ell}^q\right).
\]
Using the shape expansion and
\[
 c_2
 =\mathcal Q_B(w_1,w_1)
 =\mathsf H_\ell(B,w_1),
\]
we obtain
\[
 E_\ell(B+r_\lambda)-S
 =c_2\lambda^2+O(\lambda^q)
\]
and
\[
 \mathsf H_\ell(B+r_\lambda,B+r_\lambda)
 =H_0+2c_2\lambda+O(\lambda^{q-1}).
\]
The exact minimizing identity now gives
\[
 S^H_{\lambda,\ell}
 =E_\ell(B+r_\lambda)
 -\lambda\mathsf H_\ell(B+r_\lambda,B+r_\lambda)
 =S-\lambda H_0-\lambda^2c_2+O(\lambda^q),
\]
which is \eqref{eq:intro-second-order-quotient}.

Finally, Taylor expansion of
\(Q_\ell^{-1}x^{Q_\ell/2}\) at \(S\) gives
\[
\begin{aligned}
 d_{\lambda,\ell}
={}&\frac1{Q_\ell}S^{Q_\ell/2}
 -\frac12S^{Q_\ell/2-1}H_0\lambda\\
&+\left[
 -\frac12S^{Q_\ell/2-1}c_2
 +\frac{Q_\ell-2}{8}S^{Q_\ell/2-2}H_0^2
 \right]\lambda^2
 +O(\lambda^q).
\end{aligned}
\]
\end{proof}

\begin{corollary}[Coefficient-one first correction]
\label{cor:coefficient-one-correction}
Let
\[
 \theta_\ell=\frac1{q_\ell-2},\qquad
 \widehat B_0=S_\ell^{\theta_\ell}B_0,
\]
and define
\[
 \widehat w_{1,\ell}
 =S_\ell^{\theta_\ell}w_{1,\ell}
 -\theta_\ell H_{0,\ell}
 S_\ell^{\theta_\ell-1}B_0.
\]
Then
\[
 \mathcal T_{g_\lambda}^{-1}
 \left(\sigma_\lambda
 (S^H_{\lambda,\ell})^{\theta_\ell}v_\lambda\right)
 =\widehat B_0+\lambda\widehat w_{1,\ell}
 +O_{\dot S^1_\ell}(\lambda^{q_\ell-1}),
\]
and
\[
 \Gell_\ell\widehat w_{1,\ell}
 -(q_\ell-1)\widehat B_0^{q_\ell-2}\widehat w_{1,\ell}
 =\rho^{-1}\widehat B_0.
\]
\end{corollary}

\begin{proof}
The quotient expansion gives
\[
 (S^H_{\lambda,\ell})^{\theta_\ell}
 =S_\ell^{\theta_\ell}
 -\theta_\ell H_{0,\ell}
 S_\ell^{\theta_\ell-1}\lambda+O(\lambda^2).
\]
Combining this identity with
\eqref{eq:intro-modulated-shape} proves the first assertion.  Moreover,
\[
 \mathcal Q_{B_0}(B_0,\phi)
 =-(q_\ell-2)S_\ell
 \int_{\Ucal}B_0^{q_\ell-1}\phi\dd\mu_\ell.
\]
Substitution of \eqref{eq:intro-first-correction} into the definition of
\(\widehat w_{1,\ell}\) cancels the normalization term and gives the
displayed coefficient-one equation.
\end{proof}

\subsection{Local uniqueness and intrinsic radiality}

\begin{proposition}[Unique normalized branch in the normal slice]
\label{prop:unique-hardy-branch}
For all sufficiently small positive \(\lambda\), the normalized Hardy
equation
\[
 \Gell_\ell u-\lambda\rho^{-1}u
 =s|u|^{q_\ell-2}u
\]
has exactly one solution \((u,s)\) near \((B_0,S_\ell)\) satisfying
\[
 \|u\|_{q_\ell}=1,
 \qquad
 E_\ell(u-B_0,Z)=0\quad(Z\in\mathcal Z_\ell).
\]
\end{proposition}

\begin{proof}
Let \(R_E:\dot S^1_\ell\to(\dot S^1_\ell)^*\) be the Riesz map induced by
the energy, and choose a basis \(Z_1,\ldots,Z_{2n+2}\) of
\(\mathcal Z_\ell\).  Introduce the augmented map
\[
\begin{aligned}
 \mathbf F(\lambda,u,s,\mathbf b)
 =\Bigg(&R_E^{-1}\!\left[
 E_\ell(u,\cdot)-\lambda\mathsf H_\ell(u,\cdot)
 -s\int_{\Ucal}|u|^{q_\ell-2}u(\cdot)\dd\mu_\ell
 \right]
 -\sum_{a=1}^{2n+2}b_aZ_a,\\
 &\int_{\Ucal}|u|^{q_\ell}\dd\mu_\ell-1,\,
 \bigl(E_\ell(u-B_0,Z_a)\bigr)_{a=1}^{2n+2}
 \Bigg).
\end{aligned}
\]
At \((0,B_0,S_\ell,0)\), its derivative in
\((u,s,\mathbf b)\) sends \((w,\sigma,\mathbf c)\) to
\[
\begin{aligned}
 \Bigg(&A_{B_0}w-\frac{\sigma}{S_\ell}B_0
 -\sum_{a=1}^{2n+2}c_aZ_a,\,
 \frac{q_\ell}{S_\ell}E_\ell(B_0,w),\,
 \bigl(E_\ell(w,Z_a)\bigr)_{a=1}^{2n+2}\Bigg).
\end{aligned}
\]
The energy-orthogonal decomposition
\[
 \dot S^1_\ell
 =\R B_0\oplus_E\mathcal Z_\ell\oplus_EY_\ell
\]
and Corollary~\ref{cor:normal-inverse} show that this derivative is an
isomorphism.  The critical Nemytskii map is
\(C^{1,q_\ell-2}:L^{q_\ell}\to L^{q_\ell'}\), so the implicit function
theorem gives a unique nearby augmented branch.

A modulated normalized minimizer exists and solves the augmented equation
with \(\mathbf b=0\).  Uniqueness therefore forces the augmented branch
to have \(\mathbf b=0\).  Any normalized solution of the original equation
in the same slice is another augmented zero with \(\mathbf b=0\), and must
coincide with this branch.
\end{proof}

\begin{proof}[Proof of the uniqueness and symmetry assertions in
Theorem~\ref{thm:intro-symmetry}]
Let \(o=(0,0,1)\) in Siegel coordinates and let
\(K_o\simeq U(n+1)\) be its stabilizer.  With \(N=n+\ell\) and
\[
 D=(1+\rho+|z|^2)^2+t^2,
\]
the curvature-\(-4\) distance formula is
\[
 \cosh^2d_g(o,x)=\frac{D}{4\rho}.
\]
The inverse Cayley image of the coefficient-one centered bubble is
\[
 U_0(x)
 =\rho^{N/2}\widehat B_0(x)
 =C(\cosh d_g(o,x))^{-N}.
\]
It is therefore fixed by \(K_o\), and the same is true of \(B_0\).

The Iwasawa subgroup
\[
 AN\simeq\Heis^n\rtimes\R_+
\]
is the reduced translation--dilation group and, as in the harmonic
\(AN\) picture of \cite{AnkerDamekYacoub1996HarmonicAN}, acts simply
transitively
on \(\CHyp^{n+1}\).  If \(G\) denotes the identity component of the
holomorphic isometry group and \(g\in G\), choose \(\gamma\in AN\) with
\(\gamma o=go\).  Then \(\gamma^{-1}g\in K_o\), and hence the full
\(G\)-orbit of \(B_0\) equals its reduced \(AN\)-orbit.

By Proposition~\ref{prop:transported-action}, the transported \(G\)-action
preserves the energy, the Hardy form, the critical norm, and the equation.
Since \(K_o\) fixes \(B_0\), it preserves
\(\mathcal Z_\ell=T_{B_0}(G\cdot B_0)\) and the normal slice.  If
\(u_\lambda\) is the centered solution furnished by
Proposition~\ref{prop:unique-hardy-branch}, then
\(\widetilde\Pi_ku_\lambda\) is another solution in the same slice for
every \(k\in K_o\).  Local uniqueness gives
\[
 \widetilde\Pi_ku_\lambda=u_\lambda.
\]
The completed Cayley correspondence and
Theorem~\ref{thm:positive-ground-state} supply a smooth intrinsic
representative, and
transitivity of \(K_o\) on geodesic spheres turns this invariance into
geodesic radiality.  Finally, every minimizer may be placed in the normal
slice by Lemma~\ref{lem:normal-slice}; local uniqueness and the equality
of the \(G\)- and \(AN\)-orbits prove the orbit assertion.
\end{proof}

\subsection{Nondegeneracy of the small-coupling ground state}

We now show that the pure-bubble nondegeneracy persists along the
small-coupling branch.  This assertion is modulo the full transported
holomorphic isometry group, rather than only modulo the
translation--dilation subgroup used to impose the normal gauge.

Put
\[
 N=n+\ell,\qquad q=q_\ell=2+\frac2N,\qquad
 \theta=(q-2)^{-1}=\frac N2.
\]
Let \(v_\lambda\) be the centered positive normalized minimizer furnished
by Proposition~\ref{prop:unique-hardy-branch}, and set
\[
 S_\lambda=S^H_{\lambda,\ell},\qquad
 \bar v_\lambda=S_\lambda^\theta v_\lambda.
\]
Thus
\[
 \Gell_\ell\bar v_\lambda-\lambda\rho^{-1}\bar v_\lambda
 =\bar v_\lambda^{q-1}.
\]
Define
\begin{equation}\label{eq:small-hardy-linearized-form}
 \begin{aligned}
 \mathcal Q_\lambda(\phi,\psi)
 ={}&E_\ell(\phi,\psi)-\lambda\mathsf H_\ell(\phi,\psi)\\
 &-(q-1)\int_{\Ucal}
 \bar v_\lambda^{q-2}\phi\psi\dd\mu_\ell,
 \end{aligned}
\end{equation}
and let \(A_\lambda\) be its bounded self-adjoint Riesz representative:
\[
 E_\ell(A_\lambda\phi,\psi)=\mathcal Q_\lambda(\phi,\psi).
\]

\begin{theorem}[Small-coupling ground-state nondegeneracy]
\label{thm:small-hardy-nondegeneracy}
Let \(n\geq1\), and let \(\ell>1\) be an integer.  Then there exist
\[
 0<\lambda_*<(\ell-1)^2,\qquad c_*>0,
\]
depending only on \((n,\ell)\), such that, whenever
\[
 0<\lambda<\lambda_*,
\]
one has
\begin{equation}\label{eq:small-hardy-kernel}
 \ker A_\lambda
 =T_{\bar v_\lambda}(G\cdot\bar v_\lambda),
 \qquad
 \dim\ker A_\lambda=2n+2,
\end{equation}
where \(G=\operatorname{Isom}_0(\CHyp^{n+1},g)\) acts through
Proposition~\ref{prop:transported-action}.  Moreover,
\begin{equation}\label{eq:small-hardy-uniform-inverse}
 \|A_\lambda\phi\|_{\dot S^1_\ell}
 \geq c_*\|\phi\|_{\dot S^1_\ell}
 \qquad
 \bigl(\phi\perp_E\ker A_\lambda\bigr).
\end{equation}
One may take \(c_*=1/(N+4)\) after decreasing \(\lambda_*\).

Equivalently, if
\[
 U_\lambda=\mathcal C_\ell^{-1}\bar v_\lambda,
\]
then the weak \(H^1\)-kernel of
\[
 \mathcal L_\lambda^{\mathrm{CH}}
 =-\Delta_g-(n+1)^2+(\ell-1)^2-\lambda
 -(q-1)U_\lambda^{q-2}
\]
is \(T_{U_\lambda}(G\cdot U_\lambda)\).  Here the weak kernel consists
of those \(Z\in H^1(\CHyp^{n+1})\) for which the linearized form vanishes
against every \(H^1\) test function.
\end{theorem}

\begin{proof}
For \(w\in\dot S^1_\ell(\Ucal)\), define bounded self-adjoint operators
\(P,K_w\) by
\[
 E_\ell(P\phi,\psi)=\mathsf H_\ell(\phi,\psi),\qquad
 E_\ell(K_w\phi,\psi)
 =\int_{\Ucal}|w|^{q-2}\phi\psi\dd\mu_\ell.
\]
Thus
\[
 A_\lambda=I-\lambda P-(q-1)K_{\bar v_\lambda}.
\]
The Hardy inequality gives \(\|P\|\leq(\ell-1)^{-2}\).
Since \(q-2=2/N\in(0,1)\), the scalar H\"older estimate and
\(q/(q-2)=N+1\) give
\[
 \bigl\||w|^{q-2}-|z|^{q-2}\bigr\|_{N+1}
 \leq\|w-z\|_q^{q-2}.
\]
H\"older's inequality and the sharp Sobolev embedding therefore yield
\begin{equation}\label{eq:hardy-linearization-holder}
 \|K_w-K_z\|_{\mathcal B(\dot S^1_\ell)}
 \leq S_\ell^{-1}\|w-z\|_q^{q-2}.
\end{equation}
The centered shape estimate
\eqref{eq:intro-modulated-shape} and the quotient expansion
\eqref{eq:intro-second-order-quotient} imply
\[
 \|\bar v_\lambda-\widehat B_0\|_{\dot S^1_\ell}
 =O(\lambda).
\]
Consequently, set
\[
 A_0:=I-(q-1)K_{\widehat B_0}=A_{B_0},
\]
where the last equality uses the normalized-bubble convention of
Section~\ref{sec:nondegeneracy}.  Then
\begin{equation}\label{eq:hardy-linearization-norm-convergence}
 \|A_\lambda-A_0\|_{\mathcal B(\dot S^1_\ell)}
 \leq C\lambda^{q-2}=o(1).
\end{equation}

By Theorem~\ref{thm:fixed-nondegeneracy}, zero is an eigenvalue of
\(A_0\) of multiplicity \(2n+2\).  Corollary~\ref{cor:normal-inverse}
shows that the remainder of its spectrum has distance at least
\[
 \delta_0=\frac2{N+4}
\]
from zero.  Let \(H_{\mathbb C}\) be the complexification of
\(\dot S^1_\ell(\Ucal;\mathbb R)\), and let \(A_{\lambda,\mathbb C}\)
be the complex-linear extension.  The norm convergence
\eqref{eq:hardy-linearization-norm-convergence} and the resolvent identity
show that, for small \(\lambda\), the Riesz projection
\[
 \mathbb P_{\lambda,\mathbb C}^0
 =\frac1{2\pi i}\int_{|z|=\delta_0/2}
 (z-A_{\lambda,\mathbb C})^{-1}\dd z
\]
has complex rank \(2n+2\).  It commutes with conjugation and is the
complexification of the real spectral projection
\[
 \mathbb P_\lambda^0
 =\mathbf 1_{(-\delta_0/2,\delta_0/2)}(A_\lambda).
\]
Hence
\begin{equation}\label{eq:small-hardy-upper-nullity}
 \dim_{\mathbb R}\ker A_\lambda\leq2n+2.
\end{equation}

It remains to produce all the zero modes in the form domain.  By the
radiality established in the preceding proof, the intrinsic ground state
is radial:
\[
 U_\lambda(x)=f_\lambda(r),\qquad r=d_g(o,x).
\]
Set \(m=n+1\) and
\[
 \mu_\lambda=(\ell-1)^2-\lambda>0.
\]
For holomorphic sectional curvature \(-4\), the radial Laplacian and
Jacobian are
\[
 \Delta_gf=f''+a_m(r)f',\qquad
 a_m(r)=(2m-1)\coth r+\tanh r,
\]
\[
 J_m(r)=(\sinh r)^{2m-1}\cosh r
\]
up to a fixed positive factor.  Thus
\begin{equation}\label{eq:small-hardy-radial-ode}
 -f_\lambda''-a_mf_\lambda'
 +(\mu_\lambda-m^2)f_\lambda=f_\lambda^{q-1}.
\end{equation}
With \(y_\lambda=J_m^{1/2}f_\lambda\), this becomes
\[
 -y_\lambda''
 +\bigl(\mu_\lambda+V_m(r)-f_\lambda^{q-2}\bigr)y_\lambda=0,
\]
where
\[
 V_m(r)=\frac12a_m'(r)+\frac14a_m(r)^2-m^2
 \longrightarrow0
\]
exponentially.  Since \(U_\lambda\in H^1\),
\(y_\lambda\in H^1(1,\infty)\).  The coefficient in the last equation
tends to \(\mu_\lambda>0\); the one-dimensional exponential dichotomy,
with its growing solution excluded by \(y_\lambda\in L^2\), gives, for
every \(0<\varepsilon<\sqrt{\mu_\lambda}\),
\begin{equation}\label{eq:small-hardy-radial-decay}
 |f_\lambda(r)|+|f_\lambda'(r)|+|f_\lambda''(r)|
 \leq C_\varepsilon
 e^{-(m+\sqrt{\mu_\lambda}-\varepsilon)r}.
\end{equation}

Let \(X\in\mathfrak p\simeq T_o\CHyp^m\), put
\[
 g_s=\exp_G(sX),\qquad p_s=g_so,\qquad
 r_s(x)=d_g(p_s,x),
\]
and let \(\Gamma_{s,x}\) be the affinely parametrized geodesic from
\(p_s\) to \(x\).  The center-moving field
\[
 \mathcal J_{X,x}
 =\left.\partial_s\right|_{s=0}\Gamma_{s,x}
\]
is a Jacobi field with
\(\mathcal J_{X,x}(0)=X\) and \(\mathcal J_{X,x}(1)=0\).
Set
\[
 \omega_X(x)=\left.\partial_s\right|_{s=0}r_s(x).
\]
The first-variation formula gives
\[
 \omega_X(x)=-\langle X,\vartheta(x)\rangle,\qquad
 |\omega_X(x)|\leq|X|,
\]
where \(\vartheta(x)=\dot\gamma_{o,x}(0)\).

For \(Y\in T_x\CHyp^m\), let \(\mathcal K_Y\) be the endpoint-moving
Jacobi field along \(\gamma_{o,x}\), with
\[
 \mathcal K_Y(0)=0,\qquad
 \mathcal K_Y(r)=Y^\perp,
\]
and put \(\eta_Y=\nabla_t\mathcal K_Y(0)\).  Then
\[
 d\vartheta_x[Y]=\eta_Y.
\]
The sectional curvatures lie in \([-4,-1]\), so Rauch comparison gives
\[
 \sinh r\,|\eta_Y|
 \leq|\mathcal K_Y(r)|
 \leq\frac12\sinh(2r)|\eta_Y|.
\]
It follows that
\begin{equation}\label{eq:small-hardy-angular-gradient}
 |\nabla\omega_X(x)|
 \leq\frac{|X|}{\sinh r}
 \leq C_m|X|(1+\coth r).
\end{equation}

Since
\[
 U_\lambda\circ g_s^{-1}(x)=f_\lambda(r_s(x)),
\]
the pointwise orbit derivative is
\[
 Z_{\lambda,X}=f_\lambda'(r)\omega_X.
\]
Equations \eqref{eq:small-hardy-radial-decay} and
\eqref{eq:small-hardy-angular-gradient} give
\[
 |Z_{\lambda,X}|\leq|X||f_\lambda'|,
\]
\[
 |\nabla Z_{\lambda,X}|
 \leq C_m|X|
 \bigl(|f_\lambda''|+(1+\coth r)|f_\lambda'|\bigr).
\]
Smooth radiality gives \(f_\lambda'(r)=O(r)\) at the origin, while
\eqref{eq:small-hardy-radial-decay} is square integrable at infinity
against \(J_m(r)\dd r\).  Hence
\[
 Z_{\lambda,X}\in H^1(\CHyp^m).
\]

To justify differentiation in the global energy topology, write
\(\pi(g)U=U\circ g^{-1}\).  This is a strongly continuous isometric
representation on \(H^1\).  The pointwise fundamental theorem of calculus
gives first distributionally, and then as a Bochner identity in \(H^1\),
\begin{equation}\label{eq:small-hardy-group-integral}
 \pi(g_s)U_\lambda-U_\lambda
 =\int_0^s\pi(g_t)Z_{\lambda,X}\dd t.
\end{equation}
Therefore
\[
 \frac{\pi(g_s)U_\lambda-U_\lambda}{s}
 \longrightarrow Z_{\lambda,X}
 \quad\text{in }H^1.
\]
Differentiating the invariant weak equation is now legitimate and shows
that \(Z_{\lambda,X}\) belongs to the weak kernel of
\(\mathcal L_\lambda^{\mathrm{CH}}\).
If \(\mathcal Q_\lambda^{\mathrm{CH}}\) denotes the intrinsic linearized
form, the completed Cayley identities give
\[
 \mathcal Q_\lambda
 \bigl(\mathcal C_\ell Z,\mathcal C_\ell\Psi\bigr)
 =4\mathcal Q_\lambda^{\mathrm{CH}}(Z,\Psi)
 \qquad(Z,\Psi\in H^1(\CHyp^{n+1})).
\]
Hence \(\mathcal C_\ell Z_{\lambda,X}\in\ker A_\lambda\).

The \(2m=2n+2\) intrinsic modes are linearly independent.  Indeed,
\(f_\lambda\) is not constant, and on a geodesic sphere on which
\(f_\lambda'\neq0\), the functions \(\omega_X\) form the real first
isotropy representation.  Their Cayley images remain linearly independent
because \(\mathcal C_\ell\) is an energy-space isomorphism.  Together with
\eqref{eq:small-hardy-upper-nullity}, this proves
\eqref{eq:small-hardy-kernel}.  Since the rest of the spectrum lies
outside \((-\delta_0/2,\delta_0/2)\), the spectral theorem gives
\eqref{eq:small-hardy-uniform-inverse} with
\[
 c_*=\frac{\delta_0}{2}=\frac1{N+4}.
\]
\end{proof}

\begin{corollary}[Uniform constrained coercivity]
\label{cor:small-hardy-constrained-coercivity}
After decreasing \(\lambda_*\), one has
\[
 \mathcal Q_\lambda(\phi,\phi)
 \geq\frac1{2(N+4)}E_\ell(\phi)
\]
whenever
\[
 \int_{\Ucal}v_\lambda^{q-1}\phi\dd\mu_\ell=0,\qquad
 \phi\perp_E\ker A_\lambda.
\]
In particular, the coefficient-one ground state has Morse index one and
nullity \(2n+2\).
\end{corollary}

\begin{proof}
The minimizing property makes \(\mathcal Q_\lambda\) nonnegative on the
tangent space to the normalized \(L^q\)-sphere, whereas
\[
 \mathcal Q_\lambda(\bar v_\lambda,\bar v_\lambda)
 =(2-q)\|\bar v_\lambda\|_q^q<0.
\]
Thus the negative index of the form is one.

Let \(\mathbb P_\lambda^-\) be the rank-one spectral projection around
\(-(q-2)\), retain \(\mathbb P_\lambda^0\) for the kernel projection,
and put
\[
 \mathbb P_\lambda^+
 =I-\mathbb P_\lambda^--\mathbb P_\lambda^0.
\]
The separated spectral projections converge in operator norm to their
\(\lambda=0\) counterparts.  Choose a unit vector
\(e_\lambda\in\operatorname{Ran}\mathbb P_\lambda^-\), with its sign fixed
so that
\[
 e_\lambda\longrightarrow
 e_0=\frac{B_0}{\|B_0\|_{\dot S^1_\ell}},
\]
and write
\[
 A_\lambda e_\lambda=-\nu_\lambda e_\lambda,\qquad
 \nu_\lambda\longrightarrow q-2.
\]
On the positive spectral subspace,
\[
 \mathcal Q_\lambda(w,w)
 \geq\frac{\delta_0}{2}\|w\|_{\dot S^1_\ell}^2.
\]

Define
\[
 \mathfrak c_\lambda(\phi)
 =\int_{\Ucal}v_\lambda^{q-1}\phi\dd\mu_\ell.
\]
The continuity of
\(v\mapsto|v|^{q-2}v:L^q\to L^{q'}\) gives
\(\mathfrak c_\lambda\to\mathfrak c_0\) in the dual energy space.
Moreover,
\[
 \mathfrak c_0(e_0)\neq0,\qquad
 \mathfrak c_0|_{\operatorname{Ran}\mathbb P_0^+}=0.
\]
It follows that there is \(\varepsilon_\lambda\to0\) such that, whenever
\[
 \phi\perp_E\ker A_\lambda,\qquad
 \mathfrak c_\lambda(\phi)=0,
\]
the unique decomposition
\[
 \phi=ae_\lambda+w,\qquad
 w\in\operatorname{Ran}\mathbb P_\lambda^+,
\]
satisfies
\[
 |a|\leq\varepsilon_\lambda\|w\|_{\dot S^1_\ell}.
\]
Since \(\nu_\lambda\) remains bounded, decreasing \(\lambda_*\) gives
\[
 \begin{aligned}
 \mathcal Q_\lambda(\phi,\phi)
 &=-\nu_\lambda a^2+\mathcal Q_\lambda(w,w)\\
 &\geq\frac{\delta_0}{4}
 \|\phi\|_{\dot S^1_\ell}^2
 =\frac1{2(N+4)}
 \|\phi\|_{\dot S^1_\ell}^2.
 \end{aligned}
\]
\end{proof}

\section{Global compactness for the invariant Hardy perturbation}
\label{sec:hardy-global}

Retain the assumptions of Section~\ref{sec:hardy} and write
\[
 H_\ell=\dot S^1_\ell(\Ucal;\R),\qquad
 J^H_{\lambda,\ell}(v)
 =\frac12A^H_{\lambda,\ell}(v)
 -\frac1{q_\ell}\|v\|_{q_\ell}^{q_\ell}.
\]
The Hardy form is invariant under every reduced translation and dilation.
Thus it persists in every escaping profile equation, in contrast with the
compact potentials considered in Section~\ref{sec:compact-potentials}.

\begin{theorem}[Palais--Smale global compactness]
\label{thm:hardy-global}
Let \(n\geq1\), let \(\ell>1\) be an integer, and assume
\(0<\lambda<(\ell-1)^2\).  Let \((v_k)\subset H_\ell\) be real and satisfy
\[
 J^H_{\lambda,\ell}(v_k)\longrightarrow c,
 \qquad
 \|(J^H_{\lambda,\ell})'(v_k)\|_{H_\ell^*}\longrightarrow0.
\]
Then \((v_k)\) is bounded.  After passage to one subsequence, there are
finitely many nonzero real solutions
\(\phi^1,\ldots,\phi^K\in H_\ell\) of
\[
 \Gell_\ell\phi-\lambda\rho^{-1}\phi
 =|\phi|^{q_\ell-2}\phi,
\]
pairwise orthogonal reduced frames \(g_k^j\), and a remainder
\(\omega_k\to0\) strongly in \(H_\ell\) and \(L^{q_\ell}\), such that
\begin{equation}\label{eq:hardy-ps-expansion}
 v_k=\sum_{j=1}^K\mathcal T_{g_k^j}\phi^j+\omega_k.
\end{equation}
A nonzero weak limit is included as an identity-frame profile.  Moreover,
\begin{align}
 E_\ell(v_k)
 &=\sum_{j=1}^KE_\ell(\phi^j)+o(1),
 \label{eq:hardy-ps-energy}\\
 \mathsf H_\ell(v_k,v_k)
 &=\sum_{j=1}^K\mathsf H_\ell(\phi^j,\phi^j)+o(1),
 \label{eq:hardy-ps-hardy}\\
 \|v_k\|_{q_\ell}^{q_\ell}
 &=\sum_{j=1}^K\|\phi^j\|_{q_\ell}^{q_\ell}+o(1),
 \label{eq:hardy-ps-mass}\\
 c&=\sum_{j=1}^KJ^H_{\lambda,\ell}(\phi^j).
 \label{eq:hardy-ps-action}
\end{align}
The finite-stage derivative splitting holds in \(H_\ell^*\).
\end{theorem}

\begin{proof}
For every \(v\in H_\ell\),
\begin{equation}\label{eq:ps-bounded-identity}
 q_\ell J^H_{\lambda,\ell}(v)
 -\langle(J^H_{\lambda,\ell})'(v),v\rangle
 =\left(\frac{q_\ell}{2}-1\right)
 A^H_{\lambda,\ell}(v).
\end{equation}
The coercivity following from Lemma~\ref{lem:hardy-bottom-paper} and
\eqref{eq:ps-bounded-identity} show that \((v_k)\) is bounded.

Apply Theorem~\ref{thm:reduced-profile}, including a nonzero weak limit as
the identity profile.  The bounded inclusion into
\(L^2(\rho^{-1}\dd\mu_\ell)\), together with unitarity of the frame actions
for the Hardy norm, gives finite-stage Hardy splitting.  Consequently the
quadratic form \(A^H_{\lambda,\ell}\) also splits at each fixed stage.

Fix a profile frame.  Local Rellich compactness gives, after the common
subsequence,
\[
 \mathcal T_{g_k^j}^{-1}v_k\to\phi^j
 \quad\text{almost everywhere locally},
\]
while the sequence remains bounded in \(L^{q_\ell}\).  The weak convergence
of the critical Nemytskii term and invariance of the linear form imply
\((J^H_{\lambda,\ell})'(\phi^j)=0\).

The derivative splitting uses the frame-independent structure of the
Nemytskii map.  Regard it as the functional
\[
 \mathcal N(u)[\varphi]
 =\int_{\Ucal}|u|^{q_\ell-2}u\varphi\,d\mu_\ell.
\]
If \(\mathcal T_g^*\Lambda=\Lambda\circ\mathcal T_g^{-1}\), then the
critical change of variables gives
\(\mathcal N(\mathcal T_gu)=\mathcal T_g^*\mathcal N(u)\).
Since \(\mathcal T_g\) is an isometry on \(L^{q_\ell}\) and on the energy
space, the induced maps are isometries on \(L^{q_\ell'}\) and on the dual
energy space.  All constants and error norms are therefore independent of
the profile frame.  Write \(r_k^0=v_k\).  At the \(j\)-th extraction, the
common diagonal subsequence gives
\[
 \mathcal T_{g_k^j}^{-1}r_k^{j-1}
 =\phi^j+\mathcal T_{g_k^j}^{-1}r_k^j
 \longrightarrow\phi^j
 \quad\text{almost everywhere locally},
\]
and this normalized sequence is bounded in \(L^{q_\ell}\).  Applying
Lemma~\ref{lem:derivative-bl} in the normalized frame and transporting the
result back by the dual isometry yields
\begin{equation}\label{eq:moving-frame-derivative-bl}
 \left\|
 \mathcal N(r_k^{j-1})
 -\mathcal N(\mathcal T_{g_k^j}\phi^j)
 -\mathcal N(r_k^j)
 \right\|_{L^{q_\ell'}}\longrightarrow0.
\end{equation}
The critical Sobolev embedding sends this error continuously into
\(H_\ell^*\).  Summing \eqref{eq:moving-frame-derivative-bl} over
\(j=1,\ldots,J\) is legitimate for each fixed \(J\); the linear part splits
exactly because the reduced frames preserve both the energy and Hardy forms.
Consequently,
\[
 \left\|
 (J^H_{\lambda,\ell})'(v_k)
 -\sum_{j=1}^J
 (J^H_{\lambda,\ell})'(\mathcal T_{g_k^j}\phi^j)
 -(J^H_{\lambda,\ell})'(r_k^J)
 \right\|_{H_\ell^*}\longrightarrow0.
\]
The telescoping is finite: it is applied only for a fixed \(J\), not in a
double limit.  Thus every fixed-stage remainder is again asymptotically
critical.

If \(\phi\ne0\) is a critical profile, then
\[
 A^H_{\lambda,\ell}(\phi)=\|\phi\|_{q_\ell}^{q_\ell}
 \geq\bigl(S^H_{\lambda,\ell}\bigr)^{Q_\ell/2}.
\]
Since \(E_\ell(\phi)\geq A^H_{\lambda,\ell}(\phi)\), the finite-stage
Hilbert splitting gives
\[
 \#\{j:\phi^j\ne0\}
 \bigl(S^H_{\lambda,\ell}\bigr)^{Q_\ell/2}
 \leq \limsup_{k\to\infty}E_\ell(v_k).
\]
Thus the number of nonzero profiles is bounded independently of the
terminal remainder.  This proves finite termination before that remainder
is used.

Let \(K\) be the last nonzero profile and put \(\omega_k=r_k^K\).
The ordered remainder statement in
Theorem~\ref{thm:reduced-profile} now gives
\(\|\omega_k\|_{q_\ell}\to0\).  The derivative splitting and coercivity give
\[
 A^H_{\lambda,\ell}(\omega_k)
 =\langle(J^H_{\lambda,\ell})'(\omega_k),\omega_k\rangle
 +\|\omega_k\|_{q_\ell}^{q_\ell}\longrightarrow0,
\]
hence \(\omega_k\to0\) in \(H_\ell\).  Substitution in the finite-stage
identities proves \eqref{eq:hardy-ps-energy}--\eqref{eq:hardy-ps-action}.
\end{proof}

Recall
\begin{equation}\label{eq:hardy-quantum}
 d_{\lambda,\ell}
 =\frac1{Q_\ell}
 \bigl(S^H_{\lambda,\ell}\bigr)^{Q_\ell/2}.
\end{equation}

\begin{corollary}[Sharp quantum and profile count]
\label{cor:hardy-quantum}
Every nonzero real critical point \(\phi\) satisfies
\[
 J^H_{\lambda,\ell}(\phi)\geq d_{\lambda,\ell},
\]
and equality is attained by the coefficient-one rescaling of a quotient
minimizer.  Every genuinely sign-changing critical point satisfies
\[
 J^H_{\lambda,\ell}(\phi)\geq2d_{\lambda,\ell}.
\]
If \(K_{\mathrm{nod}}\) of the \(K\) profiles in
Theorem~\ref{thm:hardy-global} are nodal, then
\[
 K\leq\left\lfloor\frac c{d_{\lambda,\ell}}\right\rfloor,
 \qquad
 K+K_{\mathrm{nod}}
 \leq\left\lfloor\frac c{d_{\lambda,\ell}}\right\rfloor.
\]
There is no finite-level Palais--Smale sequence for \(c<0\) or for
\(0<c<d_{\lambda,\ell}\); at \(c=0\) every such sequence converges strongly
to zero.  If
\[
 d_{\lambda,\ell}\leq c<2d_{\lambda,\ell},
\]
there is exactly one nonnodal profile.  At \(c=d_{\lambda,\ell}\) it is a
quotient ground state and convergence is strong after applying the inverse
profile frame.
\end{corollary}

\begin{proof}
The Nehari identity and the sharp quotient give
\[
 J^H_{\lambda,\ell}(\phi)
 =\frac1{Q_\ell}A^H_{\lambda,\ell}(\phi)
 \geq d_{\lambda,\ell}.
\]
The rescaled minimizer realizes equality.  By
Lemma~\ref{lem:lattice-completion}, the positive and negative parts of a
critical point belong to \(H_\ell\) and split the energy, Hardy form, and
critical power.  Testing the equation with the two parts shows that each
nonzero part carries one quantum.  The exact action identity
\eqref{eq:hardy-ps-action} then gives the two profile counts and the stated
edge levels.
\end{proof}

The completed Cayley map transports
Theorem~\ref{thm:hardy-global} to \(H^1(\CHyp^m,g)\).  Since
\[
 J^H_{\lambda,\ell}(\mathcal C_\ell U)
 =4J^{\mathrm{CH}}_{\mu,\ell}(U),
\]
the intrinsic ground-state quantum is
\[
 d^{\mathrm{CH}}_{\mu,\ell}=\frac14d_{\lambda,\ell}.
\]
Here the quantum is a sharp positive lower bound in the finite energy
identity; it does not imply that all critical values form an integer lattice.
Together with Theorem~\ref{thm:positive-ground-state}, this proves
Theorem~\ref{thm:intro-hardy}.

\section{Compact quadratic perturbations}
\label{sec:compact-potentials}

Let \(V\geq0\) be locally integrable and suppose that
\[
 P_V(u,v)=\int_{\Ucal}Vuv\dd\mu_\ell
\]
extends continuously to
\(\dot S^1_\ell(\Ucal)\times\dot S^1_\ell(\Ucal)\).
We distinguish the multiplier
\[
 M_{\sqrt V}:\dot S^1_\ell(\Ucal)\longrightarrow L^2(d\mu_\ell),
 \qquad M_{\sqrt V}u=\sqrt V\,u,
\]
from its Riesz representative
\[
 K_V=M_{\sqrt V}^*M_{\sqrt V},
 \qquad E_\ell(K_Vu,v)=P_V(u,v).
\]
Throughout this section \(M_{\sqrt V}\) is compact.

\begin{lemma}[A concrete compact-multiplier class]
\label{lem:lp-compact-multiplier}
If \(V\in L^{Q_\ell/2}(\Ucal,d\mu_\ell)\), then
\(M_{\sqrt V}\) is compact.
\end{lemma}

\begin{proof}
Hölder's inequality and Theorem~\ref{thm:sharp-sobolev} give boundedness.
Approximate \(V\) in \(L^{Q_\ell/2}\) by bounded compactly supported
potentials.  For each approximant, local Rellich compactness of the reduced
energy space gives compactness of the multiplier.  The Hölder estimate shows
convergence in operator norm.
\end{proof}

The first consequence records the distinction between attractive and
repulsive perturbations.  Set
\[
 \lambda_1(V)
 =\inf_{P_V(u,u)>0}\frac{E_\ell(u)}{P_V(u,u)}.
\]

\begin{proposition}[Bottom critical quotients]\label{prop:compact-quotients}
Let \(n\geq1\), let \(\ell\geq1\) be an integer, and let
\(V\geq0\) be locally integrable.  Assume that \(P_V\) extends continuously
to \(\dot S^1_\ell(\Ucal)\times\dot S^1_\ell(\Ucal)\) and that
\(M_{\sqrt V}\) is compact.  Then the following statements hold.
\begin{enumerate}[label=\textup{(\roman*)}]
\item Suppose
\[
 0<\lambda<\lambda_1(V)
\]
and
\[
 S^-_{\lambda,V}
 =\inf_{u\ne0}
 \frac{E_\ell(u)-\lambda P_V(u,u)}{\|u\|_{q_\ell}^2}
 <S_\ell.
\]
Then \(S^-_{\lambda,V}>0\), the infimum is attained by a nonnegative
normalized weak solution, and every normalized minimizing sequence is
sequentially precompact in the energy and critical norms.

\item For the repulsive quotient,
\[
 S^+_V
 =\inf_{u\ne0}
 \frac{E_\ell(u)+P_V(u,u)}{\|u\|_{q_\ell}^2}
 =S_\ell.
\]
It is attained if and only if some pure Sobolev extremal \(B\) satisfies
\(P_V(B,B)=0\).  In particular, it is not attained when \(V>0\) almost
everywhere.
\end{enumerate}
\end{proposition}

\begin{proof}
Compactness of the multiplier makes \(P_V\) continuous along weakly
convergent energy sequences.  In the attractive case, coercivity,
Brezis--Lieb splitting, and the strict inequality below the pure threshold
exclude vanishing and dichotomy.  A minimizing sequence therefore converges
strongly, and taking the absolute value gives a nonnegative minimizer.

For the repulsive quotient, the sharp inequality gives
\(S^+_V\geq S_\ell\).  Translating or dilating a pure extremal
along an escaping frame makes it weakly null; compactness of
\(M_{\sqrt V}\) then makes its potential mass tend to zero.  Hence the
reverse inequality holds.  Equality for an actual minimizer forces equality
in the pure sharp inequality and zero potential mass, which proves the
attainment criterion.
\end{proof}

Thus a solution for a nontrivial repulsive potential cannot in general be
obtained by minimizing the bottom quotient.  We turn to higher critical
levels.  Define
\[
 I_V(u)
 =\frac12\bigl(E_\ell(u)+P_V(u,u)\bigr)
 -\frac1{q_\ell}\|u\|_{q_\ell}^{q_\ell},
\]
\[
 I_0(u)=\frac12E_\ell(u)
 -\frac1{q_\ell}\|u\|_{q_\ell}^{q_\ell},
 \qquad
 d_{0,\ell}
 =\frac1{Q_\ell}
 \bigl(S_\ell\bigr)^{Q_\ell/2}.
\]

\begin{theorem}[Compact-potential global compactness]
\label{thm:compact-potential-global}
Let \(n\geq1\), let \(\ell\geq1\) be an integer, and let
\(V\geq0\) be locally integrable.  Assume that \(P_V\) extends continuously
to \(\dot S^1_\ell(\Ucal)\times\dot S^1_\ell(\Ucal)\) and that
\(M_{\sqrt V}\) is compact.  Let
\((u_k)\subset\dot S^1_\ell(\Ucal)\) be real and satisfy
\[
 I_V(u_k)\longrightarrow c,
 \qquad
 \|I_V'(u_k)\|_{(\dot S^1_\ell)^*}\longrightarrow0.
\]
After passage to a subsequence,
\begin{equation}\label{eq:compact-potential-expansion}
 u_k=u^0+\sum_{j=1}^K\mathcal T_{g_k^j}\phi^j+\omega_k,
\end{equation}
where
\[
 \omega_k\longrightarrow0
 \quad\text{in }\dot S^1_\ell(\Ucal)\text{ and }L^{q_\ell},
\]
\(u^0\) may vanish and solves
\[
 \Gell_\ell u^0+Vu^0=|u^0|^{q_\ell-2}u^0,
\]
and every \(\phi^j\ne0\) solves the pure equation
\[
 \Gell_\ell\phi^j=|\phi^j|^{q_\ell-2}\phi^j.
\]
The frames are pairwise orthogonal and escape the identity.  Moreover,
\begin{align}
 E_\ell(u_k)
 &=E_\ell(u^0)+\sum_jE_\ell(\phi^j)+o(1),
 \label{eq:compact-potential-energy}\\
 P_V(u_k,u_k)&=P_V(u^0,u^0)+o(1),
 \label{eq:compact-potential-mass}\\
 \|u_k\|_{q_\ell}^{q_\ell}
 &=\|u^0\|_{q_\ell}^{q_\ell}
 +\sum_j\|\phi^j\|_{q_\ell}^{q_\ell}+o(1),
 \label{eq:compact-potential-critical}\\
 c&=I_V(u^0)+\sum_jI_0(\phi^j).
 \label{eq:compact-potential-action}
\end{align}
Every nonzero full or pure term has action at least \(d_{0,\ell}\).
A nodal pure profile has action at least \(2d_{0,\ell}\), while every
one-sign pure profile has action exactly \(d_{0,\ell}\).  Consequently,
\(I_V\) satisfies \((PS)_c\) whenever
\begin{equation}\label{eq:compact-ps-window}
 d_{0,\ell}<c<2d_{0,\ell}.
\end{equation}
There is no finite-level Palais--Smale sequence for \(c<0\) or
\(0<c<d_{0,\ell}\), and level zero converges strongly to zero.  At
\(c=d_{0,\ell}\) one pure bubble may escape; no compactness assertion is
made at \(c=2d_{0,\ell}\).
\end{theorem}

\begin{proof}
The identity
\[
 q_\ell I_V(u)-\langle I_V'(u),u\rangle
 =\left(\frac{q_\ell}{2}-1\right)
 \bigl(E_\ell(u)+P_V(u,u)\bigr)
\]
gives boundedness.  Pass to a subsequence with
\(u_k\rightharpoonup u^0\) in the energy space and almost everywhere
locally.  Weak convergence of the critical Nemytskii term shows that
\(u^0\) solves the full equation.

Put \(w_k=u_k-u^0\).  Compactness of \(M_{\sqrt V}\) gives
\[
 P_V(w_k,w_k)\longrightarrow0,
 \qquad
 \|K_Vw_k\|_{\dot S^1_\ell}\longrightarrow0.
\]
Hilbert splitting and the scalar and derivative Brezis--Lieb lemmas then
give
\[
 I_0(w_k)=c-I_V(u^0)+o(1),
 \qquad
 I_0'(w_k)\longrightarrow0.
\]
Thus the entire weakly null residual is a Palais--Smale sequence for the
pure functional.

Apply Theorem~\ref{thm:reduced-profile} to \(w_k\).  In a profile frame,
local almost-everywhere convergence and weak convergence in the energy space
give
\[
 \mathcal T_{g_k^j}^{-1}w_k\rightharpoonup\phi^j.
\]
The frame actions preserve \(I_0\).  Passing to the limit in the derivative,
using Lemma~\ref{lem:derivative-bl}, shows that
\(I_0'(\phi^j)=0\).  The derivative splitting is frame-independent: the
Nemytskii map satisfies
\(\mathcal N(\mathcal T_gu)=\mathcal T_g^*\mathcal N(u)\), and since the
frame actions are isometric on both \(L^{q_\ell}\) and the energy space, the
induced dual-space norms are the same in every frame.  More precisely, with
\(r_k^0=w_k\), the common diagonal subsequence gives
\[
 \mathcal T_{g_k^j}^{-1}r_k^{j-1}
 =\phi^j+\mathcal T_{g_k^j}^{-1}r_k^j
 \longrightarrow\phi^j
 \quad\text{almost everywhere locally}.
\]
The normalized sequences are bounded in \(L^{q_\ell}\), so the transported
estimate \eqref{eq:moving-frame-derivative-bl} applies verbatim.  Summing only
over a fixed number of profiles gives
\[
 I_0(w_k)
 =\sum_{j=1}^JI_0(\phi^j)+I_0(r_k^J)+o(1),
 \qquad I_0'(r_k^J)\longrightarrow0.
\]
The potential term does not appear in the pure residual functional: each
escaping profile frame carries the compact multiplier \(M_{\sqrt V}\) to
zero, so the pure profiles solve the pure Geller equation.
Every nonzero profile satisfies the pure Nehari identity and hence has
energy at least
\(\bigl(S_\ell\bigr)^{Q_\ell/2}\).
The finite-stage Hilbert splitting therefore permits only finitely many
nonzero profiles.  If \(K\) is their number, the ordered remainder estimate
gives \(\|r_k^K\|_{q_\ell}\to0\), and
\[
 E_\ell(r_k^K)
 =\langle I_0'(r_k^K),r_k^K\rangle
  +\|r_k^K\|_{q_\ell}^{q_\ell}\longrightarrow0.
\]
Thus the terminal remainder converges strongly.  Combining this pure
decomposition with the compact-potential splitting at the identity frame
proves
\eqref{eq:compact-potential-expansion}--\eqref{eq:compact-potential-action}.

The Nehari identity, nonnegativity of \(V\), and
Theorem~\ref{thm:sharp-sobolev} show that every nonzero full or pure critical
point has action at least \(d_{0,\ell}\).  Lemma~\ref{lem:lattice-completion}
gives the two-quantum bound for a nodal pure profile.
Theorem~\ref{thm:weak} classifies a one-sign pure profile, and its
coefficient-one normalization gives action exactly \(d_{0,\ell}\).

Suppose \eqref{eq:compact-ps-window} holds.  If \(u^0\ne0\), the action
identity and the upper bound \(c<2d_{0,\ell}\) leave no room for an escaping
profile.  If \(u^0=0\), the identity and \(c>0\) force at least one pure
profile, while \(c<2d_{0,\ell}\) permits only one.  It cannot be nodal and
therefore has action exactly \(d_{0,\ell}\), contradicting
\(c>d_{0,\ell}\).  Hence \(u^0\ne0\) and convergence is strong.

All terms in \eqref{eq:compact-potential-action} have nonnegative action and
every nonzero term contributes at least \(d_{0,\ell}\).  This excludes
finite-level Palais--Smale sequences for \(c<0\) and for
\(0<c<d_{0,\ell}\); at \(c=0\) all terms vanish and the sequence converges
strongly to zero.  At \(c=d_{0,\ell}\) a single pure bubble may escape, and
the two-quantum endpoint is not covered by the preceding strict argument.
\end{proof}

\section{Busemann center--scale linking and a repulsive-potential solution}
\label{sec:minmax}

Let
\[
 \mathcal M_\ell
 =\{u\in\dot S^1_\ell(\Ucal):\|u\|_{q_\ell}=1\},
 \qquad
 F_V(u)=E_\ell(u)+P_V(u,u).
\]
Fix the positive centered extremal \(B_0\in\mathcal M_\ell\), so that
\[
 E_\ell(B_0)=S_\ell.
\]

\subsection{Global center and scale coordinates}

For \(u\in\mathcal M_\ell\), the polar lift defines a probability measure on
\(\Heis^{n+\ell}\),
\[
 \dd\nu_u^{\mathrm H}
 =c_\ell^{-1}|u^\uparrow|^{q_\ell}\dd Z\dd t.
\]
Push this measure to the ideal boundary of \(\CHyp^{n+\ell+1}\) by the
boundary Cayley map.  Since \(\nu_u^{\mathrm H}\) has an \(L^1\) density with
respect to Heisenberg Lebesgue measure, it has no atoms.  The boundary Cayley
map is a smooth identification away from its single point at infinity, which
also has zero mass; hence the resulting boundary probability \(\nu_u\) is
atomless.

We use the complex-hyperbolic metric of holomorphic sectional curvature
\(-4\).  Its sectional curvatures lie in \([-4,-1]\), so its maximum
sectional curvature is \(-1\), exactly the normalization used in the
rank-one barycenter construction.  This construction goes back to
Besson--Courtois--Gallot \cite{BessonCourtoisGallot1995Entropies}; we use
the form stated for \(K\)-hyperbolic spaces in
\cite{Savini2023NaturalMaps}.  The ambient
complex dimension is \(r=n+\ell+1\).  In Savini's notation the dimensional
hypothesis is \(d(r-1)\geq2\); here \(d=2\) and
\(2(n+\ell)\geq4\).  Thus the Busemann functional of \(\nu_u\) has a unique
minimizer, denoted \(b(u)\), and the barycenter is weak-* continuous on
atomless probabilities and equivariant under holomorphic isometries.

Write \(\varrho\) for the Siegel height on the fixed copy of
\(\CHyp^{n+1}\), and let \(o=(0,0,1)\) be its basepoint in
\((z,t,\varrho)\)-coordinates.  For
\(\eta=(z_*,t_*)\in\Heis^n\) and \(h>0\), let \(g_{h,\eta}\) be the
rotation-free Heisenberg translation--dilation isometry normalized by
\[
 g_{h,\eta}o=(z_*,t_*,h^2).
\]
We write
\[
 P_{\mathrm{link}}
 =\{g_{h,\eta}:h>0,\ \eta\in\Heis^n\}
 \cong\Heis^n\rtimes\R_+
\]
for this rotation-free subgroup.

\begin{lemma}[Busemann center--scale map]\label{lem:busemann-coordinates}
The barycenter \(b(u)\) belongs to the \(U(\ell)\)-fixed totally geodesic
copy of \(\CHyp^{n+1}\).  Write
\[
 b(u)=(z_b(u),t_b(u),\varrho_b(u))
\]
in this copy and set
\[
 \beta(u)=(z_b(u),t_b(u))\in\Heis^n,
 \qquad
 \gamma(u)=\frac12\log\varrho_b(u)\in\R.
\]
Then \(\beta\) and \(\gamma\) are continuous,
\[
 b(u)=g_{e^{\gamma(u)},\beta(u)}o,
\]
and, on the normalized bubble orbit,
\begin{equation}\label{eq:barycenter-calibration}
 (\beta,\gamma)(\pm\mathcal T_{h,\eta}B_0)
 =(\eta,\log h).
\end{equation}
\end{lemma}

\begin{proof}
The lifted mass is invariant under \(U(\ell)\).  Uniqueness and equivariance
of the barycenter therefore make \(b(u)\) fixed by \(U(\ell)\).  The fixed
locus is the copy of \(\CHyp^{n+1}\) obtained by setting the auxiliary complex
coordinates equal to zero.  The Heisenberg translation--dilation subgroup
acts simply transitively on this space, which defines \(\beta\) and
\(\gamma\).

Strong convergence in \(\dot S^1_\ell\) implies strong convergence in
\(L^{q_\ell}\), hence total-variation convergence of the lifted critical
mass.  Since total variation dominates weak-*, and the limit probability is
atomless, the weak-* continuous barycenter of
\cite{Savini2023NaturalMaps} gives continuity of
\((\beta,\gamma)\).

By the boundary Cayley Jacobian formula of Frank--Lieb
\cite[Appendix~A]{FrankLieb2012SharpHeisenberg}, the boundary Cayley
push-forward of the normalized lifted mass of \(B_0\) is normalized
spherical measure.  Here normalization follows from
\(\|B_0\|_{q_\ell}=1\) and the polar identity.  With the interior Cayley
normalization fixed in Section~\ref{sec:cayley}, this spherical measure is
the visual probability at \(o\), whose barycenter is \(o\).  Equivariance
under reduced translations and dilations proves
\eqref{eq:barycenter-calibration}.  The sign has no effect on the critical
mass.
\end{proof}

\begin{lemma}[The centered gap]\label{lem:centered-gap}
Assume \(V\geq0\), \(V\not\equiv0\), and \(M_{\sqrt V}\) is compact.  Then
\[
 c_V
 :=\inf\{F_V(u):u\in\mathcal M_\ell,\ \beta(u)=0,\
 \gamma(u)=0\}
 >S_\ell.
\]
\end{lemma}

\begin{proof}
Suppose that \(u_k\) satisfies the two constraints and
\(F_V(u_k)\to S_\ell\).  The sharp inequality and
nonnegativity of \(P_V\) imply
\[
 E_\ell(u_k)\to S_\ell,
 \qquad P_V(u_k,u_k)\to0.
\]
By Corollary~\ref{cor:sharp-compactness}, there are signs \(\sigma_k\) and
frames \(g_k\) such that
\[
 u_k-\sigma_k\mathcal T_{g_k}B_0\longrightarrow0
 \quad\text{in }\dot S^1_\ell.
\]
Pulling back by \(g_k\), continuity and equivariance of \(b\) give
\[
 g_k^{-1}o
 =b(\mathcal T_{g_k}^{-1}u_k)\longrightarrow b(B_0)=o.
\]
Properness of the Heisenberg--Iwasawa coordinates implies \(g_k\to e\):
explicitly, the orbit map \(g_{h,\eta}\mapsto g_{h,\eta}o=(z_*,t_*,h^2)\)
is a homeomorphism from the rotation-free subgroup
\(P_{\mathrm{link}}\cong\Heis^n\rtimes\R_+\) to the Siegel domain
\(\{(z,t,\varrho):\varrho>0\}\), with inverse
\((z_*,t_*,\varrho)\mapsto g_{\varrho^{1/2},(z_*,t_*)}\).  Convergence
\(g_k^{-1}o\to o\) therefore gives \(h_k\to1\) and \(\eta_k\to0\) in the
Lie-group topology.
Thus \(u_k\to\pm B_0\) strongly.  Continuity of the multiplier yields
\[
 P_V(u_k,u_k)\longrightarrow P_V(B_0,B_0)>0,
\]
because \(B_0>0\) almost everywhere and \(V\) is nonnegative and nonzero.
This contradiction proves the claim.
\end{proof}

\subsection{The linking level}

For \(R>0\), set
\[
 K_R=\overline{B_R^{\,2n+1}}\times[-R,R]
 \subset\Heis^n\times\R
\]
and
\[
 \Phi(\eta,s)=\mathcal T_{e^s,\eta}B_0.
\]
By \eqref{eq:barycenter-calibration},
\[
 (\beta,\gamma)\circ\Phi(\eta,s)=(\eta,s).
\]
Every sequence on \(\partial K_R\), with \(R\to\infty\), is an escaping
reduced frame.  Since \(M_{\sqrt V}\) is compact,
\[
 \max_{y\in\partial K_R}P_V(\Phi(y),\Phi(y))
 \longrightarrow0.
\]
Indeed, otherwise there would be \(R_j\to\infty\) and
\(y_j=(\eta_j,s_j)\in\partial K_{R_j}\) for which the potential mass is
bounded away from zero.  Since either \(|\eta_j|_{\mathbb R^{2n+1}}=R_j\)
or \(|s_j|=R_j\), the corresponding reduced frames escape, and
\(\Phi(y_j)\rightharpoonup0\) in \(\dot S^1_\ell(\Ucal)\).
Compactness of \(M_{\sqrt V}\) then gives
\(P_V(\Phi(y_j),\Phi(y_j))\to0\), a contradiction.
For \(R\) sufficiently large, the boundary energy is therefore below
\(c_V\).  Define
\[
 \Gamma_R
 =\{\zeta\in C(K_R,\mathcal M_\ell):
 \zeta=\Phi\text{ on }\partial K_R\}
\]
and
\begin{equation}\label{eq:minmax-quotient}
 a_{\mathrm{mm}}
 =\inf_{\zeta\in\Gamma_R}\max_{y\in K_R}F_V(\zeta(y)).
\end{equation}

\begin{proposition}[The min--max window]\label{prop:minmax-window}
Let \(n\geq1\), let \(\ell\geq1\) be an integer, and let
\(V\geq0\), \(V\not\equiv0\), be locally integrable.  Assume that
\(P_V\) extends continuously to
\(\dot S^1_\ell(\Ucal)\times\dot S^1_\ell(\Ucal)\) and that
\(M_{\sqrt V}\) is compact.  If
\begin{equation}\label{eq:bubble-orbit-smallness}
 \sup_{h>0,\,\eta\in\Heis^n}
 P_V(\mathcal T_{h,\eta}B_0,\mathcal T_{h,\eta}B_0)
 <\bigl(2^{2/Q_\ell}-1\bigr)S_\ell,
\end{equation}
then
\[
 S_\ell<c_V\leq a_{\mathrm{mm}}
 <2^{2/Q_\ell}S_\ell.
\]
Equivalently,
\[
 d_{0,\ell}<c_{\mathrm{mm}}<2d_{0,\ell},
 \qquad
 c_{\mathrm{mm}}
 =\frac1{Q_\ell}a_{\mathrm{mm}}^{Q_\ell/2}.
\]
\end{proposition}

\begin{proof}
For every \(\zeta\in\Gamma_R\), the boundary map
\((\beta,\gamma)\circ\zeta\) agrees with the identity on
\(\partial K_R\).  Identify
\(\Heis^n\times\mathbb R\) with \(\mathbb R^{2n+2}\).  Then \(K_R\) is a
compact topological ball, \((0,0)\in\operatorname{int}K_R\), and the
boundary map avoids \((0,0)\).  Since
\((\beta,\gamma)\circ\zeta|_{\partial K_R}=\mathrm{id}\), the Brouwer
degree of \((\beta,\gamma)\circ\zeta\) at \((0,0)\) equals one.  In
particular, every \(\zeta\in\Gamma_R\) meets the centered constraint
\(\{(\beta,\gamma)=(0,0)\}\).  Lemma~\ref{lem:centered-gap}
gives
\[
 a_{\mathrm{mm}}\geq c_V>S_\ell.
\]
The canonical filling \(\Phi\) is admissible, and
\[
 F_V(\Phi(\eta,s))
 =S_\ell
 +P_V(\mathcal T_{e^s,\eta}B_0,
      \mathcal T_{e^s,\eta}B_0).
\]
Condition \eqref{eq:bubble-orbit-smallness} gives the strict upper bound.
The action conversion follows from
\(q_\ell=2Q_\ell/(Q_\ell-2)\).
\end{proof}

\begin{theorem}[A higher critical point]\label{thm:repulsive-existence}
Let \(n\geq1\), let \(\ell\geq1\) be an integer, and let
\(V\geq0\), \(V\not\equiv0\), be locally integrable.  Assume that
\(P_V\) extends continuously to
\(\dot S^1_\ell(\Ucal)\times\dot S^1_\ell(\Ucal)\), that
\(M_{\sqrt V}\) is compact, and that
\eqref{eq:bubble-orbit-smallness} holds.  Then the equation
\[
 \Gell_\ell v+Vv=|v|^{q_\ell-2}v
\]
has a nonzero one-sign weak solution \(v\in\dot S^1_\ell(\Ucal)\) such that
\[
 d_{0,\ell}<I_V(v)<2d_{0,\ell}.
\]
The concrete assumptions
\[
 V\geq0,\qquad
 0<\|V\|_{L^{Q_\ell/2}(d\mu_\ell)}
 <\bigl(2^{2/Q_\ell}-1\bigr)S_\ell
\]
are sufficient.
\end{theorem}

\begin{proof}
The constraint \(\mathcal M_\ell=\{\|u\|_{q_\ell}^{q_\ell}=1\}\) is the
regular level set of \(G(u)=\|u\|_{q_\ell}^{q_\ell}\), whose derivative
satisfies \(G'(u)[u]=q_\ell\) on \(\mathcal M_\ell\).  Since
\(\dot S^1_\ell(\Ucal)\) is a Hilbert space and the critical-norm
functional is \(C^1\) with nonvanishing derivative on \(\mathcal M_\ell\),
the constraint is a complete \(C^1\) codimension-one submanifold.

The boundary value \(\max_{\partial K_R}F_V(\Phi)\) is strictly below
\(a_{\mathrm{mm}}\) for \(R\) large (the escaping potential decay proved
above).  The constrained minimax principle in the form of
\cite[Theorem~3.2 and Corollary~3.3]{Willem1996Minimax} applies: if no
constrained Palais--Smale sequence existed at level \(a_{\mathrm{mm}}\),
a tangential
pseudo-gradient field on \(\mathcal M_\ell\), supported in a thin energy
strip around \(a_{\mathrm{mm}}\) and vanishing below the boundary level,
would generate a deformation flow on \(\mathcal M_\ell\).  Integrating this
flow deforms any near-optimal filling relative to \(\partial K_R\) (where
the flow vanishes) below the min--max level, contradicting
\eqref{eq:minmax-quotient}.  The completeness of the constraint and the
strict boundary gap \(\max_{\partial K_R}F_V(\Phi)<a_{\mathrm{mm}}\) are
essential for the deformation to be well-defined and relative.

Thus there are \(u_k\in\mathcal M_\ell\) such that
\[
 F_V(u_k)\longrightarrow a_{\mathrm{mm}},
 \qquad
 d(F_V|_{\mathcal M_\ell})(u_k)\longrightarrow0.
\]

Consequently there are multipliers \(\theta_k\) for which
\[
 F_V'(u_k)-\theta_kG'(u_k)\longrightarrow0
 \quad\text{in }(\dot S^1_\ell)^*.
\]
Pairing with \(u_k\) and using
\(F_V'(u_k)[u_k]=2F_V(u_k)\) gives
\[
 \theta_k=\frac{2a_{\mathrm{mm}}}{q_\ell}+o(1).
\]
After division by two, the multiplier identity becomes
\[
 E_\ell(u_k,\varphi)+P_V(u_k,\varphi)
 -a_{\mathrm{mm}}
 \int_{\Ucal}|u_k|^{q_\ell-2}u_k\varphi\dd\mu_\ell
 =o(1)\|\varphi\|_{\dot S^1_\ell}.
\]
Hence
\[
 v_k=a_{\mathrm{mm}}^{1/(q_\ell-2)}u_k
\]
is an \(I_V\)-Palais--Smale sequence.  Since
\[
 \frac12-\frac1{q_\ell}=\frac1{Q_\ell},
 \qquad
 \frac{q_\ell}{q_\ell-2}=\frac{Q_\ell}{2},
\]
its action converges to
\[
 c_{\mathrm{mm}}
 =\frac1{Q_\ell}a_{\mathrm{mm}}^{Q_\ell/2}.
\]
Proposition~\ref{prop:minmax-window} places this level in the compactness
window \eqref{eq:compact-ps-window}.
Theorem~\ref{thm:compact-potential-global} therefore gives strong convergence to a
nonzero critical point \(v\).

If \(v\) changed sign, Lemma~\ref{lem:lattice-completion} would permit
testing with \(v^+\) and \(-v^-\).  The two parts satisfy separate Nehari
identities, and the pure sharp inequality gives
\[
 I_V(v)\geq2d_{0,\ell},
\]
contrary to the strict upper bound.  Thus \(v\) has one sign.

Finally, if \(V\in L^{Q_\ell/2}\), Hölder's inequality and
\(\|B_0\|_{q_\ell}=1\) give
\[
 P_V(\mathcal T_{h,\eta}B_0,\mathcal T_{h,\eta}B_0)
 \leq\|V\|_{L^{Q_\ell/2}},
\]
uniformly in the bubble frame.  Lemma~\ref{lem:lp-compact-multiplier}
and the displayed norm bound imply the compact-multiplier and
orbit-smallness hypotheses.
\end{proof}

The degree construction is analogous in purpose to barycenter methods for
critical Euclidean and real-hyperbolic potential equations, but the present
coordinates and mechanisms of loss of compactness are those of the reduced
Geller geometry; compare
\cite{Passaseo1996PositiveSolutions,LiWang2026Global}.

\begin{corollary}[Conditional strict positivity]\label{cor:potential-positive}
Assume the \(L^{Q_\ell/2}\) hypotheses in the final display of
Theorem~\ref{thm:repulsive-existence}, and let \(v\) be the solution supplied
by that theorem.  Choose its sign so that \(v\geq0\), and define the lifted
potential by
\[
 V^\uparrow(z,w,t)=V(z,t,|w|^2),
 \qquad w\in\C^\ell.
\]
Suppose
\[
 V^\uparrow\in L^p_{\mathrm{loc}}(\Heis^{n+\ell})
 \quad\text{for some }p>Q_\ell/2.
\]
Then \(v\) has a locally H\"older representative which is everywhere
strictly positive.  If \(V^\uparrow\) is smooth, the lifted and reduced
representatives are smooth.
\end{corollary}

\begin{proof}
The polar lift and the axis cutoffs from
Lemmas~\ref{lem:polar-lift} and \ref{lem:axis} extend the reduced weak
identity to full-space tests, provided the potential term is controlled on
strips about the auxiliary axis.  For the potential strip estimate, write
\(W=V^\uparrow\), \(U=v^\uparrow\), \(N=n+\ell\), \(Q=Q_\ell\), and
\(q=q_\ell\).  Polar integration gives the three identities
\[
 \int_{\Heis^N}|W|^{Q/2}\,dZ\,dt
 =c_\ell\int_{\Ucal}|V|^{Q/2}\,d\mu_\ell,
 \qquad
 \int_{\Heis^N}|U|^q\,dZ\,dt
 =c_\ell\int_{\Ucal}|v|^q\,d\mu_\ell,
\]
\[
 \int_{\Heis^N}W|U|^2\,dZ\,dt
 =c_\ell\int_{\Ucal}Vv^2\,d\mu_\ell.
\]
In particular, \(W\in L^{Q/2}(\Heis^N)\) and
\(U\in L^q(\Heis^N)\).

Let \(\Phi\) be a compactly supported smooth \(U(\ell)\)-invariant test,
and let \(\chi_\varepsilon(w)=\eta_\varepsilon(|w|^2)\), where
\(\eta_\varepsilon\) is the power cutoff from Lemma~\ref{lem:axis} when
\(\ell>1\) and its logarithmic cutoff when \(\ell=1\).  Thus
\(\chi_\varepsilon\) vanishes near the auxiliary axis and equals one away
from a shrinking tube about that axis.  If \(\mathcal K\) contains the
support of \(\Phi\), H\"older's inequality and
\(2/Q+1/q+1/q=1\) give
\[
 \left|\int_{\mathcal K}WU(1-\chi_\varepsilon)\Phi\right|
 \leq
 \|W\|_{L^{Q/2}(\mathcal K)}
 \|U\|_{L^q(\mathcal K)}
 \|(1-\chi_\varepsilon)\Phi\|_{L^q(\mathcal K)}
 \longrightarrow0.
\]
The last factor tends to zero by dominated convergence, since the auxiliary
axis has Lebesgue measure zero.  The nonlinear strip term also tends to zero:
\[
 \left|\int_{\mathcal K}|U|^{q-2}U
       (1-\chi_\varepsilon)\Phi\right|
 \leq
 \|U\|_{L^q(\operatorname{supp}(1-\chi_\varepsilon)\cap\mathcal K)}^{q-1}
 \|\Phi\|_{L^q(\mathcal K)}
 \longrightarrow0.
\]
Finally, the gradient cutoff error tends to zero by the vanishing-capacity
estimate in Lemma~\ref{lem:axis}.  We may therefore test the reduced equation
with \(\chi_\varepsilon\Phi\), pass to the limit, and cancel the common polar
factor \(c_\ell\).  This proves the full weak identity
\[
 -\Delta_{\Heis^{n+\ell}}v^\uparrow+V^\uparrow v^\uparrow
 =(v^\uparrow)^{q_\ell-1}
\]
first for invariant tests.  Averaging an arbitrary compactly supported smooth
test over \(U(\ell)\) leaves all three pairings unchanged, so the identity
holds for every full-space test.
Lemma~\ref{lem:lifted-regularity} gives local boundedness, H\"older
regularity, Harnack's inequality, and strict positivity.  Smooth
subelliptic bootstrapping and invariant factorization give the last
assertion.
\end{proof}

Theorem~\ref{thm:repulsive-existence} and
Corollary~\ref{cor:potential-positive} prove
Theorem~\ref{thm:intro-potential}; the argument uses the open action window
\((d_{0,\ell},2d_{0,\ell})\), and neither endpoint is included.  Potentials
violating \eqref{eq:bubble-orbit-smallness} are not covered by the present
argument.

\appendix

\section{Splitting and regularity lemmas}
\label{app:technical}

\begin{lemma}[Derivative Brezis--Lieb lemma]
\label{lem:derivative-bl}
Let \(2<q<\infty\), let \(f_k\to f\) almost everywhere, and suppose that
\((f_k)\) is bounded in \(L^q\).  Then, with \(q'=q/(q-1)\),
\[
 \bigl\|
 |f_k|^{q-2}f_k
 -|f_k-f|^{q-2}(f_k-f)
 -|f|^{q-2}f
 \bigr\|_{q'}\longrightarrow0.
\]
\end{lemma}

\begin{proof}
Put \(g_k=f_k-f\) and \(N(s)=|s|^{q-2}s\).  Pointwise continuity gives
\[
 R_k:=N(g_k+f)-N(g_k)-N(f)\longrightarrow0
\quad\text{almost everywhere}.
\]
The scalar inequality
\[
 |N(a+b)-N(a)|
 \leq C_q\bigl(|a|^{q-2}|b|+|b|^{q-1}\bigr)
\]
implies
\[
 |R_k|^{q'}
 \leq C_q\bigl(
 |g_k|^{(q-2)q'}|f|^{q'}+|f|^q
 \bigr).
\]
For every measurable set \(A\), Hölder's inequality gives
\[
 \int_A|g_k|^{(q-2)q'}|f|^{q'}
 \leq
 \left(\int_A|g_k|^q\right)^{(q-2)/(q-1)}
 \left(\int_A|f|^q\right)^{1/(q-1)}.
\]
The first factor is uniformly bounded, while the second is uniformly small
when the \(L^q\)-mass of \(f\) on \(A\) is small.  Thus
\((|R_k|^{q'})\) is uniformly integrable, and Vitali's theorem completes the
proof.
\end{proof}

\begin{lemma}[Lattice operations in the energy completion]
\label{lem:lattice-completion}
For every real \(u\in\dot S^1_\ell(\Ucal)\), the functions
\[
 u^+=\max\{u,0\},\qquad u^-=\max\{-u,0\}
\]
belong to \(\dot S^1_\ell(\Ucal)\), and
\[
 E_\ell(u)=E_\ell(u^+)+E_\ell(u^-),
 \qquad E_\ell(u^+,u^-)=0.
\]
The critical power splits in the same way.  When the Hardy form is
continuous, it splits exactly; the same is true for every nonnegative
potential form \(P_V\) considered in Section~\ref{sec:compact-potentials}.
\end{lemma}

\begin{proof}
Choose \(u_j\in C_c^\infty(\Ucal)\) converging to \(u\) in the energy norm.
Approximate \(s^\pm\) by smooth normal contractions
\(\theta_r^\pm\) satisfying
\[
 \theta_r^\pm(0)=0,\qquad |(\theta_r^\pm)'|\leq1.
\]
The chain rule for each of the first-order fields in
\[
 \Gamma_\ell(f,f)
 =|\nabla_Hf|^2+4\rho(|f_\rho|^2+|f_t|^2)
\]
shows that \(\theta_r^\pm(u_j)\) are uniformly controlled in the energy
space.  Passing first in \(r\) and then in \(j\), using the local Sobolev
chain rule and the fact that the weak derivatives of \(u\) vanish almost
everywhere on \(\{u=0\}\), gives
\[
 D u^+=\mathbf1_{\{u>0\}}Du,\qquad
 D u^-=-\mathbf1_{\{u<0\}}Du
\]
for every field \(D\) occurring in \(\Gamma_\ell\).
The energy identities follow because the supports of these derivatives are
disjoint.  The critical-power and potential identities are pointwise.
Continuity of the Hardy form permits passage from the contraction
approximants and yields its splitting.
\end{proof}

\begin{lemma}[Local regularity for the lifted equation]
\label{lem:lifted-regularity}
Let \(Q=Q_\ell\), let \(U\geq0\) be a nonzero finite-energy weak solution on
\(\Heis^{n+\ell}\) of
\[
 -\Delta_{\Heis^{n+\ell}}U+WU=U^{2Q/(Q-2)-1},
\qquad W\geq0.
\]
Assume \(W\in L^p_{\mathrm{loc}}\) for some \(p>Q/2\).
Then \(U\) is locally bounded, locally H\"older continuous, and strictly
positive.  If \(W\) is smooth, then \(U\) is smooth.
\end{lemma}

\begin{proof}
Write the equation as
\[
 -\Delta_{\Heis^{n+\ell}}U
 =aU,\qquad a=U^{4/(Q-2)}-W.
\]
Initially \(U^{4/(Q-2)}\in L^{Q/2}_{\mathrm{loc}}\).
Since \(p>Q/2\), the same is true locally of \(W\), and hence of \(a\).
Fix \(B_r\Subset B_R\), choose a horizontal Lipschitz cutoff \(\eta\), and
put
\[
 U_L=\min\{U,L\},\qquad
 Y=U\,U_L^{(\beta-2)/2},\qquad
 \chi=\frac{Q}{Q-2},
\]
where \(\beta\geq2\).  Testing with
\(\eta^2U\,U_L^{\beta-2}\), and then using the local Folland--Stein
inequality, gives
\begin{equation}\label{eq:local-caccioppoli-iteration}
 \|\eta Y\|_{2\chi}^2
 \leq C\beta^2\left[
 \int (|\nabla_H\eta|^2+\eta^2)Y^2
 +\int |a|(\eta Y)^2
 \right].
\end{equation}
The truncation derivative has the favorable sign, and the constant is
independent of \(L\).  For a fixed \(\beta\), split
\(|a|=a_M+b_M\), where \(0\leq a_M\leq M\) and
\(\|b_M\|_{L^{Q/2}(B_R)}\to0\).  Choosing \(M=M(\beta)\) so that the
tail term is absorbed in the left-hand side of
\eqref{eq:local-caccioppoli-iteration}, and then letting \(L\to\infty\),
gives the finite exponent gain
\[
 \|U\|_{L^{\beta\chi}(B_r)}
 \leq C(\beta,r,R,a)\|U\|_{L^\beta(B_R)}.
\]
Only finitely many such gains are needed for any prescribed finite
exponent.  Therefore
\[
 U\in L^r_{\mathrm{loc}}\qquad\text{for every finite }r.
\]
This is the critical Brezis--Kato step; the same critical-to-finite
integrability threshold also appears in
\cite[Theorem~4.1]{Vassilev2006CarnotRegularity}.

Choose \(r\) so large that \(U^{4/(Q-2)}\) belongs locally to some
\(L^{p_0}\) with \(p_0>Q/2\).  Thus
\[
 a\in L^s_{\mathrm{loc}},\qquad
 s=\min\{p,p_0\}>\frac Q2.
\]
For this strictly supercritical exponent, set
\[
 r_s=\frac{2s}{s-1}<2\chi,\qquad
 \theta=\frac{Q}{2s}\in(0,1).
\]
Interpolation and Young's inequality give
\[
 \|f\|_{r_s}^2
 \leq\varepsilon\|f\|_{2\chi}^2
 +C\varepsilon^{-\theta/(1-\theta)}\|f\|_2^2.
\]
Applied to \(f=\eta Y\), this absorbs the coefficient term in
\eqref{eq:local-caccioppoli-iteration} and yields
\[
 \|U\|_{L^{\beta\chi}(B_r)}
 \leq
 \left[C\beta^\kappa
 \bigl(1+(R-r)^{-2}+\|a\|_{L^s(B_R)}^\kappa\bigr)
 \right]^{1/\beta}
 \|U\|_{L^\beta(B_R)}
\]
for some \(\kappa=\kappa(Q,s)\).  Iterating
\(\beta_j=2\chi^j\) on nested balls proves
\(U\in L^\infty_{\mathrm{loc}}\), since both
\(\sum_j\beta_j^{-1}\) and \(\sum_j j\beta_j^{-1}\) converge.

We apply the local theory of
\cite[Theorems~3.1 and~3.35]{CapognaDanielliGarofalo1993Harnack}; we record
the coefficient identification because the integrability threshold is
essential here.  On a bounded open set compactly contained in the Heisenberg
group, use their equation
\(\sum_jX_j^*A_j(x,u,Xu)=f(x,u,Xu)\) with principal exponent \(2\) and set
\[
 A_j(x,u,\xi)=\xi_j,
 \qquad f(x,u,\xi)=a(x)u.
\]
In their structural conditions, take \(c_1=1\),
\(f_2=|a|\), and
\[
 f_1=f_3=g_2=g_3=h_3=0.
\]
The growth inequality for \(f\) is then an equality, while the coercivity
condition reads
\(
 |\xi|^2\geq |\xi|^2-|a(x)|\,|u|^2
\).
All conditions involving the zero coefficient functions are automatic, and
the sole nontrivial source hypothesis is precisely
\(f_2\in L^s_{\mathrm{loc}}\) with \(s>Q/2\).  Moreover, the additive term
\(K(R)\) in their Harnack estimate vanishes because
\(f_3=g_3=h_3=0\).  The global finite-energy and critical-norm bounds imply
\(U\in S^{1,2}_{\mathrm{loc}}\), and the preceding iteration supplies the
local boundedness required by their H\"older theorem; its strengthened
condition on \(g_2,g_3\) is again automatic.  Consequently, whenever
\(B_{4r}\) is relatively compact,
\[
 \sup_{B_r}U\leq C\inf_{B_r}U.
\]
The locally H\"older representative therefore has an open and closed zero
set.  Since \(\Heis^{n+\ell}\) is connected and \(U\not\equiv0\), it follows
that \(U>0\) everywhere.  If \(W\) is smooth, positivity gives a positive
lower bound on each compact subset, so the nonlinearity is smooth there;
subelliptic interior estimates then bootstrap \(U\) to smoothness.
\end{proof}

\section{Normalization identities}
\label{app:normalization}

In the strict Hardy window
\[
 \ell\in\mathbb N,\qquad \ell>1,\qquad
 0<\lambda<(\ell-1)^2,
\]
the constants connecting the reduced and intrinsic formulations are as
follows:
\[
 Q_\ell=2(n+\ell)+2,\qquad
 q_\ell=\frac{2Q_\ell}{Q_\ell-2},
\]
\[
 \mathcal C_\ell U=\rho^{-(n+\ell)/2}U,\qquad
 \mu=(\ell-1)^2-\lambda,
\]
\[
 E_\ell(\mathcal C_\ell U,\mathcal C_\ell\Phi)
 =4\mathfrak b_{(\ell-1)^2}(U,\Phi),
\]
\[
 \mathsf H_\ell(\mathcal C_\ell U,\mathcal C_\ell\Phi)
 =4\int_{\CHyp^m}U\Phi\dd V_g,
\qquad
 \|\mathcal C_\ell U\|_{q_\ell}^{q_\ell}
 =4\|U\|_{q_\ell}^{q_\ell},
\]
\[
 S^H_{\lambda,\ell}
 =4^{1/(n+\ell+1)}S^{\mathrm{CH}}_{\mu,\ell},
\qquad
 d^{\mathrm{CH}}_{\mu,\ell}
 =\frac14d_{\lambda,\ell}.
\]
These identities use the metric of holomorphic sectional curvature \(-4\),
the volume normalization
in \eqref{eq:ch-volume-laplacian}, and the sub-Laplacian convention fixed in
Section~2.

\section*{Declarations}

\noindent\textbf{Acknowledgement.}
The author was supported by the National Natural Science Foundation of
China (Grant No.~12571127).

\medskip
\noindent\textbf{Conflict of interest.}
The author declares no conflict of interest.

\medskip
\noindent\textbf{Data and code availability.}\ Not applicable.

\begingroup
\sloppy
\bibliographystyle{amsplain}
\bibliography{project}

@article{JerisonLee1988Extremals,
  author   = {Jerison, David and Lee, John M.},
  title    = {Extremals for the {Sobolev} inequality on the {Heisenberg} group and the {CR} {Yamabe} problem},
  journal  = {Journal of the American Mathematical Society},
  volume   = {1},
  number   = {1},
  pages    = {1--13},
  year     = {1988},
  doi      = {10.1090/S0894-0347-1988-0924699-9},
  mrnumber = {924699}
}

@article{Vassilev2006CarnotRegularity,
  author  = {Vassilev, Dimiter},
  title   = {Existence of solutions and regularity near the characteristic boundary for sub-{Laplacian} equations on {Carnot} groups},
  journal = {Pacific Journal of Mathematics},
  volume  = {227},
  number  = {2},
  pages   = {361--397},
  year    = {2006},
  doi     = {10.2140/pjm.2006.227.361}
}

@article{CapognaDanielliGarofalo1993Harnack,
  author  = {Capogna, Luca and Danielli, Donatella and Garofalo, Nicola},
  title   = {An embedding theorem and the {Harnack} inequality for nonlinear
             subelliptic equations},
  journal = {Communications in Partial Differential Equations},
  volume  = {18},
  number  = {9--10},
  pages   = {1765--1794},
  year    = {1993},
  doi     = {10.1080/03605309308820992}
}

@article{IvanovVassilev2015Obata,
  author        = {Ivanov, Stefan and Vassilev, Dimiter},
  title         = {The {Lichnerowicz} and {Obata} first eigenvalue theorems and the {Obata} uniqueness result in the {Yamabe} problem on {CR} and quaternionic contact manifolds},
  journal       = {Nonlinear Analysis},
  volume        = {126},
  pages         = {262--323},
  year          = {2015},
  doi           = {10.1016/j.na.2015.06.024},
  eprint        = {1504.03259},
  archiveprefix = {arXiv},
  primaryclass  = {math.DG}
}

@article{LuYang2022SiegelHardySobolevMazya,
  author        = {Lu, Guozhen and Yang, Qiaohua},
  title         = {Sharp {Hardy--Sobolev--Maz'ya}, {Adams} and {Hardy--Adams} inequalities on the {Siegel} domains and complex hyperbolic spaces},
  journal       = {Advances in Mathematics},
  volume        = {405},
  pages         = {108512},
  year          = {2022},
  doi           = {10.1016/j.aim.2022.108512},
  eprint        = {2106.02103},
  archiveprefix = {arXiv},
  primaryclass  = {math.AP}
}

@article{LuYang2019PaneitzHyperbolic,
  author        = {Lu, Guozhen and Yang, Qiaohua},
  title         = {{Paneitz} operators on hyperbolic spaces and high order {Hardy--Sobolev--Maz'ya} inequalities on half spaces},
  journal       = {American Journal of Mathematics},
  volume        = {141},
  number        = {6},
  pages         = {1777--1816},
  year          = {2019},
  doi           = {10.1353/ajm.2019.0047},
  eprint        = {1703.08171},
  archiveprefix = {arXiv},
  primaryclass  = {math.AP}
}

@article{GrahamZworski2003Scattering,
  author        = {Graham, C. Robin and Zworski, Maciej},
  title         = {Scattering matrix in conformal geometry},
  journal       = {Inventiones Mathematicae},
  volume        = {152},
  number        = {1},
  pages         = {89--118},
  year          = {2003},
  doi           = {10.1007/s00222-002-0268-1},
  eprint        = {math/0109089},
  archiveprefix = {arXiv},
  primaryclass  = {math.DG}
}

@article{FrankGonzalezMonticelliTan2015CRExtension,
  author        = {Frank, Rupert L. and Gonz\'alez, Mar\'ia del Mar and Monticelli, Dario D. and Tan, Jinggang},
  title         = {An extension problem for the {CR} fractional {Laplacian}},
  journal       = {Advances in Mathematics},
  volume        = {270},
  pages         = {97--137},
  year          = {2015},
  doi           = {10.1016/j.aim.2014.09.026},
  eprint        = {1312.3381},
  archiveprefix = {arXiv},
  primaryclass  = {math.AP}
}

@article{BiroliSchindlerTintarev2003Symmetries,
  author  = {Biroli, Marco and Schindler, Ian and Tintarev, Cyril},
  title   = {Semilinear equations on {Hausdorff} spaces with symmetries},
  journal = {Rendiconti dell'Accademia Nazionale delle Scienze detta dei XL. Memorie di Matematica e Applicazioni (5)},
  volume  = {27},
  pages   = {175--189},
  year    = {2003}
}

@article{FrankLieb2012SharpHeisenberg,
  author  = {Frank, Rupert L. and Lieb, Elliott H.},
  title   = {Sharp constants in several inequalities on the {Heisenberg} group},
  journal = {Annals of Mathematics},
  volume  = {176},
  number  = {1},
  pages   = {349--381},
  year    = {2012},
  doi     = {10.4007/annals.2012.176.1.6}
}

@article{Benameur2008ProfileHeisenberg,
  author  = {Benameur, Jamel},
  title   = {Description du d{\'e}faut de compacit{\'e} de l'injection de {Sobolev} sur le groupe de {Heisenberg}},
  journal = {Bulletin of the Belgian Mathematical Society -- Simon Stevin},
  volume  = {15},
  number  = {4},
  pages   = {599--624},
  year    = {2008},
  doi     = {10.36045/bbms/1225893942}
}

@article{BianchiEgnell1991Sobolev,
  author  = {Bianchi, Gabriele and Egnell, Henrik},
  title   = {A note on the {Sobolev} inequality},
  journal = {Journal of Functional Analysis},
  volume  = {100},
  number  = {1},
  pages   = {18--24},
  year    = {1991},
  doi     = {10.1016/0022-1236(91)90099-Q}
}

@article{Loiudice2005ImprovedSobolev,
  author   = {Loiudice, Annunziata},
  title    = {Improved {Sobolev} inequalities on the {Heisenberg} group},
  journal  = {Nonlinear Analysis: Theory, Methods \& Applications},
  volume   = {62},
  number   = {5},
  pages    = {953--962},
  year     = {2005},
  doi      = {10.1016/j.na.2005.04.015},
  mrnumber = {2153462}
}

@article{ChenLuTangWang2026HeisenbergStability,
  author        = {Chen, Lu and Lu, Guozhen and Tang, Hanli and Wang, Bohan},
  title         = {Asymptotically sharp stability of {Sobolev} inequalities
                   on the {Heisenberg} group with dimension-dependent constants},
  journal       = {Journal de Math\'{e}matiques Pures et Appliqu\'{e}es},
  volume        = {206},
  pages         = {103832},
  year          = {2026},
  doi           = {10.1016/j.matpur.2025.103832},
  eprint        = {2507.12725},
  archiveprefix = {arXiv},
  primaryclass  = {math.AP}
}

@misc{ChenFanLiao2025GlobalCompactness,
  author        = {Chen, Hua and Fan, Yun-lu and Liao, Xin},
  title         = {Sharp stability of global compactness on the {Heisenberg}
                   group},
  year          = {2025},
  eprint        = {2506.11404},
  archiveprefix = {arXiv},
  primaryclass  = {math.AP},
  doi           = {10.48550/arXiv.2506.11404}
}

@article{DolbeaultEstebanFigalliFrankLoss2025Sharp,
  author  = {Dolbeault, Jean and Esteban, Maria J. and Figalli, Alessio and
             Frank, Rupert L. and Loss, Michael},
  title   = {Sharp stability for {Sobolev} and log-{Sobolev} inequalities,
             with optimal dimensional dependence},
  journal = {Cambridge Journal of Mathematics},
  volume  = {13},
  number  = {2},
  pages   = {359--430},
  year    = {2025}
}

@article{FrankSeiringer2008Hardy,
  author  = {Frank, Rupert L. and Seiringer, Robert},
  title   = {Non-linear ground state representations and sharp {Hardy}
             inequalities},
  journal = {Journal of Functional Analysis},
  volume  = {255},
  number  = {12},
  pages   = {3407--3430},
  year    = {2008},
  doi     = {10.1016/j.jfa.2008.05.015}
}

@article{Anker1990Multipliers,
  author  = {Anker, Jean-Philippe},
  title   = {{$L_p$} {Fourier} multipliers on {Riemannian} symmetric spaces
             of the noncompact type},
  journal = {Annals of Mathematics},
  series  = {2},
  volume  = {132},
  number  = {3},
  pages   = {597--628},
  year    = {1990}
}

@article{AnkerDamekYacoub1996HarmonicAN,
  author  = {Anker, Jean-Philippe and Damek, Ewa and Yacoub, Chokri},
  title   = {Spherical analysis on harmonic {$AN$} groups},
  journal = {Annali della Scuola Normale Superiore di Pisa,
             Classe di Scienze},
  series  = {4},
  volume  = {23},
  number  = {4},
  pages   = {643--679},
  year    = {1996}
}

@article{AnkerJi1999HeatGreen,
  author  = {Anker, Jean-Philippe and Ji, Lizhen},
  title   = {Heat kernel and {Green} function estimates on noncompact
             symmetric spaces},
  journal = {Geometric and Functional Analysis},
  volume  = {9},
  pages   = {1035--1091},
  year    = {1999},
  doi     = {10.1007/s000390050107}
}

@article{AnkerPierfelice2009NLS,
  author  = {Anker, Jean-Philippe and Pierfelice, Vittoria},
  title   = {Nonlinear {Schr\"odinger} equation on real hyperbolic spaces},
  journal = {Annales de l'Institut Henri Poincar\'e,
             Analyse Non Lin\'eaire},
  volume  = {26},
  number  = {5},
  pages   = {1853--1869},
  year    = {2009},
  doi     = {10.1016/j.anihpc.2009.01.009}
}

@article{AnkerPierfeliceVallarino2015Wave,
  author  = {Anker, Jean-Philippe and Pierfelice, Vittoria and
             Vallarino, Maria},
  title   = {The wave equation on {Damek--Ricci} spaces},
  journal = {Annali di Matematica Pura ed Applicata},
  volume  = {194},
  number  = {3},
  pages   = {731--758},
  year    = {2015},
  doi     = {10.1007/s10231-013-0395-x}
}

@article{Strichartz1983CompleteLaplacian,
  author  = {Strichartz, Robert S.},
  title   = {Analysis of the {Laplacian} on the complete {Riemannian} manifold},
  journal = {Journal of Functional Analysis},
  volume  = {52},
  number  = {1},
  pages   = {48--79},
  year    = {1983},
  doi     = {10.1016/0022-1236(83)90090-3}
}

@article{LiLuYang2022HyperbolicBN,
  author        = {Li, Jungang and Lu, Guozhen and Yang, Qiaohua},
  title         = {Higher order {Brezis--Nirenberg} problem on hyperbolic spaces: Existence, nonexistence and symmetry of solutions},
  journal       = {Advances in Mathematics},
  volume        = {399},
  pages         = {108259},
  year          = {2022},
  doi           = {10.1016/j.aim.2022.108259},
  eprint        = {2107.03967},
  archiveprefix = {arXiv},
  primaryclass  = {math.AP}
}

@article{BrezisNirenberg1983Critical,
  author  = {Brezis, Ha{\"\i}m and Nirenberg, Louis},
  title   = {Positive solutions of nonlinear elliptic equations involving critical {Sobolev} exponents},
  journal = {Communications on Pure and Applied Mathematics},
  volume  = {36},
  number  = {4},
  pages   = {437--477},
  year    = {1983},
  doi     = {10.1002/cpa.3160360405}
}

@article{BenciCerami1990Positive,
  author  = {Benci, Vieri and Cerami, Giovanna},
  title   = {Existence of positive solutions of the equation
             \({-\Delta u+a(x)u=u^{(N+2)/(N-2)}}\) in \(\mathbb R^N\)},
  journal = {Journal of Functional Analysis},
  volume  = {88},
  number  = {1},
  pages   = {90--117},
  year    = {1990},
  doi     = {10.1016/0022-1236(90)90120-A}
}

@article{Passaseo1996PositiveSolutions,
  author  = {Passaseo, Donato},
  title   = {Some sufficient conditions for the existence of positive
             solutions to the equation
             \({-\Delta u+a(x)u=u^{2^*-1}}\) in bounded domains},
  journal = {Annales de l'Institut Henri Poincar{\'e} C,
             Analyse non lin{\'e}aire},
  volume  = {13},
  number  = {2},
  pages   = {185--227},
  year    = {1996},
  doi     = {10.1016/S0294-1449(16)30102-0}
}

@article{Struwe1984Global,
  author  = {Struwe, Michael},
  title   = {A global compactness result for elliptic boundary value problems
             involving limiting nonlinearities},
  journal = {Mathematische Zeitschrift},
  volume  = {187},
  number  = {4},
  pages   = {511--517},
  year    = {1984},
  doi     = {10.1007/BF01174186}
}

@article{Friedrichs1934Spektraltheorie,
  author  = {Friedrichs, Kurt},
  title   = {Spektraltheorie halbbeschr{\"a}nkter Operatoren und Anwendung
             auf die Spektralzerlegung von Differentialoperatoren. I},
  journal = {Mathematische Annalen},
  volume  = {109},
  pages   = {465--487},
  year    = {1934},
  doi     = {10.1007/BF01449150}
}

@article{ManciniSandeep2008Semilinear,
  author  = {Mancini, Gianni and Sandeep, Kunnath},
  title   = {On a semilinear elliptic equation in \(\mathbb H^n\)},
  journal = {Annali della Scuola Normale Superiore di Pisa,
             Classe di Scienze},
  series  = {5},
  volume  = {7},
  number  = {4},
  pages   = {635--671},
  year    = {2008},
  doi     = {10.2422/2036-2145.2008.4.03}
}

@article{PalatucciPiccininiTemperini2025Struwe,
  author        = {Palatucci, Giampiero and Piccinini, Mirco and
                   Temperini, Letizia},
  title         = {{Struwe}'s global compactness and energy approximation of
                   the critical {Sobolev} embedding in the {Heisenberg} group},
  journal       = {Advances in Calculus of Variations},
  volume        = {18},
  number        = {3},
  pages         = {731--754},
  year          = {2025},
  doi           = {10.1515/acv-2024-0044},
  eprint        = {2308.01153},
  archiveprefix = {arXiv},
  primaryclass  = {math.AP}
}

@article{BessonCourtoisGallot1995Entropies,
  author  = {Besson, G{\'e}rard and Courtois, Gilles and Gallot, Sylvestre},
  title   = {Entropies et rigidit{\'e}s des espaces localement sym{\'e}triques
             de courbure strictement n{\'e}gative},
  journal = {Geometric and Functional Analysis},
  volume  = {5},
  number  = {5},
  pages   = {731--799},
  year    = {1995},
  doi     = {10.1007/BF01897050}
}

@book{Willem1996Minimax,
  author    = {Willem, Michel},
  title     = {Minimax Theorems},
  series    = {Progress in Nonlinear Differential Equations and
               their Applications},
  volume    = {24},
  publisher = {Birkh{\"a}user},
  address   = {Boston},
  year      = {1996},
  doi       = {10.1007/978-1-4612-4146-1}
}

@article{MalchiodiUguzzoni2002Webster,
  author  = {Malchiodi, Andrea and Uguzzoni, Francesco},
  title   = {A perturbation result for the {Webster} scalar curvature
             problem on the {CR} sphere},
  journal = {Journal de Math\'{e}matiques Pures et Appliqu\'{e}es},
  volume  = {81},
  number  = {10},
  pages   = {983--997},
  year    = {2002},
  doi     = {10.1016/S0021-7824(01)01249-1}
}

@article{QiangTangZhang2026Nondegeneracy,
  author  = {Qiang, Jiechen and Tang, Zhongwei and Zhang, Yichen},
  title   = {On the non-degeneracy and existence of sign-changing solutions
             to elliptic problem on the {Heisenberg} group},
  journal = {Nonlinear Analysis},
  volume  = {264},
  pages   = {113999},
  year    = {2026},
  doi     = {10.1016/j.na.2025.113999}
}

@article{Folland1972TangentialCR,
  author   = {Folland, Gerald B.},
  title    = {The tangential {Cauchy--Riemann} complex on spheres},
  journal  = {Transactions of the American Mathematical Society},
  volume   = {171},
  pages    = {83--133},
  year     = {1972},
  doi      = {10.1090/S0002-9947-1972-0309156-X},
  mrnumber = {0309156}
}

@article{TangZhangZhang2024Minimizer,
  author  = {Tang, Zhongwei and Zhang, Bingwei and Zhang, Yichen},
  title   = {Existence of a minimizer for the {Bianchi--Egnell}
             inequality on the {Heisenberg} group},
  journal = {Journal of Geometric Analysis},
  volume  = {34},
  number  = {5},
  pages   = {148},
  year    = {2024},
  doi     = {10.1007/s12220-024-01578-w}
}

@article{Savini2023NaturalMaps,
  author  = {Savini, Alessio},
  title   = {Natural maps for measurable cocycles of compact hyperbolic
             manifolds},
  journal = {Journal of the Institute of Mathematics of Jussieu},
  volume  = {22},
  number  = {1},
  pages   = {421--448},
  year    = {2023},
  doi     = {10.1017/S1474748021000475}
}

@misc{LiWang2026Global,
  author        = {Li, Jungang and Wang, Zhiwei},
  title         = {Two-channel global compactness and existence for critical
                   {GJMS} equations on hyperbolic space},
  year          = {2024},
  eprint        = {2411.14719v3},
  archiveprefix = {arXiv},
  primaryclass  = {math.AP},
  note          = {arXiv preprint, v3, revised 8 August 2026},
  doi           = {10.48550/arXiv.2411.14719}
}

@article{Koenig2023BianchiEgnell,
  author        = {K{\"o}nig, Tobias},
  title         = {On the sharp constant in the {Bianchi--Egnell} stability
                   inequality},
  journal       = {Bulletin of the London Mathematical Society},
  volume        = {55},
  number        = {4},
  pages         = {2070--2075},
  year          = {2023},
  doi           = {10.1112/blms.12837},
  eprint        = {2210.08482},
  archiveprefix = {arXiv},
  primaryclass  = {math.AP}
}

@article{Koenig2025SobolevMinimizer,
  author  = {K{\"o}nig, Tobias},
  title   = {Stability for the {Sobolev} inequality: Existence of a
             minimizer},
  journal = {Journal of the European Mathematical Society},
  year    = {2025},
  note    = {Published online first},
  doi     = {10.4171/JEMS/1582}
}

@article{FigalliGlaudo2020SharpStability,
  author        = {Figalli, Alessio and Glaudo, Federico},
  title         = {On the sharp stability of critical points of the {Sobolev}
                   inequality},
  journal       = {Archive for Rational Mechanics and Analysis},
  volume        = {237},
  number        = {1},
  pages         = {201--258},
  year          = {2020},
  doi           = {10.1007/s00205-020-01506-6},
  eprint        = {1910.11351},
  archiveprefix = {arXiv},
  primaryclass  = {math.AP}
}

@article{Aubin1976Sobolev,
  author  = {Aubin, Thierry},
  title   = {Probl{\`e}mes isop{\'e}rim{\'e}triques et espaces de {Sobolev}},
  journal = {Journal of Differential Geometry},
  volume  = {11},
  number  = {4},
  pages   = {573--598},
  year    = {1976},
  doi     = {10.4310/jdg/1214433725}
}

@article{Talenti1976BestConstant,
  author  = {Talenti, Giorgio},
  title   = {Best constant in {Sobolev} inequality},
  journal = {Annali di Matematica Pura ed Applicata},
  series  = {4},
  volume  = {110},
  pages   = {353--372},
  year    = {1976},
  doi     = {10.1007/BF02418013}
}

@article{ChenFrankWeth2013Remainder,
  author  = {Chen, Shibing and Frank, Rupert L. and Weth, Tobias},
  title   = {Remainder terms in the fractional {Sobolev} inequality},
  journal = {Indiana University Mathematics Journal},
  volume  = {62},
  number  = {4},
  pages   = {1381--1397},
  year    = {2013},
  doi     = {10.1512/iumj.2013.62.5065},
  eprint        = {1205.5666},
  archiveprefix = {arXiv},
  primaryclass  = {math.AP}
}

@article{EdmundsFortunatoJannelli1990Biharmonic,
  author  = {Edmunds, David E. and Fortunato, Donato and Jannelli, Enrico},
  title   = {Critical exponents, critical dimensions and the biharmonic operator},
  journal = {Archive for Rational Mechanics and Analysis},
  volume  = {112},
  number  = {3},
  pages   = {269--289},
  year    = {1990},
  doi     = {10.1007/BF00381236}
}

@article{Gazzola1998Polyharmonic,
  author  = {Gazzola, Filippo},
  title   = {Critical growth problems for polyharmonic operators},
  journal = {Proceedings of the Royal Society of Edinburgh Section A: Mathematics},
  volume  = {128},
  number  = {2},
  pages   = {251--263},
  year    = {1998},
  doi     = {10.1017/S0308210500012774}
}

@article{GazzolaGrunauSquassina2003Critical,
  author  = {Gazzola, Filippo and Grunau, Hans-Christoph and Squassina, Marco},
  title   = {Existence and nonexistence results for critical growth biharmonic elliptic equations},
  journal = {Calculus of Variations and Partial Differential Equations},
  volume  = {18},
  number  = {2},
  pages   = {117--143},
  year    = {2003},
  doi     = {10.1007/s00526-002-0182-9}
}

@article{BhaktaSandeep2012Poincare,
  author        = {Bhakta, Mousomi and Sandeep, Kunnath},
  title         = {Poincar{\'e}--{Sobolev} equations in the hyperbolic space},
  journal       = {Calculus of Variations and Partial Differential Equations},
  volume        = {44},
  number        = {1--2},
  pages         = {247--269},
  year          = {2012},
  doi           = {10.1007/s00526-011-0433-8},
  eprint        = {1103.4779},
  archiveprefix = {arXiv},
  primaryclass  = {math.AP}
}

@article{BhaktaGangulyKarmakarMazumdar2025Stability,
  author        = {Bhakta, Mousomi and Ganguly, Debdip and Karmakar, Debabrata and Mazumdar, Saikat},
  title         = {Sharp quantitative stability of {Poincar\'e--Sobolev} inequality in the hyperbolic space and applications to fast diffusion flows},
  journal       = {Calculus of Variations and Partial Differential Equations},
  volume        = {64},
  number        = {1},
  pages         = {23},
  year          = {2025},
  doi           = {10.1007/s00526-024-02878-3},
  eprint        = {2207.11024},
  archiveprefix = {arXiv},
  primaryclass  = {math.AP}
}

@article{BhaktaGangulyKarmakarMazumdar2025Struwe,
  author        = {Bhakta, Mousomi and Ganguly, Debdip and Karmakar, Debabrata and Mazumdar, Saikat},
  title         = {Sharp quantitative stability of {Struwe}'s decomposition of the {Poincar\'e--Sobolev} inequalities on the hyperbolic space: Part {I}},
  journal       = {Advances in Mathematics},
  volume        = {479},
  pages         = {110447},
  year          = {2025},
  doi           = {10.1016/j.aim.2025.110447},
  eprint        = {2211.14618},
  archiveprefix = {arXiv},
  primaryclass  = {math.AP}
}

@article{FollandStein1974Heisenberg,
  author  = {Folland, Gerald B. and Stein, Elias M.},
  title   = {Estimates for the \(\bar\partial_b\) complex and analysis on the {Heisenberg} group},
  journal = {Communications on Pure and Applied Mathematics},
  volume  = {27},
  number  = {4},
  pages   = {429--522},
  year    = {1974},
  doi     = {10.1002/cpa.3160270403}
}

@article{Geller1980Kohn,
  author  = {Geller, Daryl},
  title   = {The {Laplacian} and the {Kohn Laplacian} for the sphere},
  journal = {Journal of Differential Geometry},
  volume  = {15},
  number  = {3},
  pages   = {417--435},
  year    = {1980},
  doi     = {10.4310/jdg/1214435651}
}

@article{FlynnLuYang2025CRTrace,
  author        = {Flynn, Joshua and Lu, Guozhen and Yang, Qiaohua},
  title         = {Conformally covariant boundary operators and sharp higher order {CR} {Sobolev} trace inequalities on the {Siegel} domain and complex ball},
  journal       = {Journal f{\"u}r die reine und angewandte Mathematik},
  volume        = {824},
  pages         = {39--87},
  year          = {2025},
  doi           = {10.1515/crelle-2025-0017},
  eprint        = {2311.09956},
  archiveprefix = {arXiv},
  primaryclass  = {math.AP}
}
\endgroup

\end{document}